\documentclass[oneside, 12pt]{amsart}

\usepackage{amsmath, amssymb, amsthm, amsfonts}
\usepackage{newtxtext,newtxmath}
\usepackage[cal=euler,scr=rsfs]{mathalfa}
\usepackage[all]{xy}
\usepackage[a4paper, margin=1.1in]{geometry}
\usepackage{tikz}
\usetikzlibrary{arrows.meta}
\usetikzlibrary{decorations.markings}
\usepackage[backref=page]{hyperref}
\hypersetup{pdfborder={0 0 0}, colorlinks=true, urlcolor=blue, linkcolor=blue, citecolor=red}
\numberwithin{equation}{section}

\SelectTips{eu}{12}

\tikzset{myarrow/.style={postaction={decorate},decoration={markings,mark=at position #1 with {\arrow{latex}},}}}

\theoremstyle{plain}
\newtheorem{theorem}{Theorem}[section]
\newtheorem{proposition}[theorem]{Proposition}
\newtheorem{lemma}[theorem]{Lemma}
\newtheorem{corollary}[theorem]{Corollary}

\newtheorem{problem}[theorem]{Problem}

\theoremstyle{definition}
\newtheorem{definition}[theorem]{Definition}
\newtheorem{example}[theorem]{Example}

\theoremstyle{remark}
\newtheorem{remark}[theorem]{Remark}

\newcommand{\Z}{\mathbb{Z}}

\newcommand{\cSet}{\mathsf{cSet}}
\newcommand{\SymcSet}{\mathsf{SymcSet}}
\newcommand{\QSymcSet}{\mathsf{QSymcSet}}
\newcommand{\sd}{\operatorname{sd}}
\newcommand{\ex}{\operatorname{Ex}}

\title{Discrete homotopy groups of cubical sets}

\author{Daisuke Kishimoto}
\address{Faculty of Mathematics, Kyushu University, Fukuoka 819-0395, Japan}
\email{kishimoto@math.kyushu-u.ac.jp}

\author{Yichen Tong}
\address{Institute for Theoretical Sciences, Westlake Institute of Advanced Study, Westlake University, Hangzhou 310030, China}
\email{tongyichen@westlake.edu.cn}

\date{\today}



\begin{document}

	\begin{abstract}
		We extend the notion of discrete homotopy groups of graphs to arbitrary cubical sets, and show that the discrete homotopy groups of quasisymmetric cubical sets are naturally isomorphic to the homotopy groups of their geometric realizations. Here, quasisymmetric cubical sets are cubical sets equipped with coordinate permutation symmetries that are compatible with faces and degeneracies, but not necessarily with connections. We give a purely combinatorial construction of the left adjoint of the forgetful functor from the category of quasisymmetric cubical sets to the category of cubical sets, and prove that the unit of this adjunction is a natural weak equivalence. As a consequence, we obtain a purely combinatorial description of the homotopy groups of the geometric realizations of arbitrary cubical sets. As an application, we establish the discrete Hurewicz map from the discrete homotopy groups of a cubical set to the homology groups, and prove a Hurewicz theorem with respect to the discrete Hurewicz map whenever the cubical set is quasisymmetric. Applied to the 1-nerve of a directed graph, it also yields a Hurewicz theorem relating the homotopy groups of Li, Wu, Yau, and Zhang to singular cubical homology.
	\end{abstract}
	\maketitle
	\section{Introduction}\label{Introduction}
	
	Cubical sets, like simplicial sets, are combinatorial models of spaces in algebraic topology. They are defined as presheaves on the category $\Box$ of combinatorial cubes. In recent years, cubical sets have attracted renewed attention in connection with combinatorics \cite{CK2,EIXYZ}. Homotopy groups are among the most fundamental invariants of cubical sets, but their direct definition requires the cubical set to be a Kan complex. As in the simplicial setting, the Kan condition plays an essential role in defining the homotopy groups of cubical sets: horn fillers are needed to define the concatenation of homotopies and the group operation on homotopy groups. See \cite{CK1}.
	
	Many important cubical sets are not Kan complexes. Their homotopy groups must therefore be defined via a fibrant replacement or geometric realization. Unlike simplicial sets, which admit Kan's $\mathrm{Ex}^\infty$ construction, cubical sets lack such a standard fibrant replacement. Maps and homotopies in a fibrant replacement or in the geometric realization are generally not described directly in terms of the original combinatorial data. This difficulty motivates us to define homotopy groups directly in terms of the original cubical set, avoiding both fibrant replacement and geometric realization. Such a definition is particularly useful for applications involving combinatorics.

	
	
	Discrete homotopy theory provides such a construction for graphs. Discrete homotopy theory, introduced in \cite{BBdLL,BKLW}, is a combinatorial homotopy theory of graphs. It builds on earlier work of Atkin \cite{At1,At2} concerning algebraic representations of simplicial complexes, and later reformulated in \cite{BL} as a homotopy theory of graphs and simplicial complexes. This theory has found numerous applications, including in matroid theory, hyperplane arrangements, combinatorial time series analysis and topological data analysis. The fundamental invariants in discrete homotopy theory of graphs are the $A$-groups, named after Atkin, which we refer to as the \emph{discrete homotopy groups}.

	Cubical sets arise naturally in the study of discrete homotopy theory of graphs. Babson, Barcelo, de Longueville, and Laubenbacher \cite{BBdLL} associated a space $X_G$ to a graph $G$, and conjectured that the homotopy groups of $X_G$ are isomorphic to the discrete homotopy groups of $G$. Carranza and Kapulkin \cite{CK2} later constructed a cubical set $\mathrm{N}_1G$, called the \emph{$1$-nerve} of $G$, whose geometric realization coincides with the space $X_G$. They showed that the homotopy groups of the geometric realization of $\mathrm{N}_1G$ are naturally isomorphic to the discrete homotopy groups of $G$, thereby proving the conjecture. Their proof relies on applying an appropriate fibrant replacement to the cubical set $\mathrm{N}_1G$.

	In this paper, we introduce the discrete homotopy groups $\pi_n^\delta(X,x_0)$ of an arbitrary pointed cubical set $(X,x_0)$. The definition is purely combinatorial and does not require $X$ to be a Kan complex. In particular, representatives, homotopies, and multiplication are specified by finitely many cubes of $X$ and their face relations; no choice of horn fillers is part of the definition. We prove that these groups are functorial and invariant under pointed cubical homotopy equivalences. As usual, $\pi_n^\delta(X,x_0)$ is a pointed set for $n=0$, a group for $n=1$, and an abelian group for $n\ge2$.

	The next question is whether these groups recover the homotopy groups of the geometric realization. For this comparison, coordinate symmetries play an essential role in our constructions. Every geometric $n$-cube is the product of $n$ copies of the interval $[0,1]$, and thus carries an intrinsic symmetry arising from permutations of its coordinates. This symmetry is implicitly and frequently used in classical homotopy theory, for example in the Eckmann-Hilton argument. However, the category $\Box$ itself is not endowed with this symmetry. We define a \emph{quasisymmetric cubical set} by adjoining coordinate permutation symmetries that are required to be compatible with faces and degeneracies, but not necessarily with connections. Grandis and Mauri \cite{GM1} introduced a category obtained from $\Box$ by adjoining the coordinate permutation symmetries that are compatible with faces, degeneracies, and connections. Later, Isaacson \cite{I} introduced the term \emph{symmetric cubical set} for a presheaf on this category and studied the category of symmetric cubical sets from the perspective of model category theory. In particular, he showed that symmetric cubical sets also form a model for the homotopy theory of spaces. By definition, the class of quasisymmetric cubical sets strictly contains that of symmetric cubical sets.

	Let $\cSet$ and $\QSymcSet$ denote the categories of cubical sets and quasisymmetric cubical sets respectively. Then there is a canonical forgetful functor
	\[
	\mathcal{U}\colon\QSymcSet\to\cSet.
	\]
	Throughout this paper, when a cubical set $X$ is assumed to be quasisymmetric, $X$ denotes the underlying cubical set $\mathcal{U}Y$ of an object $Y\in\QSymcSet$, with the symmetry operators supplied by $Y$. The forgetful functors remain explicit in expressions such as $\mathcal{U}\mathcal{S}$. For a cubical set $X$, let $|X|$ denote its geometric realization. Our first main result is the following.
	
	\begin{theorem}
		\label{main 1}
		Let $(X,x_0)$ be a pointed quasisymmetric cubical set. There is a natural isomorphism
		\[
		\pi_*^\delta(X,x_0)\cong\pi_*(|X|,x_0).
		\]
	\end{theorem}
	
	Thus the homotopy groups of the realization of a quasisymmetric cubical set admit a purely combinatorial description, even when it is not Kan.

	We next relate arbitrary cubical sets to quasisymmetric cubical sets. Let $\SymcSet$ denote the category of symmetric cubical sets. In \cite{I}, Isaacson showed that the forgetful functor $\mathcal{U}\colon\SymcSet\to\cSet$ admits a left adjoint $\mathcal{F}\colon\cSet\to\SymcSet$ such that the unit $1_\cSet\Rightarrow\mathcal{UF}$ is a natural weak equivalence, where a weak equivalence is a cubical map that induces isomorphisms on homotopy groups of geometric realizations with respect to all basepoints. The left adjoint $\mathcal{F}$ is defined via a left Kan extension. Let $\QSymcSet$ denote the category of quasisymmetric cubical sets. In this paper, we combinatorially and explicitly construct a left adjoint $\mathcal{S}\colon\cSet\to\QSymcSet$ of the forgetful functor $\mathcal{U}\colon\QSymcSet\to\cSet$, and prove the following second main result. 
	\begin{theorem}
		\label{main 2}
		One can combinatorially construct a left adjoint $\mathcal{S}\colon\cSet\to\QSymcSet$ of the forgetful functor $\mathcal{U}\colon\QSymcSet\to\cSet$ such that the unit
		\[
		\mu\colon 1_{\cSet}\Rightarrow\mathcal{U}\mathcal{S}
		\]
		is a natural weak equivalence.
	\end{theorem}
	
	Combining Theorems \ref{main 1} and \ref{main 2}, we obtain a natural isomorphism
	\[
	\pi_n^\delta(\mathcal{U}\mathcal{S}X,\mu_X(x_0))
	\cong\pi_n(|X|,x_0)
	\]
	for every pointed cubical set $(X,x_0)$. This gives a combinatorial description of the homotopy groups of the realization of an arbitrary cubical set by first adjoining quasisymmetries and then taking discrete homotopy groups. Combining Theorem \ref{main 2} with the constructions of Section \ref{Fibrant replacement}, we obtain a functorial fibrant replacement which sends a cubical set to a quasisymmetric Kan complex (see Corollary \ref{fib replacement}). Consequently, for any pointed Kan complex $(X,x_0)$ which is not necessarily quasisymmetric, this fibrant replacement, together with the homotopy invariance of the discrete homotopy groups, yields a natural isomorphism
	\[
	\pi_*(X,x_0)\cong \pi_*^\delta(X,x_0).
	\]
	Note that the discrete homotopy groups on the right-hand side are defined purely combinatorially, whereas the homotopy groups on the left-hand side are defined homotopically.

	Finally, we apply Theorem \ref{main 1} to establish the Hurewicz theorem for the discrete homotopy groups of quasisymmetric cubical sets. The precise definition of the discrete Hurewicz map will be given in Section \ref{Discrete homotopy groups of quasisymmetric cubical sets}. A cubical set $X$ is called \emph{$n$-connected} if $\pi_*^\delta(X,x_0)=0$ for all $*\le n$ and for every basepoint $x_0$.
	\begin{theorem}
		\label{main 3}
		Let $(X,x_0)$ be a pointed quasisymmetric cubical set. If $X$ is $(n-1)$-connected, then the map
		\[
		\pi_n^\delta(X,x_0)\to H_n(X),\quad
		[x]\mapsto\left[\sum_{i_1,\ldots,i_n}
		\mathrm{sgn}_{i_1,\ldots,i_n}(x)x(i_1,\ldots,i_n)\right]
		\]
		is the abelianization for $n=1$ and an isomorphism for $n\ge2$.
	\end{theorem}
	
	We also apply the theory to directed graphs. Grigor'yan, Lin, Muranov, and Yau introduced the path homology $\mathrm{PH}_*(G)$ in \cite{GLMY1} and the fundamental group $\pi_1(G,v)$ of a pointed directed graph $(G,v)$ in \cite{GLMY2}. They proved that the abelianization of this fundamental group is the first path homology group. The authors subsequently generalized this one-dimensional Hurewicz theorem in \cite{KT}. Li, Wu, Yau, and Zhang \cite{LWYZ} introduced higher homotopy groups $\bar{\pi}_n(G,v)$ satisfying
	\[
	\bar{\pi}_1(G,v)\cong\pi_1(G,v).
	\]
	Theriault, Wu, Yau, and Zhang \cite{TWYZ} constructed directed graphs with nontrivial higher homotopy groups. They used a homomorphism $\tilde{h}_*$ to path homology which factors as
	\begin{equation}
		\label{TWYZ Hurewicz map}
		\bar{\pi}_n(G,v)\xrightarrow{h_*}\mathrm{QH}_n(G)
		\xrightarrow{\phi_G}\mathrm{PH}_n(G).
	\end{equation}
	Here $\mathrm{QH}_n(G)$ denotes the singular cubical homology group of $G$, developed in \cite{GM2}.
	
	We define the $1$-nerve $\mathrm{N}_1G$ of a directed graph $G$, which is a symmetric cubical set, and identify $\bar{\pi}_n(G,v)$ with $\pi_n^\delta(\mathrm{N}_1G,v)$. Under this identification and the identification of $\mathrm{QH}_n(G)$ with $H_n(\mathrm{N}_1G)$, the map $h_*$ in \eqref{TWYZ Hurewicz map} agrees with the discrete Hurewicz map. We say that $G$ is $n$-connected if $\mathrm{N}_1G$ is $n$-connected. Applying Theorem \ref{main 3} gives the following result.
	
	\begin{theorem}
		\label{main 4}
		Let $(G,v)$ be a pointed directed graph. If $G$ is $(n-1)$-connected, then the homomorphism
		\[
		h_*\colon\bar{\pi}_n(G,v)\to\mathrm{QH}_n(G)
		\]
		is the abelianization for $n=1$ and an isomorphism for $n\ge2$.
	\end{theorem}
	
	Since the first path homology group and the first singular cubical homology group of a directed graph are isomorphic, this theorem generalizes the Hurewicz theorem in \cite{GLMY2}. We also obtain an alternative proof of the Hurewicz theorem for graphs proved in \cite{CK2}.
	
	The organization of this paper is as follows. In Section \ref{Preliminaries}, we recall the basic notions of cubical sets, their products and geometric realizations, cubical homology, and the homotopy groups of Kan complexes. In Section \ref{Discrete homotopy groups of cubical sets}, we define grids, grid homotopy, and the discrete homotopy groups of arbitrary cubical sets, and establish their basic properties. In Section \ref{Quasisymmetric cubical sets}, we discuss symmetric and quasisymmetric cubical sets, together with their signed versions, and prove Theorem \ref{main 2}. In Section \ref{Fibrant replacement}, we construct an explicit fibrant replacement of a quasisymmetric cubical set. In Section \ref{Discrete homotopy groups of quasisymmetric cubical sets}, we prove Theorems \ref{main 1} and \ref{main 3}. Finally, in Section \ref{Discrete homotopy groups of directed graphs}, we compare the discrete homotopy groups of a directed graph with those of its $1$-nerve and prove Theorem \ref{main 4}.
	
	\section*{Acknowledgment}
	The authors would like to thank Kensuke Arakawa for bringing relevant references to their attention, and Xing Gu for inspiring discussions. The first author was partially supported by JSPS KAKENHI Grant Number JP22K03284.
	
	\section{Preliminaries}\label{Preliminaries}
	
	In this section, we recall the basic notions of cubical sets and Kan complexes, together with several key constructions, including products and geometric realization. We also review the homology of cubical sets and the homotopy groups of Kan complexes.

	\subsection{Cubical sets}
	
	We begin by defining the cubical category $\Box$. Let $[1]$ denote the poset $\{0<1\}$. The \emph{cubical category} $\Box$ has objects $[1]^n$ for $n\ge 0$, and morphisms generated under composition by the following three classes of maps:
	
	\begin{itemize}
		\item (faces) $\partial_{i,\alpha}\colon[1]^{n-1}\to[1]^n$ for $i=1,\ldots,n$ and $\alpha=0,1$ given by
		\[
		\partial_{i,\alpha}(x_1,\ldots,x_{n-1})=(x_1,\ldots,x_{i-1},\alpha,x_i,\ldots,x_{n-1}).
		\]
		
		\item (degeneracies) $\sigma_i\colon[1]^n\to[1]^{n-1}$ for $i=1,\ldots,n$ given by
		\[
		\sigma_i(x_1,\ldots,x_n)=(x_1,\ldots,x_{i-1},x_{i+1},\ldots,x_n).
		\]
		
		\item (connections) $\gamma_{i,\alpha}\colon[1]^{n}\to[1]^{n-1}$ for $i=1,\ldots,n-1$ and $\alpha=0,1$ given by
		\begin{align*}
			\gamma_{i,0}(x_1,\ldots,x_{n})=(x_1,\ldots,x_{i-1},\max\{x_i,x_{i+1}\},x_{i+2},\ldots,x_n)\\
			\gamma_{i,1}(x_1,\ldots,x_{n})=(x_1,\ldots,x_{i-1},\min\{x_i,x_{i+1}\},x_{i+2},\ldots,x_n).
		\end{align*}
	\end{itemize}
	
	\noindent It is straightforward to verify that these maps satisfy the following \emph{cubical identities}:
	\begin{alignat*}{2}
		\partial_{j,\beta}\partial_{i,\alpha}&=\partial_{i+1,\alpha}\partial_{j,\beta}\quad(j\le i)&\sigma_i\sigma_j&=\sigma_j\sigma_{i+1}\quad(j\le i)\\
		\sigma_j\partial_{i,\alpha}&=
		\begin{cases}
			\partial_{i-1,\alpha}\sigma_j&(j<i)\\
			1&(j=i)\\
			\partial_{i,\alpha}\sigma_{j-1}&(j>i)
		\end{cases}&
		\gamma_{j,\beta}\gamma_{i,\alpha}&=
		\begin{cases}
			\gamma_{i,\alpha}\gamma_{j+1,\beta}&(j>i)\\
			\gamma_{i,\alpha}\gamma_{i+1,\alpha}&(j=i,\,\beta=\alpha)
		\end{cases}\\
		\gamma_{j,\beta}\partial_{i,\alpha}&=
		\begin{cases}
			\partial_{i-1,\alpha}\gamma_{j,\beta}&(j<i-1)\\
			1&(j=i-1,i,\,\beta=\alpha)\\
			\partial_{j,\alpha}\sigma_j&(j=i-1,i,\,\beta=1-\alpha)\\
			\partial_{i,\alpha}\gamma_{j-1,\beta}&(j>i)
		\end{cases}\;&
		\sigma_j\gamma_{i,\alpha}&=
		\begin{cases}
			\gamma_{i-1,\alpha}\sigma_j&(j<i)\\
			\sigma_i\sigma_i&(j=i)\\
			\gamma_{i,\alpha}\sigma_{j+1}&(j>i)
		\end{cases}
	\end{alignat*}
	One may regard the cubical category $\Box$ as the category whose objects are $[1]^n$ for $n\ge 0$ and whose morphisms are freely generated by faces, degeneracies and connections, subject to the cubical identities.

	As in the simplicial setting, morphisms in the cubical category admit standard factorizations; see \cite[Theorem 5.1]{GM1}.

	\begin{lemma}
		\label{factorization}
		Every morphism $f\colon[1]^m\to[1]^n$ in the cubical category $\Box$ admits a unique factorization
		\[
		f=(\partial_{i_1,\alpha_1}\cdots\partial_{i_r,\alpha_r})(\gamma_{j_1,\beta_1}\cdots\gamma_{j_s,\beta_s})(\sigma_{k_1}\cdots\sigma_{k_t}),
		\]
		where $i_1>\cdots>i_r$, $j_1\le\cdots\le j_s$ and $k_1<\cdots<k_t$ such that $\beta_p\ne\beta_{p+1}$ whenever $j_p=j_{p+1}$.
	\end{lemma}

	\begin{definition}
		The category $\cSet$ is the functor category $\mathsf{Set}^{\Box^\mathrm{op}}$. The objects and morphisms of $\cSet$ are referred to as \emph{cubical sets} and \emph{cubical maps}.
	\end{definition}

	\begin{remark}
		In this paper we work with cubical sets equipped with \emph{connections}. Connections on cubical sets were introduced by Brown and Higgins \cite{BH} as additional degeneracies. Although connections are not always assumed in the literature, they play a crucial role in the homotopy theory of cubical sets and will be essential for our purposes.
	\end{remark}

	Let $X$ be a cubical set. For $n\ge 0$, we write $X_n=X([1]^n)$, and denote by $X$ the union of the sets $X_n$ over all $n\ge 0$, when no confusion is likely to arise. As in the simplicial setting, every cubical set is equivalently described by a collection of sets $\{X_n\}_{n\ge 0}$ together with faces, degeneracies, and connections satisfying the cubical identities. An element of $X_n$ is called an \emph{$n$-cube}. We introduce three basic examples of cubical sets.

	\begin{example}
		\begin{enumerate}
			\item The cubical set $\Box^n=\Box(-,[1]^n)$ is called the \emph{$n$-cube}.
			
			\item The \emph{boundary} of $\Box^n$ is the cubical subset
			\[
			\partial\Box^n=\bigcup_{\substack{i=1,\ldots,n\\\alpha=0,1}}\mathrm{Im}\,\partial_{i,\alpha}.
			\]
			
			\item For $i=1,\ldots,n$ and $\alpha=0,1$, the \emph{$(i,\alpha)$-horn} in $\Box^n$ is the cubical subset of $\partial\Box^n$ defined by
			\[
			\sqcap^n_{i,\alpha}=\bigcup_{(j,\beta)\ne(i,\alpha)}\mathrm{Im}\,\partial_{j,\beta}.
			\]
		\end{enumerate}
	\end{example}

	Let $X$ be a cubical set. By definition, there is a one-to-one correspondence between $X_n$ and the set of cubical maps $\Box^n\to X$. This justifies the terminology: both $\Box^n$ and the elements of $X_n$ are called $n$-cubes. Given a morphism $\theta$ in $\Box$ and an element $x\in X$, we may regard the cube obtained by applying $\theta$ to $x$ as the composite
	\[
	\Box^m\xrightarrow{\theta}\Box^n\xrightarrow{x}X.
	\]
	Hence we denote it by $x\theta$. An $n$-cube $x\in X_n$ is called \emph{degenerate} if there exist $i\in\{1,\ldots,n\}$ and an element $y\in X_{n-1}$ such that $x=y\sigma_i$. Otherwise, $x$ is called \emph{nondegenerate}.

	\begin{definition}
		A cubical set $X$ is called a \emph{Kan complex} if for any $n\ge 1$, $i=1,\ldots,n$ and $\alpha=0,1$, every cubical map $\sqcap_{i,\alpha}^n\to X$ extends to a cubical map $\Box^n\to X$.
	\end{definition}

	\begin{example}
		Define the singular functor $\mathrm{Sing}\colon\mathsf{Top}\to\cSet$ by setting $(\mathrm{Sing}\,X)_n$ to be the set of continuous maps $[0,1]^n\to X$, where faces, degeneracies and connections are induced from the corresponding maps between cubes. Then for every space $X$, the cubical set $\mathrm{Sing}\,X$ is a Kan complex.
	\end{example}

	\subsection{Products}
	
	We define the geometric product of cubical sets.
	
	\begin{definition}
		The functor
		\[
		\otimes\colon\cSet\times\cSet\to\cSet
		\]
		is defined as follows. For cubical sets $X$ and $Y$, the cubical set $X\otimes Y$ is given in degree $n$ by
		\[
		(X\otimes Y)_n=\left(\coprod_{k+l=n}X_k\times Y_l\right)/(x\sigma_{k+1},y)\sim(x,y\sigma_1)
		\]
		Faces, degeneracies and connections are defined by
		\begin{align*}
			(x,y)\partial_{i,\alpha}&=
			\begin{cases}
				(x\partial_{i,\alpha},y)&(i=1,\ldots,k)\\
				(x,y\partial_{i-k,\alpha})&(i=k+1,\ldots,k+l)
			\end{cases}\\
			(x,y)\sigma_i&=
			\begin{cases}
				(x\sigma_i,y)&(i=1,\ldots,k+1)\\
				(x,y\sigma_{i-k})&(i=k+1,\ldots,k+l+1)
			\end{cases}\\
			(x,y)\gamma_{i,\alpha}&=
			\begin{cases}
				(x\gamma_{i,\alpha},y)&(i=1,\ldots,k)\\
				(x,y\gamma_{i-k,\alpha})&(i=k+1,\ldots,k+l)
			\end{cases}
		\end{align*}
		where $(x,y)\in X_k\times Y_l$.
	\end{definition}

	\begin{example}
		It is straightforward to verify that there is a natural isomorphism
		\[
		\Box^m\otimes\Box^n\cong\Box^{m+n}.
		\]
	\end{example}

	By definition, the functor $\otimes\colon\cSet\times\cSet\to\cSet$ is a monoidal product. Define a functor $\otimes\colon\Box\times\Box\to\Box$ by $\otimes([1]^m,[1]^n)=[1]^{m+n}$. By \cite[Proposition 1.24]{DKLS}, the functor $\otimes\colon\cSet\times\cSet\to\cSet$ is the left Kan extension along the Yoneda embeddings, i.e., it fits into a commutative diagram
	\[
	\xymatrix{
		\Box\times\Box\ar[d]\ar[r]^\otimes&\Box\ar[d]\\
		\cSet\times\cSet\ar@{-->}[r]&\cSet
	}
	\]
	where the vertical maps are the Yoneda embeddings.

	We define elementary homotopies between cubical maps. Unlike the situation for continuous maps, elementary homotopies between cubical maps do not in general define an equivalence relation. However, elementary homotopy does induce an equivalence relation on maps into a Kan complex (see Proposition \ref{homotopy equivalence relation} below). The following definition is due to \cite[Definitions 2.8 and 2.11]{CK1}.

	\begin{definition}
		Let $f,g\colon X\to Y$ be cubical maps. An \emph{elementary homotopy} from $f$ to $g$ is a cubical map $h\colon X\otimes\Box^1\to Y$ such that the diagram
		\[
		\xymatrix{X\otimes\Box^0\ar[r]^{\partial_{1,0}}\ar[rd]_f&X\otimes\Box^1\ar[d]^h&X\otimes\Box^0\ar[l]_{\partial_{1,1}}\ar[ld]^g\\
			&Y}
		\]
		commutes. A \emph{cubical homotopy} from $f$ to $g$ is a zig-zag of elementary homotopies from $f$ to $g$. 
	\end{definition}
	The notion of cubical homotopy leads us to the notion of homotopy equivalence, which is defined in \cite[Definition 2.13]{CK1}.
	\begin{definition}
		A cubical map $f\colon X\to Y$ is a \emph{homotopy equivalence} if there is a cubical map $g\colon Y\to X$ such that $g\circ f$ is cubically homotopic to $1_X$ and $f\circ g$ is cubically homotopic to $1_Y$.
	\end{definition}

\subsection{Geometric realization}

For $i=1,\ldots,n$ and $\alpha=0,1$, let
\begin{align*}
	\partial_{i,\alpha}\colon [0,1]^{n-1}\to [0,1]^n,&\quad(x_1,\ldots,x_{n-1})\mapsto(x_1,\ldots,x_{i-1},\alpha,x_i,\ldots,x_{n-1}),\\
	\sigma_i\colon [0,1]^n\to [0,1]^{n-1},&\quad(x_1,\ldots,x_n)\mapsto(x_1,\ldots,x_{i-1},x_{i+1},\ldots,x_n).
\end{align*}
For $x\in X_n$, let $|x|=[0,1]^n$, where $X$ is a cubical set.

\begin{definition}
	The \emph{geometric realization} is a functor $|\cdot|\colon\cSet\to\mathsf{Top}$ defined by
	\[
	|X|=\left(\coprod_{x\in X}|x|\right)/\sim,
	\]
	where the equivalence relation is generated by $\partial_{i,\alpha}|x|\sim|x\partial_{i,\alpha}|$ and $\sigma_i|x|\sim|x\sigma_i|$ for all $i$ and $\alpha$.
\end{definition}

We record a property of the geometric realization of a cubical set that follows directly from the definition.

\begin{lemma}
	\label{realization CW}
	The geometric realization $|X|$ of every cubical set $X$ naturally carries a CW structure, whose cells are given by $\mathrm{Int}(|x|)$ for all nondegenerate cubes $x\in X$.
\end{lemma}

The category of simplicial sets carries a model category structure, where weak equivalences are simplicial maps that are homotopy equivalences on geometric realizations, and cofibrations are injective simplicial maps. The following analogue for cubical sets was proved by Cisinski \cite{C}.

\begin{theorem}
	\label{cSet model structure}
	The category $\cSet$ carries a model category structure in which weak equivalences are cubical maps that are homotopy equivalences on geometric realizations, and cofibrations are injective cubical maps.
\end{theorem}

Cisinski \cite{C} proved that $\cSet$, equipped with the model category structure of Theorem \ref{cSet model structure}, is Quillen equivalent to $\mathsf{Top}$ via geometric realization. Consequently, cubical sets provide a model for the homotopy theory of spaces. We recall some basic properties of homotopy equivalence between cubical sets.
\begin{lemma}
	\label{he is we}
	If a cubical map $f\colon X\to Y$ is a homotopy equivalence, then it is a weak equivalence.
\end{lemma}
\begin{lemma}
	\label{he vs we}
	Let $f\colon X\to Y$ be a cubical map between Kan complexes. Then $f$ is a homotopy equivalence if and only if it is a weak equivalence.
\end{lemma}

Although we require cubical sets to have connections, the geometric realization defined above does not impose the corresponding connection identifications. To justify this inconsistency, we define the geometric realization of a cubical set that takes into account connections, and show that the two geometric realizations of a cubical set are naturally homotopy equivalent. For $i=1,\ldots,n-1$ and $\alpha=0,1$, let
\begin{alignat*}{2}
	&\gamma_{i,0}\colon[0,1]^n\to[0,1]^{n-1},\quad&(x_1,\ldots,x_{n})\mapsto(x_1,\ldots,x_{i-1},\max\{x_i,x_{i+1}\},x_{i+2},\ldots,x_n)\\
	&\gamma_{i,1}\colon[0,1]^n\to[0,1]^{n-1},\quad&(x_1,\ldots,x_{n})\mapsto(x_1,\ldots,x_{i-1},\min\{x_i,x_{i+1}\},x_{i+2},\ldots,x_n).
\end{alignat*}

\begin{definition}
	Define a functor $|\cdot|_\gamma\colon\cSet\to\mathsf{Top}$ by
	\[
	|X|_\gamma=\left(\coprod_{x\in X}|x|\right)/\sim,
	\]
	where the equivalence relation is generated by $\partial_{i,\alpha}|x|\sim|x\partial_{i,\alpha}|$, $\sigma_i|x|\sim|x\sigma_i|$ and $\gamma_{i,\alpha}|x|\sim|x\gamma_{i,\alpha}|$ for all $i$ and $\alpha$.
\end{definition}

Antolini \cite{A} claimed that the following proposition holds. However, she only showed that for every cubical set $X$, the map $|X|\to|X|_\gamma$ induces isomorphisms on fundamental groups and on homology. It is well known that such a map between CW complexes need not be a homotopy equivalence. We provide a correct proof of the claim.

\begin{proposition}
	\label{realization connection}
	The canonical natural transformation $|\cdot|\Rightarrow|\cdot|_\gamma$ is a natural homotopy equivalence.
\end{proposition}
Let $X$ be a cubical set. The \emph{cube category} $\Box\downarrow X$ is the category whose objects are cubical maps $\Box^n\to X$ for $n\ge 0$ and morphisms are commutative diagrams
\begin{equation}
	\label{cube cat morphism}
	\xymatrix{
		\Box^m\ar[r]^\theta\ar[d]&\Box^n\ar[d]\\
		X\ar@{=}[r]&X
	}
\end{equation}
where $\theta$ is any cubical map. 
Let $|\Box|\colon\Box\downarrow X\to\mathsf{Top}$ be a functor that assigns to every $\Box^n\to X$ the geometric $n$-cube $[0,1]^n$. Then there is a natural homeomorphism
\begin{equation}
	\label{cube cat homeo}
	|X|_\gamma\cong\underset{\Box\downarrow X}{\mathrm{colim}}\,|\Box|.
\end{equation}

To prove Proposition \ref{realization connection}, we need the following lemma, which is also proved in \cite[Proposition 3.12, Corollary 3.14]{ACK}.
\begin{lemma}
	\label{cube cat hocolim}
	Let $X$ be a cubical set. Then the canonical map
	\[
	\underset{\Box\downarrow X}{\mathrm{hocolim}}\,|\Box|\to\underset{\Box\downarrow X}{\mathrm{colim}}\,|\Box|
	\]
	is a homotopy equivalence.
\end{lemma}

\begin{proof}
	We refer to \cite{H} for Reedy categories. By Lemma \ref{factorization}, the cube category $\mathcal{C}=\Box\downarrow X$ is a Reedy category such that $\overset{\rightarrow}{\mathcal{C}}$ consists of composites of faces and $\overset{\leftarrow}{\mathcal{C}}$ consists of composites of degeneracies and connections. Let $x\colon\Box^n\to X$ be an object of $\mathcal{C}$. Since the latching object for $x$ is $\partial|x|$, the functor $|\Box|$ is Reedy cofibrant.
	
	If $x$ is nondegenerate and not a connection, then the matching category of $\mathcal{C}$ at $x$ is empty. Suppose $x$ is degenerate or a connection. By the cubical Eilenberg--Zilber lemma \cite[Proposition 3.6 and Theorem 7.9]{Ca}, there is a unique decomposition
	\[
	x=y\alpha,
	\]
	where $y$ is nondegenerate and not a connection, and $\alpha$ is a nonidentity composite of degeneracies and connections.
	
	Let $\beta\colon x\to z$ be any object of the matching category at $x$. Applying the same lemma to $z$, write
	\[
	z=w\gamma,
	\]
	where $w$ is nondegenerate and not a connection, and $\gamma$ is a composite of degeneracies and connections, possibly the identity. Then
	\[
	x=z\beta=w(\gamma\beta)=y\alpha.
	\]
	By uniqueness of the decomposition, we have $w=y$ and $\gamma\beta=\alpha$. Hence $\gamma$ defines a morphism
	\[
	(x\xrightarrow{\beta}z)\xrightarrow{\gamma}
	(x\xrightarrow{\alpha}y)
	\]
	in the matching category. Thus the matching category of $\mathcal{C}$ at $x$ is connected.
	
	Consequently, by \cite[Proposition 15.10.2]{H}, the Reedy category $\mathcal{C}$ has fibrant constants. Since $|\Box|$ is Reedy cofibrant, \cite[Theorem 19.9.1(1)]{H} implies that the canonical map
	\[
	\underset{\Box\downarrow X}{\mathrm{hocolim}}\,|\Box|
	\longrightarrow
	\underset{\Box\downarrow X}{\mathrm{colim}}\,|\Box|
	\]
	is a weak homotopy equivalence. Both spaces have the homotopy type of CW complexes, so the map is a homotopy equivalence.
\end{proof}

For a small category $\mathcal{C}$, let $B\mathcal{C}$ denote the classifying space of $\mathcal{C}$, i.e. the simplicial geometric realization of the nerve of $\mathcal{C}$.

\begin{lemma}
	\label{cube cat classifying space}
	Let $X$ be a cubical set. Then there is a natural homotopy equivalence
	\[
	|X|_\gamma\simeq B(\Box\downarrow X).
	\]
\end{lemma}

\begin{proof}
	By definition, $B(\Box\downarrow X)$ is the homotopy colimit of the constant functor over $\Box\downarrow X$. Hence by the homotopy invariance of homotopy colimits, there is a natural homotopy equivalence
	\[
	B(\Box\downarrow X)\simeq\underset{\Box\downarrow X}{\mathrm{hocolim}}\,|\Box|.
	\]
	The claim follows from Lemma \ref{cube cat hocolim}.
\end{proof}

We are ready to prove Proposition \ref{realization connection}.

\begin{proof}
	[Proof of Proposition \ref{realization connection}]
	Let $X$ be a cubical set, and let $\overline{\Box\downarrow X}$ denote the category whose objects are cubical maps $\Box^n\to X$ and morphisms are commutative diagrams \eqref{cube cat morphism} such that $\theta$ is the composite of faces and degeneracies. Then
	\[
	\underset{\overline{\Box\downarrow X}}{\mathrm{colim}}\,|\Box|=|X|
	\]
	and analogously to Lemma \ref{cube cat classifying space}, we can show that
	\begin{equation}
		\label{realization classfying space without connections}
		|X|\simeq B(\overline{\Box\downarrow X}).
	\end{equation}
	Hence it is sufficient to show that the inclusion $\iota\colon\overline{\Box\downarrow X}\to\Box\downarrow X$ induces a homotopy equivalence on classifying spaces. For every cubical map $x\colon\Box^n\to X$, the comma category $\iota/x$ is isomorphic to $\overline{\Box\downarrow\Box^n}$. The inclusion $\overline{\Box\downarrow\Box^n}\to\Box\downarrow\Box^n$ induces the quotient map
	\[
	|\Box^n|=\underset{\overline{\Box\downarrow \Box^n}}{\mathrm{colim}}\,|\Box|\to\underset{\Box\downarrow\Box^n}{\mathrm{colim}}\,|\Box|=|\Box^n|_\gamma.
	\]
	By \cite[Propositions 2.1 and 2.2]{A}, the quotient map induces isomorphisms on $\pi_1$ and homology. Clearly, $|\Box^n|_\gamma\cong[0,1]^n$ is contractible. Hence $|\Box^n|$ is also contractible, and by \eqref{realization classfying space without connections}, $B(\iota/x)\cong B(\overline{\Box\downarrow\Box^n})$ is also contractible. By Quillen's Theorem A, the map $B(\iota)\colon B(\overline{\Box\downarrow X})\to B(\Box\downarrow X)$ is a homotopy equivalence.
\end{proof}
Thus adding the connection identifications does not change the homotopy type. 

\subsection{Homology}

We recall the homology of cubical sets. Let $X$ be a cubical set. Define a chain complex $C(X)$ by setting $C_n(X)$ to be the free abelian group generated by $X_n$, with boundary operator
\[
\partial x=\sum_{i=1}^n\sum_{\alpha=0,1}(-1)^{i+\alpha}x\partial_{i,\alpha}
\]
for $x\in X_n$. Let $D_n(X)$ denote the subgroup of $C_n(X)$ generated by degenerate $n$-cubes of $X$. By the cubical identities, $\partial D_n(X)\subset D_{n-1}(X)$, so $D(X)$ is a chain subcomplex of $C(X)$. Define the chain complex
\[
N(X)=C(X)/D(X)
\]
which is called the \emph{normalized chain complex} of $X$. Let $\mathsf{GrAb}$ denote the category of graded abelian groups.

\begin{definition}
	Define a functor
	\[
	H_*\colon\cSet\to\mathsf{GrAb},\quad X\mapsto H_*(N(X)).
	\]
	Each $H_n(X)$ is called the $n$-th homology of a cubical set $X$.
\end{definition}

We record some fundamental properties of the homology of cubical sets; see, for example, \cite{BGJW,F}.

\begin{proposition}
	Let $f,g\colon X\to Y$ be cubical maps between cubical sets $X$ and $Y$. If $f$ and $g$ are cubically homotopic, then
	\[
	f_*=g_*\colon H_*(X)\to H_*(Y).
	\]
\end{proposition}

\begin{proposition}
	For every cubical set $X$, there is a natural isomorphism
	\[
	H_*(X)\cong H_*(|X|).
	\]
\end{proposition}

\subsection{Homotopy groups}

In \cite{CK1}, Carranza and Kapulkin define the homotopy groups of cubical Kan complexes in close analogy with the simplicial setting. We briefly recall their construction.

A \emph{pair of cubical sets} $(X,A)$ consists of a cubical set $X$ together with a cubical subset $A\subset X$. A \emph{cubical map of pairs} $f\colon(X,A)\to(Y,B)$ is a cubical map $f\colon X\to Y$ such that $f(A)\subset B$. A cubical homotopy between cubical maps of pairs $(X,A)\to(Y,B)$ is defined in the obvious manner. In particular, for maps $(\Box^n,\partial\Box^n)\to(X,x_0)$, an elementary cubical homotopy has domain $(\Box^n\otimes\Box^1,\partial\Box^n\otimes\Box^1)$; its two time faces need not be constant.

Let $X$ be a cubical set. For $x\in X_0$, we denote by the same symbol $x$ the cubical subset of $X$ consisting of the cubes $x\sigma_1\cdots\sigma_n$ for all $n\ge 0$. A \emph{pointed cubical set} is a pair of cubical sets $(X,x_0)$ for some $x_0\in X_0$. A \emph{pointed cubical map} $f\colon(X,x_0)\to(Y,y_0)$ between pointed cubical sets is a cubical map of pairs. A homotopy equivalence between pointed cubical sets is defined in the obvious manner.

The following proposition is proved in \cite[Proposition 2.10]{CK1}.

\begin{proposition}
	\label{homotopy equivalence relation}
	Let $X$ be a cubical set, and let $Y$ be a Kan complex. Then the relation of elementary homotopy defines an equivalence relation on the set of cubical maps from $X$ to $Y$.
\end{proposition}

This proposition readily generalizes to pointed cubical maps.

\begin{definition}
	Let $(X,x_0)$ be a pointed Kan complex. The $n$-th homotopy group $\pi_n(X,x_0)$ is defined as the set of elementary homotopy classes of cubical maps $(\Box^n,\partial\Box^n)\to(X,x_0)$.
\end{definition}
We denote the elementary homotopy class by $[-]_\Box$. By definition, $\pi_0(X,x_0)$ coincides with $\pi_0(|X|,x_0)$, the set of path components of $|X|$. We now define a group structure on $\pi_n(X,x_0)$ for $n\ge 1$. Let $f,g\colon(\Box^n,\partial\Box^n)\to(X,x_0)$ be cubical maps. Define a cubical map $h\colon\sqcap^{n+1}_{n+1,0}\to X$ by setting
\[
h\partial_{i,\alpha}=\begin{cases}
	f&((i,\alpha)=(n,0))\\
	g&((i,\alpha)=(n+1,1))\\
	x_0&(\text{otherwise}),
\end{cases}
\]
where $x_0$ denotes the constant cubical map with value $x_0$. Since $X$ is a Kan complex, the cubical map $h$ extends to a cubical map $\bar{h}\colon\Box^{n+1}\to X$. As shown in \cite{CK1}, the elementary homotopy class of a cubical map $\bar{h}\partial_{n+1,0}\colon(\Box^n,\partial\Box^n)\to(X,x_0)$ is independent of the choice of extension $\bar{h}$. We therefore define the product of elementary homotopy classes by
\begin{equation}
	\label{product Kan complex}
	[f]_\Box\cdot [g]_\Box=[\bar{h}\partial_{n+1,0}]_\Box.
\end{equation}
The following lemma is proved in \cite{CK1}.

\begin{lemma}
	Let $(X,x_0)$ be a pointed Kan complex. For $n\ge 1$, the product \eqref{product Kan complex} defines a group structure on $\pi_n(X,x_0)$.
\end{lemma}

We record some fundamental properties of the homotopy groups of Kan complexes; see \cite[Theorem 3.17 and Proposition 4.3]{CK1}.

\begin{proposition}
	Let $f\colon(X,x_0)\to(Y,y_0)$ be a pointed cubical map between pointed Kan complexes. Then
	\[
	f_*\colon\pi_n(X,x_0)\to\pi_n(Y,y_0),\quad[x]_\Box\mapsto[f\circ x]_\Box
	\]
	is well-defined. Moreover, for $n\ge 1$, it is a group homomorphism.
\end{proposition}

\begin{proposition}
	Let $f,g\colon(X,x_0)\to(Y,y_0)$ be pointed cubical maps between pointed Kan complexes. If $f$ and $g$ are pointed cubically homotopic, then
	\[
	f_*=g_*\colon\pi_n(X,x_0)\to\pi_n(Y,y_0).
	\]
\end{proposition}
\begin{corollary}
	Let $(X,x_0)$, $(Y,y_0)$ be pointed Kan complexes. If $(X,x_0)$ and $(Y,y_0)$ have the same pointed homotopy type, then there is an isomorphism
	\[
	\pi_n(X,x_0)\cong\pi_n(Y,y_0).
	\]
\end{corollary}

\begin{proposition}
	Let $(X,x_0)$ be a pointed Kan complex. If $n\ge 2$, then $\pi_n(X,x_0)$ is an abelian group.
\end{proposition}

One of the main results of \cite[Theorem 3.24]{CK1} is the following theorem.

\begin{theorem}
	\label{homotopy group realization}
	Let $(X,x_0)$ be a pointed Kan complex. There is a natural isomorphism
	\[
	\pi_*(X,x_0)\cong\pi_*(|X|,x_0).
	\]
\end{theorem}

\section{Discrete homotopy groups of cubical sets}\label{Discrete homotopy groups of cubical sets}

In this section, we define the discrete homotopy groups of a cubical set in a purely combinatorial manner, extending the notion of the discrete homotopy groups of a graph, and show their fundamental properties.
\begin{definition}
	For $l\ge1$ and 
	\[
	s\colon\{1,\ldots,l\}\to\{0,1\},
	\]
	let $I_s$ be the colimit of the diagram
	\[
	\xymatrix{
		&\Box^0\ar[dl]_{\partial_{1,1-s(1)}}\ar[dr]_{\partial_{1,s(2)}}&&\Box^0\ar[dl]_{\partial_{1,1-s(2)}}\ar[dr]_{\partial_{1,s(3)}}&&\cdots\ar[dl]_{\partial_{1,1-s(3)}}&\cdots\ar[dr]_{\partial_{1,s(l-1)}}&&\Box^0\ar[dl]_{\partial_{1,1-s(l-1)}}\ar[dr]_{\partial_{1,s(l)}}\\
		\Box^1&&\Box^1&&\Box^1&\cdots&\cdots&\Box^1&&\Box^1,
	}
	\]
	and let
	\[
	a_k\colon\Box^1\to I_s
	\]
	be the inclusion into the $k$-th 1-cube in $I_s$ for $k=1,\ldots,l$. We identify
	\[
	(I_s)_0=[l]=\{0,\ldots,l\}. 
	\]
	For 
	\[
	s'\colon\{1,\ldots,l'\}\to\{0,1\},
	\]
	we will identify a cubical map $I_s\to I_{s'}$ with the function
	\[
	[l]=(I_s)_0\to(I_{s'})_0=[l'].
	\]
\end{definition}
\begin{definition}
	Let $X$ be a cubical set. A $0$-grid in $X$ is an element of $X_0$, and for $n\ge 1$, an \emph{$n$-grid} of size $(l_1,\ldots,l_n)$ in $X$ is a cubical map
	\[
	x\colon I_{s_1}\otimes\cdots\otimes I_{s_n}\to X,
	\]
	where
	\[
	s_k\colon\{1,\ldots,l_k\}\to\{0,1\}
	\]
	for $k=1,\ldots,n$.
\end{definition}
Thus an $n$-grid is a cubical map from a finite rectangular cubical subdivision of the $n$-cube into $X$, where $s=(s_1,\ldots,s_n)$ records whether the $k$-th coordinate of each elementary cube is traversed in the standard or reversed direction for $k=1,\ldots,n$.
\begin{lemma}
	\label{grid unpack}
	For a cubical set $X$ and $n\ge1$, an $n$-grid of size $(l_1,\ldots,l_n)$ in $X$
	\[
	x\colon I_{s_1}\otimes\cdots\otimes I_{s_n}\to X
	\]
	is equivalent to a function
	\[
	x\colon\prod_{k=1}^n\{1,\ldots,l_k\}\to X_n
	\]
	satisfying
	\[
	x(i_1,\ldots,i_n)\partial_{k,1-s_k(i_k)}=x(i_1,\ldots,i_{k-1},i_k+1,i_{k+1},\ldots,i_n)\partial_{k,s_k(i_k+1)}
	\]
	for all $i_1,\ldots,i_n$ and $k$ such that $i_k<l_k$.
\end{lemma}
The proof is clear by observing that $x(i_1,\ldots,i_n)\in X_n$ is precisely the cubical map
\[
\Box^n\xrightarrow{a_{i_1}\otimes\cdots\otimes a_{i_n}}I_{s_1}\otimes\cdots\otimes I_{s_n}\xrightarrow{x}X.
\]
\begin{example}
	Consider a $2$-grid 
	\[x\colon I_{s_1}\otimes I_{s_2}\to X\] 
	of size $(3,2)$, where $s_1$, $s_2$ are specified by $s_1(1)=0,\,s_1(2)=1,\,s_1(3)=1$ and $s_2(1)=0,\,s_2(2)=1$. This grid is illustrated as follows.
	
	\begin{figure}[htbp]
		\centering
		\begin{tikzpicture}[x=5mm, y=5mm, thick]
			\draw[myarrow=.6](0,0)--(0,3);
			\draw[myarrow=.6](3,0)--(3,3);
			\draw[myarrow=.6](6,0)--(6,3);
			\draw[myarrow=.6](9,0)--(9,3);
			\draw[myarrow=.6](0,6)--(0,3);
			\draw[myarrow=.6](3,6)--(3,3);
			\draw[myarrow=.6](6,6)--(6,3);
			\draw[myarrow=.6](9,6)--(9,3);
			\draw[myarrow=.6](0,0)--(3,0);
			\draw[myarrow=.6](0,3)--(3,3);
			\draw[myarrow=.6](0,6)--(3,6);
			\draw[myarrow=.6](6,0)--(3,0);
			\draw[myarrow=.6](6,3)--(3,3);
			\draw[myarrow=.6](6,6)--(3,6);
			\draw[myarrow=.6](9,0)--(6,0);
			\draw[myarrow=.6](9,3)--(6,3);
			\draw[myarrow=.6](9,6)--(6,6);
			\fill[black](0,0) circle(2pt);
			\fill[black](3,0) circle(2pt);
			\fill[black](6,0) circle(2pt);
			\fill[black](9,0) circle(2pt);
			\fill[black](0,3) circle(2pt);
			\fill[black](3,3) circle(2pt);
			\fill[black](6,3) circle(2pt);
			\fill[black](9,3) circle(2pt);
			\fill[black](0,6) circle(2pt);
			\fill[black](3,6) circle(2pt);
			\fill[black](6,6) circle(2pt);
			\fill[black](9,6) circle(2pt);
			\draw(1.5,0.9) node[above]{$x(1,1)$};
			\draw(4.5,0.9) node[above]{$x(2,1)$};
			\draw(7.5,0.9) node[above]{$x(3,1)$};
			\draw(1.5,3.9) node[above]{$x(1,2)$};
			\draw(4.5,3.9) node[above]{$x(2,2)$};
			\draw(7.5,3.9) node[above]{$x(3,2)$};
		\end{tikzpicture}
	\end{figure}
\end{example}

We introduce the cubical operators on grids. Denote
\[
I_s\partial_{1,0}=a_1\partial_{1,s(1)},\quad I_s\partial_{1,1}=a_l\partial_{1,1-s(l)},\quad\partial I_s=I_s\partial_{1,0}\sqcup I_s\partial_{1,1}.
\]
Then we define
\[
\partial_{k,\alpha}=(1^{\otimes k-1}\otimes\partial_{1,\alpha}\otimes1^{\otimes n-k})\colon I_{s_1}\otimes\cdots\otimes I_{s_{k-1}}\otimes\Box^0\otimes{I_{s_{k+1}}}\otimes\cdots\otimes I_{s_n}\to I_{s_1}\otimes\cdots\otimes I_{s_n}
\]
for $k=1,\ldots,n$, $\alpha=0,1$. The boundary of a grid $x$ is defined by
\[
\partial x=\bigcup_{i,\alpha}x\partial_{i,\alpha}.
\]
For $A=I_{s_1}\otimes\cdots\otimes I_{s_n}$, we write
\[
\partial A=\bigcup_{k=1}^n I_{s_1}\otimes\cdots\otimes I_{s_{k-1}}\otimes\partial I_{s_k}\otimes I_{s_{k+1}}\otimes\cdots\otimes I_{s_n}.
\]
Thus $\partial A$ denotes the boundary of the domain, whereas $\partial x$ denotes its image under a grid $x$.
We also define
\[
\gamma_{1,\alpha}\colon I_s\otimes I_s\to I_s,\quad\gamma_{1,\alpha}(j_0,j_1)=\begin{cases}
	a_{j_{1-\alpha}}\sigma_{1+\alpha}&(j_0<j_1)\\
	a_{j_0}\gamma_{1,s(j_0)+\alpha}&(j_0=j_1)\\
	a_{j_\alpha}\sigma_{2-\alpha}&(j_0>j_1)
\end{cases}
\]
for $s\colon\{1,\ldots,l\}\to\{0,1\}$, $j_0,j_1\in\{1,\ldots,l\}$, $\alpha=0,1$. We also define
\[
\sigma_{k,I_{s_k}}=(1^{\otimes k-1}\otimes*\otimes1^{\otimes n-k})\colon I_{s_1}\otimes\cdots\otimes I_{s_n}\to I_{s_1}\otimes\cdots\otimes I_{s_{k-1}}\otimes\Box^0\otimes{I_{s_{k+1}}}\otimes\cdots\otimes I_{s_n}
\]
for $k=1,\ldots,n$, which is induced by the unique map
\[
*\colon I_{s_k}\to\Box^0.
\]
We sometimes abbreviate $\sigma_{k,I_s}$ as $\sigma_k$ for simplicity when it does not cause any confusion. It is straightforward to verify that faces, degeneracies and connections satisfy the cubical identities. Let $f\colon X\to Y$ be a cubical map. For every $n$-grid $x$ of size $(l_1,\ldots,l_n)$ in $X$, $f\circ x$ is an $n$-grid of size $(l_1,\ldots,l_n)$ in $Y$.

Next we introduce some equivalence relations between grids. For
\[
s\colon\{1,\ldots,l\}\to\{0,1\},\quad s'\colon\{1,\ldots,l+1\}\to\{0,1\},
\]
we say $s$ is a contraction of $s'$ if there is $i=1,\ldots,l+1$ such that
\[
s'(p)=\begin{cases}
	s(p)&(p=1,\ldots,i-1)\\
	s(p-1)&(p=i+1,\ldots,l+1).
\end{cases}
\]
This induces a cubical map
\[
r\colon I_{s'}\to I_{s},\quad(0,1,\ldots,i-1,i,i+1,\ldots,l+1)\mapsto(0,1,\ldots,i-1,i-1,i,\ldots,l).
\]
For $n$-grids
\[
x\colon I_{s_1}\otimes\cdots\otimes I_{s_{k-1}}\otimes I_{s_k'}\otimes I_{s_{k+1}}\otimes\cdots\otimes I_{s_n}\to X\text{ and }y\colon I_{s_1}\otimes\cdots\otimes I_{s_{k-1}}\otimes I_{s_k}\otimes I_{s_{k+1}}\otimes\cdots\otimes I_{s_n}\to X,
\]
we write 
\[
x\xrightarrow{k}y
\]
if
\[
x=y(1^{\otimes k-1}\otimes r\otimes 1^{\otimes n-k}).
\]
If $x=x_1\xrightarrow{k_1}\cdots\xrightarrow{k_{m-1}}x_m=y$ for some $x_1,\ldots,x_m$ and $k_1,\ldots,k_{m-1}$, then we write $x\ge y$. Clearly, this provides a poset structure on the set of $n$-grids.

\begin{definition}
	\label{0homotopy}
	We say that $n$-grids $x$ and $y$ are \emph{$0$-homotopic} if there exist $n$-grids $x_1,\ldots,x_m$ in $X$ such that $x=x_1$, $y=x_m$, and for each $i=1,\ldots,m-1$, either $x_i\le x_{i+1}$ or $x_i\ge x_{i+1}$ holds. In this case, we write $x\simeq_0y$.
\end{definition}
We introduce the concatenation of grids. Geometrically, gluing two cubes of the same dimension along a common face yields a new cube. The concatenation of grids serves as a combinatorial abstraction of this observation, which was originally introduced in \cite{BHS} to define cubical $\omega$-groupoids.
\begin{definition}
	Let 
	\[
	x\colon I_{s_1}\otimes\cdots\otimes I_{s_n}\to X\text{ and }y\colon I_{s'_1}\otimes\cdots\otimes I_{s'_n}\to X
	\] 
	be $n$-grids of size $(l_1,\ldots,l_n)$ and $(l_1',\ldots,l_n')$ in a cubical set $X$ for $n\ge 1$. If $x\partial_{k,1}=y\partial_{k,0}$ for $k\in\{1,\ldots,n\}$, we have $l_i=l_i'$, $s_i=s'_i$ for $i\in\{1,\ldots,n\}-\{k\}$. Then there is a pushout diagram
	\[
	\xymatrix{
		I_{s_1}\otimes\cdots\otimes I_{s_{k-1}}\otimes I_{s_{k+1}}\otimes\cdots\otimes I_{s_n}\ar[r]^{\partial_{k,1}}\ar[d]_{\partial_{k,0}}&I_{s_1}\otimes\cdots\otimes I_{s_n}\ar[d]&\\
		I_{s'_1}\otimes\cdots\otimes I_{s'_n}\ar[r]&I_{s_1}\otimes\cdots\otimes I_{s_{k-1}}\otimes I_{s_k\cdot s'_k}\otimes I_{s_{k+1}}\otimes\cdots\otimes I_{s_n}
	}
	\]
	where
	\[
	s_k\cdot s'_k\colon\{1,\ldots,l_k+l_k'\}\to\{0,1\},\quad (s_k\cdot s'_k)(p)=\begin{cases}
		s_k(p)&(p=1,\ldots,l_k)\\
		s_k'(p-l_k)&(p=l_k+1,\ldots,l_k+l'_k).
	\end{cases}
	\]
	We define the grid
	\[x+_ky\colon I_{s_1}\otimes\cdots\otimes I_{s_{k-1}}\otimes I_{s_k\cdot s'_k}\otimes I_{s_{k+1}}\otimes\cdots\otimes I_{s_n}\to X\] 
	to be the map induced by $x$ and $y$.
\end{definition}

It is straightforward to verify that $+_k$ is associative and satisfies an interchange law
\begin{equation}
	\label{interchange law}
	(x_1+_ky_1)+_l(x_2+_ky_2)=(x_1+_lx_2)+_k(y_1+_ly_2)
\end{equation}
for $k\ne l$, whenever $+_k$ and $+_l$ are well-defined. Let $x$ be an $n$-grid of size $(l_1,\ldots,l_n)$. For 
\[
s\colon\{1,\ldots,l\}\to\{0,1\}
\]
and $j=1,\ldots,l$, let
\[
s[j]\colon\{1\}\to\{0,1\},\quad s[j](1)=s(j).
\]
Then we have
\[
s=s[1]\cdot s[2]\cdots s[l].
\]
For an $n$-grid
\[
x\colon I_{s_1}\otimes\cdots\otimes I_{s_n}\to X,
\]
$k=1,\ldots,n$ and $j=1,\ldots,l_k$, we define the \emph{slice} $x[k,j]$ of $x$ as the restriction of $x$ to
\[
I_{s_1}\otimes\cdots\otimes I_{s_{k-1}}\otimes I_{s_k[j]}\otimes I_{s_{k+1}}\otimes\cdots\otimes I_{s_n}.
\]
By definition, $x[k,j]+_kx[k,j+1]$ is well-defined for $j=1,\ldots,l_k-1$, and we have
\[
x=x[k,1]+_k\cdots+_kx[k,l_k].
\]
We provide an equivalent description of 0-homotopy as follows.
\begin{lemma}
	For a cubical set $X$ and $n$-grids $x$, $y$ of size $(l_1,\ldots,l_{k-1},l_k,l_{k+1},\ldots,l_n)$ and $(l_1,\ldots,l_{k-1},l_k-1,l_{k+1},\ldots,l_n)$, we have
	\[
	x\xrightarrow{k}y
	\]
	for $k=1,\ldots,n$ if and only if there is $j=1,\ldots,l_k$ such that
	\[
	x[k,j]=z\sigma_{k,I_{s_k[j]}}
	\]
	for some $(n-1)$-grid $z$, and
	\[
	y=x[k,1]+_k\cdots+_kx[k,j-1]+_kx[k,j+1]+_k\cdots+_kx[k,l_k].
	\]
\end{lemma}

In other words, 0-homotopy only inserts or deletes degenerate slices of a grid. Let $x$ and $y$ be $n$-grids of the same size $(l_1,\ldots,l_n)$ in a cubical set $X$. If there exists an $(n+1)$-grid $z$ of size $(l_1,\ldots,l_{k-1},l,l_k,\ldots,l_n)$ such that $z\partial_{k,0}=x$ and $z\partial_{k,1}=y$, then we write
\[
x\overset{k}{\Leftrightarrow}y.
\]

\begin{definition}
	\label{1homotopy}
	We say that $n$-grids $x$ and $y$ are \emph{$1$-homotopic} if there exist $n$-grids $x_1,\ldots,x_m$ in $X$ such that $x=x_1$, $y=x_m$, and for each $i=1,\ldots,m-1$, $x_i\overset{l_i}{\Leftrightarrow}x_{i+1}$ holds for some $l_i$. In this case, we write $x\simeq_1y$.
\end{definition}

\begin{definition}
	\label{homotopy grids def}
	We say that $n$-grids $x$ and $y$ are \emph{grid homotopic} if there exist $n$-grids $x_1,\ldots,x_k$ such that $x=x_1$, $y=x_k$, and for each $i=1,\ldots,k-1$, either $x_i\simeq_0x_{i+1}$ or $x_i\simeq_1x_{i+1}$ holds. In this case, we write $x\simeq y$.
\end{definition}
Note that 0-homotopy, 1-homotopy and grid homotopy are equivalence relations between grids.

\begin{remark}
	In \cite{KT}, the authors defined a $C_r$-homotopy for $r\ge 0$ between paths in a directed graph, and proved the Hurewicz theorem in dimension one and the Seifert-van Kampen theorem for the $r$-fundamental group defined by $C_r$-homotopies. A $0$-homotopy and a grid homotopy between $1$-grids correspond to a $C_0$-homotopy and a $C_2$-homotopy, respectively.
\end{remark}

The following lemma is immediate from the definition.

\begin{lemma}
	\label{homotopy description new}
	Let $X$ be a cubical set. Two $n$-grids $x$ and $y$ are grid homotopic if and only if there exist $n$-grids $z$ and $w$ such that
	\[
	x\le z\simeq_1 w\ge y.
	\]
\end{lemma}
Let $(X,x_0)$ be a pointed cubical set. A \emph{spherical} $n$-grid is a cubical map of pairs
\[
x\colon\bigl(I_{s_1}\otimes\cdots\otimes I_{s_n},\partial(I_{s_1}\otimes\cdots\otimes I_{s_n})\bigr)\to(X,x_0).
\]
For $n=0$, the domain is $(\Box^0,\varnothing)$. The set of spherical $n$-grids is closed under $0$-homotopy. A $1$-homotopy between spherical grids 
\[
x,y\colon\bigl(I_{s_1}\otimes\cdots\otimes I_{s_n},\partial(I_{s_1}\otimes\cdots\otimes I_{s_n})\bigr)\to(X,x_0)
\]
is an $(n+1)$-grid $z$ such that
\[
z\partial_{i,\alpha}=\begin{cases}
	x&((i,\alpha)=(k,0))\\
	y&((i,\alpha)=(k,1))\\
	x_0&(i=1,\ldots,k-1,k+1,\ldots,n+1)
\end{cases}
\]
for some $k=1,\ldots,n+1$. Hence we obtain grid homotopy of spherical grids.

\begin{definition}
	Let $(X,x_0)$ be a pointed cubical set. For $n\ge 0$, the $n$-th \emph{discrete homotopy group} $\pi_n^\delta(X,x_0)$ is defined as the set of grid homotopy classes of spherical grids.
\end{definition}
Let the grid homotopy class be denoted by $[-]$. We define a group structure on $\pi_n^\delta(X,x_0)$ for $n\ge 1$. Let 
\[
\begin{gathered}x\colon\bigl(I_{s_1}\otimes\cdots\otimes I_{s_n},\partial(I_{s_1}\otimes\cdots\otimes I_{s_n})\bigr)\to(X,x_0),\\ y\colon\bigl(I_{s_1'}\otimes\cdots\otimes I_{s_n'},\partial(I_{s_1'}\otimes\cdots\otimes I_{s_n'})\bigr)\to(X,x_0)\end{gathered}
\]
be spherical $n$-grids of size $(l_1,\ldots,l_n)$ and $(l'_1,\ldots,l'_n)$ in a pointed cubical set $(X,x_0)$ for $n\ge 1$. Define the multiplication $x\cdot y$ as the spherical $n$-grid
\[
\bigl(I_{s_1\cdot s_1'}\otimes\cdots\otimes I_{s_n\cdot s_n'},\partial(I_{s_1\cdot s_1'}\otimes\cdots\otimes I_{s_n\cdot s_n'})\bigr)\to(X,x_0)
\] 
of size $(l_1+l'_1,\ldots,l_n+l'_n)$ such that
\begin{align*}
	(x\cdot y)(i_1,\ldots,i_n)&=
	\begin{cases}
		x(i_1,\ldots,i_n)&(i_1\le l_1,\ldots,i_n\le l_n)\\
		y(i_1-l_1,\ldots,i_n-l_n)&(i_1>l_1,\ldots,i_n>l_n)\\
		x_0\sigma_1\cdots\sigma_n&(\text{otherwise}).
	\end{cases}
\end{align*}
We therefore define the multiplication of grid homotopy classes by
\[
[x]\cdot[y]=[x\cdot y].
\]

Thus the multiplication is obtained by placing the two spherical grids in diagonal blocks and sending all remaining blocks to the basepoint. The following lemma is immediate from the definitions of grid homotopies between grids and the multiplication of grids.

\begin{lemma}
	\label{grid product}
	For $n\ge 1$ and $i=1,2$, let $x_i$ and $y_i$ be spherical $n$-grids in a pointed cubical set. If $x_1\simeq x_2$ and $y_1\simeq y_2$, then
	\[
	x_1\cdot y_1\simeq x_2\cdot y_2.
	\]
\end{lemma}

Next we show that the multiplication of grids is grid homotopic to the concatenation of grids along any direction. The idea is to slide the grids along any direction through grid homotopy.
\begin{lemma}
	\label{grid slide}
	Let $x$ and $y$ be spherical $n$-grids in a pointed cubical set $(X,x_0)$. Then there exist spherical $n$-grids $a$ and $b$ of the same size such that $x\le a$, $y\le b$ and for each $i=1,\ldots,n$,
	\[
	x\cdot y\simeq a+_ib.
	\]
\end{lemma}

\begin{proof}
	Let 
	\[
	\begin{gathered}	x\colon\bigl(I_{s_1}\otimes\cdots\otimes I_{s_n},\partial(I_{s_1}\otimes\cdots\otimes I_{s_n})\bigr)\to(X,x_0),\\ y\colon\bigl(I_{s_1'}\otimes\cdots\otimes I_{s_n'},\partial(I_{s_1'}\otimes\cdots\otimes I_{s_n'})\bigr)\to(X,x_0).\end{gathered}
	\]
	By inserting degenerate slices if necessary, we can construct $n$-grids 
	\[
	a,b\colon\bigl(I_{s_1''}\otimes\cdots\otimes I_{s_n''},\partial(I_{s_1''}\otimes\cdots\otimes I_{s_n''})\bigr)\to(X,x_0)
	\]
	such that $x\le a$ and $y\le b$. We denote any $n$-grid consisting of $x_0\sigma_1\cdots\sigma_n$ simply by $x_0$. For $\{i_1<\cdots<i_k\}\subset\{1,\ldots,n\}$, we inductively define
	\[
	a+_{\{i_1,\ldots,i_k\}}b=(a+_{\{i_1,\ldots,i_{k-1}\}}x_0)+_{i_k}(x_0+_{\{i_1,\ldots,i_{k-1}\}}b).
	\]
	Then $a\cdot b=a+_{\{1,\ldots,n\}}b$. Note that
	\[
	((b\gamma_{j,1})+_j(b\gamma_{j,0}))\partial_{j+1,\alpha}=
	\begin{cases}
		x_0+_jb&(\alpha=0)\\
		b+_jx_0&(\alpha=1).
	\end{cases}
	\]
	Suppose $j>i$. Then we may define
	\[
	z=((a+_jx_0)\sigma_{j+1})+_{\{1,\ldots,j-1\}}((b\gamma_{j,1})+_j(b\gamma_{j,0})).
	\]
	By the interchange law \eqref{interchange law}, $z\partial_{j+1,\alpha}\simeq_0a+_{\{1,\ldots,j-\alpha\}}b$. Hence $a+_{\{1,\ldots,j-1\}}b\simeq a+_{\{1,\ldots,j\}}b$. Suppose $j<i$. We may also define
	\[
	w=((a+_{\{1,\ldots,j\}}x_0)\sigma_{j+1})+_{i+1}(x_0+_{\{1,\ldots,j-1\}}((b\gamma_{j,1})+_j(b\gamma_{j,0}))).
	\]
	By the interchange law, $w\partial_{j+1,\alpha}\simeq_0a+_{\{1,\ldots,j-\alpha,i\}}b$. Hence $a+_{\{1,\ldots,j-1,i\}}b\simeq a+_{\{1,\ldots,j,i\}}b$. Consequently,
	\[
	x\cdot y=x+_{\{1,\ldots,n\}}y\simeq a+_ib.
	\]
	Therefore the claim follows.
\end{proof}

We show that the multiplication above admits an inverse. Let 
\[
x\colon\bigl(I_{s_1}\otimes\cdots\otimes I_{s_n},\partial(I_{s_1}\otimes\cdots\otimes I_{s_n})\bigr)\to(X,x_0)
\]
be a spherical $n$-grid of size $(l_1,\ldots,l_n)$. Define an $n$-grid $\bar{x}$ by
\[
\begin{gathered}\bar{x}\colon\bigl(I_{s_1}\otimes\cdots\otimes I_{s_{n-1}}\otimes I_{\bar{s}_n},\partial(I_{s_1}\otimes\cdots\otimes I_{s_{n-1}}\otimes I_{\bar{s}_n})\bigr)\to(X,x_0),\\\bar{x}(i_1,\ldots,i_n)=x(i_1,\ldots,i_{n-1},l_n+1-i_n),\end{gathered}
\]
where
\begin{equation}
	\label{s^1}
	\bar{s}_n\colon\{1,\ldots,l_n\}\to\{0,1\},\quad i\mapsto 1-s_n(l_n+1-i).
\end{equation}

\begin{lemma}
	\label{grid inverse}
	For every spherical $n$-grid $x$, we have
	\[
	x\cdot\bar{x}\simeq x_0\simeq \bar{x}\cdot x.
	\]
\end{lemma}

\begin{proof}
	Consider an $(n+1)$-grid $y=(x{\gamma_{n,1}})+_{n+1}(\overline{x{\gamma_{n,1}}})$. By definition,
	\[
	y\partial_{i,\alpha}=\begin{cases}
		x+_n\bar{x}&((i,\alpha)=(n,1))\\
		x_0&\text{(otherwise)},
	\end{cases}
	\]
	which implies that $x\cdot\bar{x}\simeq x_0$ by Lemma \ref{grid slide}. Since $\overline{\bar{x}}=x$, we also obtain $\bar{x}\cdot x\simeq x_0$.
\end{proof}

\begin{proposition}
	Let $(X,x_0)$ be a pointed cubical set. Then for $n\ge 1$,
	\[
	\pi_n^\delta(X,x_0)\times\pi_n^\delta(X,x_0)\to\pi_n^\delta(X,x_0),\quad([x],[y])\mapsto[x\cdot y]
	\]
	is a well-defined group structure.
\end{proposition}

\begin{proof}
	By Lemma \ref{grid product}, the multiplication is well-defined. By definition, the multiplication is associative, and by Lemma \ref{grid inverse}, admits an inverse.
\end{proof}

Let $f\colon(X,x_0)\to(Y,y_0)$ be a pointed cubical map between pointed cubical sets $(X,x_0)$ and $(Y,y_0)$. Let $x$ and $y$ be spherical $n$-grids in $(X,x_0)$. Then $f(x)$ is a spherical $n$-grid in $(Y,y_0)$. Clearly, if spherical $n$-grids $x$ and $y$ in $(X,x_0)$ are grid homotopic, then $f(x)$ and $f(y)$ are also grid homotopic. Moreover, one has $f(x\cdot y)=f(x)\cdot f(y)$. Hence the map
\[
f_*\colon\pi_n^\delta(X,x_0)\to\pi_n^\delta(Y,y_0),\quad[x]\mapsto[f(x)]
\]
is a well-defined homomorphism. By definition, the induced homomorphisms on discrete homotopy groups satisfy functoriality. 

\begin{proposition}
	\label{homotopy invariance}
	Let $f,g\colon(X,x_0)\to(Y,y_0)$ be pointed cubical maps between pointed cubical sets. If there exists a pointed elementary homotopy from $f$ to $g$, then
	\[
	f_*=g_*\colon\pi_*^\delta(X,x_0)\to\pi_*^\delta(Y,y_0).
	\]
\end{proposition}

\begin{proof}
	Let $h\colon(X\otimes\Box^1,x_0\otimes\Box^1)\to(Y,y_0)$ be a pointed elementary homotopy from $f$ to $g$. For every spherical $n$-grid $x$ in $(X,x_0)$,
	\[
	f\circ x\overset{n+1}{\Leftrightarrow}g\circ x
	\]
	as spherical grids. Hence $[f\circ x]=[g\circ x]$ in $\pi_n^\delta(Y,y_0)$, and therefore the claim follows.
\end{proof}

\begin{corollary}
	\label{dhg under he}
	Let $(X,x_0)$, $(Y,y_0)$ be pointed cubical sets. If $(X,x_0)$ and $(Y,y_0)$ have the same pointed homotopy type, then there is an isomorphism
	\[
	\pi_n^\delta(X,x_0)\cong\pi_n^\delta(Y,y_0).
	\]
\end{corollary}

\begin{proposition}
	Let $(X,x_0)$ be a pointed cubical set. For $n\ge 2$, $\pi_n^\delta(X,x_0)$ is an abelian group.
\end{proposition}

\begin{proof}
	Let $x$ and $y$ be spherical $n$-grids in $(X,x_0)$, and let $a$ and $b$ be as in Lemma \ref{grid slide}. Since $n\ge 2$,
	\[
	x\cdot y\simeq a+_nb=\overline{\bar{b}+_n\bar{a}}\simeq \overline{\bar{b}+_1\bar{a}}\simeq b+_1a\simeq y\cdot x.
	\]
	Hence, for $n\ge 2$, $\pi_n^\delta(X,x_0)$ is an abelian group.
\end{proof}

\section{Quasisymmetric cubical sets}\label{Quasisymmetric cubical sets}

In this section, we introduce the notion of quasisymmetric cubical sets, which are cubical sets equipped with coordinate permutation symmetries that are required to be compatible with faces and degeneracies, but not necessarily with connections. We also introduce signed quasisymmetric cubical sets which are quasisymmetric cubical sets additionally equipped with symmetries by reversing cube coordinates. We then give a purely combinatorial construction of a left adjoint of the forgetful functor from the category of quasisymmetric cubical sets to $\cSet$, and show that the unit of this adjunction is a natural weak equivalence.

\subsection{Symmetric and signed symmetric cubical sets}

Let $\Sigma_n$ denote the $n$-th symmetric group for $n\ge 1$, and set $\Sigma_0=\{1\}$. 
We define the following maps:
\begin{itemize}
	\item (permutations) $\kappa\colon[1]^n\to[1]^n$ for $\kappa\in\Sigma_n$ is given by
	\[
	\kappa(x_1,\ldots,x_n)=(x_{\kappa(1)},\ldots,x_{\kappa(n)}).
	\]
\end{itemize}
For $\kappa,\lambda\in\Sigma_n$, let
\[
(\kappa\lambda)(x_1,\ldots,x_n)=\kappa(\lambda(x_1,\ldots,x_n)).
\]
Then we have
\[
(\kappa\lambda)(i)=\lambda(\kappa(i))
\]
for $i=1,\ldots,n$. The \emph{symmetric cubical category} $\Box_\mathsf{sym}$ is defined by adjoining the maps $\kappa$, for all $\kappa\in\Sigma_n$ and all $n\ge 0$, to the cubical category $\Box$. The symmetric cubical category was introduced by Grandis and Mauri \cite{GM1}, who referred to it as the extended cubical category. Since we will consider a further extension of the cubical category, we adopt the terminology symmetric cubical category for $\Box_\mathsf{sym}$.

For $i=1,\ldots,n$ and $\kappa\in\Sigma_n$, define $\kappa^{(i)}\in\Sigma_{n-1}$ by
\[
\kappa^{(i)}(j)=
\begin{cases}
	\kappa(j)&(j<i,\,\kappa(j)<\kappa(i))\\
	\kappa(j)-1&(j<i,\,\kappa(j)>\kappa(i))\\
	\kappa(j+1)&(j\ge i,\,\kappa(j+1)<\kappa(i))\\
	\kappa(j+1)-1&(j\ge i,\,\kappa(j+1)>\kappa(i)).
\end{cases}
\]
Namely, $\kappa^{(i)}$ is the restriction of $\kappa$ to the map $\{1,\ldots,n\}-\{i\}\to\{1,\ldots,n\}-\{\kappa(i)\}$ together with the identification $\{1<\ldots<n\}-\{k\}=\{1<\ldots<n-1\}$ as chains. Note that for $\kappa,\lambda\in\Sigma_n$ and $i=1,\ldots,n$,
\begin{equation}
	\label{permutation deletion}
	(\kappa\lambda)^{(i)}=\kappa^{(i)}\lambda^{(\kappa(i))}.
\end{equation}
For $\kappa\in\Sigma_n$ and $p,q\in\{1,\ldots,n+1\}$, define $\kappa_{(p,q)}\in\Sigma_{n+1}$ by
\[
\kappa_{(p,q)}(p)=q\quad\text{and}\quad(\kappa_{(p,q)})^{(p)}=\kappa.
\]
It is straightforward to verify the identities
\begin{equation}
	\label{permutation identity 1}
	\kappa\partial_{\kappa(i),\alpha}=\partial_{i,\alpha}\kappa^{(i)},\quad\kappa^{(i)}\sigma_{\kappa(i)}=\sigma_i\kappa
\end{equation}
and
\begin{equation}
	\label{permutation identity 2}
	\kappa\gamma_{\kappa(i),\alpha}=\gamma_{i,\alpha}\kappa_{(i+1,\kappa(i)+1)},\quad\gamma_{i,\alpha}(i\,i+1)=\gamma_{i,\alpha}
\end{equation}
hold, where $(i\;i+1)$ denotes the transposition of $i$ and $i+1$. Alternatively, we may write \eqref{permutation identity 2} as the following identities
\begin{equation}
	\label{permutation identity 2 alternative}
	\gamma_{i,\alpha}\kappa_{(i+1,\kappa(i)+1)}=\kappa\gamma_{\kappa(i),\alpha}=\gamma_{i,\alpha}\kappa_{(i+1,\kappa(i))}.
\end{equation}
Analogously to the cubical category $\Box$, the symmetric cubical category $\Box_\mathsf{sym}$ may be regarded as the category whose objects are $[1]^n$ for $n\ge 0$ and whose morphisms are freely generated by faces, degeneracies, connections and permutations, subject to the cubical identities together with the identities \eqref{permutation identity 1} and \eqref{permutation identity 2}.

\begin{definition}
	The category $\SymcSet$ is the functor category $\mathsf{Set}^{\Box_\mathsf{sym}^\mathrm{op}}$. The objects and morphisms of $\SymcSet$ are referred to as \emph{symmetric cubical sets} and \emph{symmetric cubical maps}.
\end{definition}

\begin{example}
	The singular functor $\mathrm{Sing}\colon\mathsf{Top}\to\cSet$ factors through $\SymcSet$.
\end{example}

Isaacson \cite{I} showed that there is a model category structure on $\SymcSet$ such that the forgetful functor $\SymcSet\to\cSet$ is a Quillen equivalence, where $\cSet$ is equipped with the model category structure of Theorem \ref{cSet model structure}. Consequently, symmetric cubical sets also provide a model for the homotopy theory of spaces. An analogue of this result for cubical sets without connections was proved in \cite[Theorem 5.9]{Ar}.

Recall that the $n$-th \emph{signed symmetric group} $B_n$ is the group of signed permutations on $\{\pm 1,\ldots,\pm n\}$. Namely, $B_n$ is the wreath product $\Z_2\wr\Sigma_n$. Hence, $B_n=(\Z_2)^n\times\Sigma_n$ as a set such that the product is given by
\[
(a,\kappa)(b,\lambda)=(a+\kappa(b),\kappa\lambda)
\]
for $(a,\kappa),(b,\lambda)\in B_n$, where $\Sigma_n$ permutes coordinates of $(\Z_2)^n$. Set $\Z_2=\{0,1\}$, and for $a\in(\Z_2)^n$, write $a=(a_1,\ldots,a_n)$. We define the following map:
\begin{itemize}
	\item (signed permutations) $(a,\kappa)\colon[1]^n\to[1]^n$ for $(a,\kappa)\in B_n$ is given by
	\[
	(a,\kappa)(x_1,\ldots,x_n)=(x_{\kappa(1)}^{a_1},\ldots,x_{\kappa(n)}^{a_n})
	\]
	where for $x\in[1]$, we set $x^0=x$ and $x^1=1-x$.
\end{itemize}
Define the \emph{signed symmetric cubical category} $\Box_\mathsf{sym}^\pm$ as the category obtained from $\Box_\mathsf{sym}$ by adjoining signed permutations.

For $(a,\kappa)\in B_n$ and $i=1,\ldots,n$, define 
\[
(a,\kappa)^{(i)}=(a^{(i)},\kappa^{(i)})\in B_{n-1},
\]
where
\[
a^{(i)}=(a_1,\ldots,a_{i-1},a_{i+1},\ldots,a_n)\in(\Z_2)^{n-1}.
\]
Analogously to \eqref{permutation deletion}, for $(a,\kappa),(b,\lambda)\in B_n$ and $i=1,\ldots,n$, we have
\[
((a,\kappa)(b,\lambda))^{(i)}=(a,\kappa)^{(i)}(b,\lambda)^{(\kappa(i))}.
\]
For $(a,\kappa)\in B_n$, $p,q\in\{1,\ldots,n+1\}$ and $\alpha=0,1$, define $(a,\kappa)_{(p,q),\alpha}\in B_{n+1}$ by
\[
((a,\kappa)_{(p,q),\alpha})(p)=(-1)^\alpha q\quad\text{and}\quad((a,\kappa)_{(p,q),\alpha})^{(p)}=(a,\kappa).
\]
It is straightforward to verify that the identities
\begin{equation}
	\label{signed permutation identity 1}
	(a,\kappa)\partial_{\kappa(i),\alpha+a_i}=\partial_{i,\alpha}(a,\kappa)^{(i)},\quad(a,\kappa)^{(i)}\sigma_{\kappa(i)}=\sigma_i(a,\kappa)
\end{equation}
and
\begin{equation}
	\label{signed permutation identity 2}
	(a,\kappa)\gamma_{\kappa(i),\alpha+a_i}=\gamma_{i,\alpha}(a,\kappa)_{(i+1,\kappa(i)+1),a_i},\quad\gamma_{i,\alpha}(0,(i\,i+1))=\gamma_{i,\alpha}.
\end{equation}
Analogously to the symmetric cubical category $\Box_\mathsf{sym}$, the signed symmetric cubical category $\Box_\mathsf{sym}^\pm$ may be regarded as the category whose objects are $[1]^n$ for $n\ge0$ and whose morphisms are freely generated by faces, degeneracies, connections and signed permutations, subject to the cubical identities together with the identities \eqref{signed permutation identity 1} and \eqref{signed permutation identity 2}.

\begin{definition}
	The category $\SymcSet^\pm$ is the functor category $\mathsf{Set}^{(\Box_\mathsf{sym}^\pm)^\mathrm{op}}$. The objects and morphisms of $\SymcSet^\pm$ are referred to as \emph{signed symmetric cubical sets} and \emph{signed symmetric cubical maps}.
\end{definition}

\subsection{Quasisymmetric and signed quasisymmetric cubical sets}

We define the quasisymmetric cubical category.

\begin{definition}
	The \emph{quasisymmetric cubical category} $\Box_\mathsf{qsym}$ is defined as the category whose objects are $[1]^n$ for $n\ge 0$ and whose morphisms are freely generated by faces, degeneracies, connections and permutations subject to the cubical identities together with the identity \eqref{permutation identity 1}.
\end{definition}

Note that we are not imposing the identities \eqref{permutation identity 2} for the quasisymmetric cubical category. By definition, there is a sequence of canonical functors
\[
\Box\to\Box_\mathsf{qsym}\to\Box_\mathsf{sym}.
\]

\begin{definition}
	The category $\QSymcSet$ is the functor category $\mathsf{Set}^{\Box_\mathsf{qsym}^\mathrm{op}}$, with forgetful functor $\mathcal{U}\colon\QSymcSet\to\cSet$. We call a cubical set $X$ \emph{quasisymmetric} when $X=\mathcal{U}Y$ for a chosen object $Y\in\QSymcSet$. The permutation operators on $X$ are those supplied by $Y$. A \emph{quasisymmetric cubical map} is the underlying cubical map of a morphism in $\QSymcSet$. Thus $X$ denotes the underlying cubical set, whereas $Y$ denotes the object carrying its chosen quasisymmetric structure.
\end{definition}




\begin{definition}
	The \emph{signed quasisymmetric cubical category} $\Box_\mathsf{qsym}^\pm$ is defined as the category whose objects are $[1]^n$ for $n\ge 0$ and whose morphisms are freely generated by faces, degeneracies, connections and signed permutations subject to the cubical identities together with the identity \eqref{signed permutation identity 1}.
\end{definition}

Note that we are not imposing the identities \eqref{signed permutation identity 2} for $\Box_\mathsf{qsym}^\pm$. By definition, there is a canonical functor
\[
\Box_\mathsf{qsym}^\pm\to\Box_\mathsf{sym}^\pm.
\]

\begin{definition}
	The category $\QSymcSet^\pm$ is the functor category $\mathsf{Set}^{(\Box_\mathsf{qsym}^\pm)^\mathrm{op}}$, with forgetful functor $\mathcal{U}^\pm\colon\QSymcSet^\pm\to\cSet$. We call a cubical set $X$ \emph{signed quasisymmetric} when $X=\mathcal{U}^\pm Y$ for a chosen object $Y\in\QSymcSet^\pm$. The signed permutation operators on $X$ are those supplied by $Y$, and a \emph{signed quasisymmetric cubical map} is the underlying cubical map of a morphism in $\QSymcSet^\pm$.
\end{definition}
\begin{definition}
	Let $X$ be a signed quasisymmetric cubical set. For an $n$-grid
	\[
	x\colon I_{s_1}\otimes\cdots\otimes I_{s_n}\to  X
	\]
	and $(a,\kappa)\in B_n$, we define an $n$-grid
	\[
	x(a,\kappa)\colon I_{s^{a_{\kappa^{-1}(1)}}_{\kappa^{-1}(1)}}\otimes\cdots\otimes I_{s^{a_{\kappa^{-1}(n)}}_{\kappa^{-1}(n)}}\to X
	\]
	such that
	\[
	x(i_{\kappa(1)}^{a_1},\ldots,i_{\kappa(n)}^{a_n})(a,\kappa)=(x(a,\kappa))(i_1,\ldots,i_n),
	\]
	where for $j=1,\ldots,n$, 
	\[
	i_{\kappa(j)}^{a_j}=\begin{cases}
		i_{\kappa(j)}&(a_j=0)\\
		l_{j}+1-i_{\kappa(j)}&(a_j=1),
	\end{cases}\quad
	s_j^0=s_j,\quad s_j^1(i)=s_j(l_j+1-i).
	\]
\end{definition}
\begin{proposition}
	\label{rotate and reflect grid}
	For a signed quasisymmetric cubical set $X$, the actions of signed symmetric groups on grids in $X$ are well-defined and satisfy \eqref{signed permutation identity 1}.
\end{proposition}
In the same manner we can obtain the following proposition.
\begin{proposition}
	\label{rotate grid}
	For a quasisymmetric cubical set $X$, the actions of symmetric groups on grids in $X$ are well-defined and satisfy \eqref{permutation identity 1}.
\end{proposition}
We will see in Remark \ref{quasisymm vs symm} that only the identities \eqref{signed permutation identity 1} are necessary for the Kan condition, which is why we work with (signed) quasisymmetric cubical sets rather than (signed) symmetric ones.

\subsection{The functor \texorpdfstring{$\mathcal{S}$}{S}}

We give a combinatorial construction of a left adjoint $\mathcal{S}\colon\cSet\to\QSymcSet$ of $\mathcal{U}$, and prove that its unit is a natural weak equivalence.

\begin{definition}
	For a cubical set $X$, define
	\[
	(\mathcal{S}X)_n=
	\left(\coprod_{m\ge 0}X_m\times
	\Box_\mathsf{qsym}([1]^n,[1]^m)\right)/\sim,
	\]
	where the equivalence relation is generated by
	\begin{equation}
		\label{identification S}
		(x\theta,u)\sim(x,\theta u)
	\end{equation}
	for $x\in X_m$, $\theta\colon[1]^p\to[1]^m$ in $\Box$, and $u\colon[1]^n\to[1]^p$ in $\Box_\mathsf{qsym}$. Write $[x,u]$ for the class of $(x,u)$. For $v\colon[1]^q\to[1]^n$ in $\Box_\mathsf{qsym}$, set
	\[
	[x,u]v=[x,uv].
	\]
	For a cubical map $f\colon X\to X'$, define
	\[
	\mathcal{S}f([x,u])=[f(x),u].
	\]
\end{definition}

Here $u$ is represented by a finite composable word in faces, degeneracies, connections and permutations, subject to the defining relations of $\Box_\mathsf{qsym}$. In \eqref{identification S}, it is enough to take $\theta$ to be a single face, degeneracy or connection. Thus the construction is given by words and relations; it does not impose the permutation--connection identities \eqref{permutation identity 2}.

\begin{proposition}
	\label{S adjunction}
	The construction above defines a functor $\mathcal{S}\colon\cSet\to\QSymcSet$ left adjoint to $\mathcal{U}$. Its unit is
	\[
	\mu_X\colon X\to\mathcal{US}X,\quad
	x\mapsto[x,1_{[1]^n}]\quad(x\in X_n).
	\]
\end{proposition}
\begin{proof}
	The action of $v$ respects \eqref{identification S}, since
	\[
	[x\theta,uv]=[x,\theta uv].
	\]
	Composition in $\Box_\mathsf{qsym}$ gives all the required identities. Likewise,
	\[
	[f(x\theta),u]=[f(x)\theta,u]=[f(x),\theta u],
	\]
	so $\mathcal{S}f$ is well-defined. For $\theta\colon[1]^p\to[1]^n$ in $\Box$,
	\[
	\mu_X(x\theta)=[x\theta,1_{[1]^p}]
	=[x,\theta]=\mu_X(x)\theta.
	\]
	Hence $\mu_X$ is a cubical map, natural in $X$.
	
	Let $Y\in\QSymcSet$ and let $f\colon X\to\mathcal{U}Y$ be a cubical map. Define
	\[
	\widehat{f}\colon\mathcal{S}X\to Y,\quad
	\widehat{f}[x,u]=f(x)u.
	\]
	Since
	\[
	f(x\theta)u=(f(x)\theta)u=f(x)(\theta u),
	\]
	this is well-defined. Moreover,
	\[
	\widehat{f}([x,u]v)=f(x)uv=\widehat{f}([x,u])v,
	\quad
	\mathcal{U}\widehat{f}\circ\mu_X=f.
	\]
	If $g\colon\mathcal{S}X\to Y$ satisfies $\mathcal{U}g\circ\mu_X=f$, then
	\[
	g[x,u]=g([x,1]u)=g[x,1]u=f(x)u.
	\]
	Thus $g=\widehat{f}$, and we obtain the natural bijection
	\[
	\mathrm{Hom}_{\QSymcSet}(\mathcal{S}X,Y)
	\cong\mathrm{Hom}_{\cSet}(X,\mathcal{U}Y).
	\]
\end{proof}

To prove that $\mu$ is a natural weak equivalence, we first consider representable cubical sets. There is a natural isomorphism
\[
\mathcal{S}\Box^m\cong\Box_\mathsf{qsym}(-,[1]^m),
\quad [\theta,u]\mapsto\theta u,
\]
whose inverse sends $u$ to $[1_{[1]^m},u]$. Put
\[
Q^m=\mathcal{U}\Box_\mathsf{qsym}(-,[1]^m).
\]
Under this identification, $\mu_{\Box^m}$ sends a morphism in $\Box$ to the same morphism in $\Box_\mathsf{qsym}$.

\begin{lemma}
	\label{S t=0,1}
	For every $m\ge 0$, the cubical set $Q^m$ is cubically homotopy equivalent to $\Box^0$. Consequently, $\mu_{\Box^m}\colon\Box^m\to Q^m$ is a weak equivalence.
\end{lemma}
\begin{proof}
	Define a functor $T\colon\Box_\mathsf{qsym}\to\Box_\mathsf{qsym}$ by
	\[
	T([1]^n)=[1]^{n+1},\quad
	T(\partial_{i,\alpha})=\partial_{i,\alpha},\quad
	T(\sigma_i)=\sigma_i,\quad
	T(\gamma_{i,\alpha})=\gamma_{i,\alpha},
	\]
	and
	\[
	T(\kappa)=\kappa_{(n+1,n+1)}
	\]
	for $\kappa\in\Sigma_n$. The morphisms on the right have both source and target dimensions increased by one. The cubical identities remain valid with this additional last coordinate, and the identities in \eqref{permutation identity 1} are preserved since
	\[
	\bigl(\kappa_{(n+1,n+1)}\bigr)^{(i)}
	=\bigl(\kappa^{(i)}\bigr)_{(n,n)}
	\quad(1\le i\le n).
	\]
	Thus $T$ is well-defined. For every $u\colon[1]^n\to[1]^m$ in $\Box_\mathsf{qsym}$, we have
	\begin{equation}
		\label{S last face}
		T(u)\partial_{n+1,\alpha}
		=\partial_{m+1,\alpha}u
		\quad(\alpha=0,1)
	\end{equation}
	by the cubical identities and \eqref{permutation identity 1}. For $m\ge 1$, define cubical maps
	\[
	j_m\colon Q^{m-1}\to Q^m,\quad u\mapsto\partial_{m,0}u,
	\quad
	r_m\colon Q^m\to Q^{m-1},\quad u\mapsto\sigma_m u.
	\]
	Then we have 
	\[r_mj_m=1.\] 
	We define a cubical map
	\[
	H_m\colon Q^m\otimes\Box^1\to Q^m,\quad u\otimes\Box^1\mapsto\gamma_{m,1}T(u).
	\]
	Since
	\[
	H_m(u\theta)=\gamma_{m,1}T(u)T(\theta)
	=H_m(u)(\theta\otimes 1_{[1]})
	\]
	for $u\in Q^m_n$ and $\theta\colon[1]^p\to[1]^n$ in $\Box$, $H_m$ is well-defined. Moreover, we have
	\[
	H_m(1\otimes\partial_{1,\alpha})=\begin{cases}
		j_mr_m&(\alpha=0)\\
		1_{Q^m}&(\alpha=1),
	\end{cases}
	\]
	implying that $j_m\colon Q^{m-1}\to Q^m$ is a cubical homotopy equivalence with homotopy inverse $r_m$. Since $Q^0=\Box^0$, and the unique map $*\colon\Box^m\to\Box^0$ is a cubical homotopy equivalence, the proof is finished.
\end{proof}

We next check that applying $\mathcal{US}$ to a boundary inclusion gives an injective cubical map.

\begin{lemma}
	\label{S gluing}
	For every $m\ge 0$ and the inclusion $\mathrm{incl}\colon\partial\Box^m\to\Box^m$, the map
	\[
	\mathcal{US}(\mathrm{incl})\colon\mathcal{US}(\partial\Box^m)\to\mathcal{US}(\Box^m)=Q^m
	\]
	is injective.
\end{lemma}

\begin{proof}
	The case $m=0$ is clear. Suppose $m\ge 1$. Under the identification
	\[
	\mathcal{US}(\Box^m)_n
	\cong\Box_\mathsf{qsym}([1]^n,[1]^m),
	\]
	we have
	\[
	\mathcal{US}(\mathrm{incl})([\theta,u])=\theta u.
	\]
	For $\theta\in(\partial\Box^m)_p$, there are $i=1,\ldots,m$, $\alpha=0,1$ and $\theta'\colon[1]^p\to[1]^{m-1}$ in $\Box$ such that
	\[
	\theta=\partial_{i,\alpha}\theta',
	\]
	which implies
	\[
	[\theta,u]=[\partial_{i,\alpha},\theta'u]
	\]
	by \eqref{identification S}. Thus every element of $\mathcal{US}(\partial\Box^m)_n$
	has a representative $[\partial_{i,\alpha},u]$. Suppose that 
	\[
	\partial_{i,\alpha}u=\mathcal{US}(\mathrm{incl})([\partial_{i,\alpha},u])=\mathcal{US}(\mathrm{incl})([\partial_{j,\beta},v])=\partial_{j,\beta}v
	\]
	in $\Box_\mathsf{qsym}$, and we aim to prove that
	\[
	[\partial_{i,\alpha},u]=[\partial_{j,\beta},v]
	\]
	in $\mathcal{US}(\partial\Box^m)$. If $i=j$, applying the canonical
	functor $\Box_\mathsf{qsym}\to\Box_\mathsf{sym}$ gives
	$\alpha=\beta$, since the $i$-th coordinates of the two
	maps are constant. We also have
	\[
	u=\sigma_i\partial_{i,\alpha}u=\sigma_i\partial_{i,\alpha}v=v.
	\]
	For the remaining case, it suffices to assume that $i<j$. Then we have
	\begin{equation}
		\label{v formula}
		v=\sigma_j\partial_{j,\beta}v=\sigma_j\partial_{i,\alpha}u=\sigma_j\partial_{i,\alpha}\sigma_i\partial_{i,\alpha}u=\sigma_j\partial_{i,\alpha}\sigma_i\partial_{j,\beta}v=\sigma_j\partial_{j,\beta}\partial_{i,\alpha}\sigma_iv=\partial_{i,\alpha}\sigma_iv.
	\end{equation}
	We have
	\begin{align*}
		[\partial_{i,\alpha},u]
		&=[\partial_{i,\alpha},\sigma_i\partial_{i,\alpha}u]\\
		&=[\partial_{i,\alpha},\sigma_i\partial_{j,\beta}v]\\
		&=[\partial_{i,\alpha},\partial_{j-1,\beta}\sigma_iv]\\
		&=[\partial_{i,\alpha}\partial_{j-1,\beta},\sigma_iv]\\
		&=[\partial_{j,\beta}\partial_{i,\alpha},\sigma_iv]\\
		&=[\partial_{j,\beta},\partial_{i,\alpha}\sigma_iv]\\
		&=[\partial_{j,\beta},v]
	\end{align*}
	by \eqref{identification S} and \eqref{v formula}, finishing the proof.
\end{proof}
We are ready to prove Theorem \ref{main 2}.

\begin{proof}[Proof of Theorem \ref{main 2}]
	For a cubical set $X$, let $X^{(m)}$ be the cubical subset generated by all cubes of dimensions at most $m$, and set $X^{(-1)}=\varnothing$. Let $E_m(X)$ be the set of $m$-cubes which are neither degenerate nor connections. Then we have a pushout diagram
	\begin{equation}
		\label{S skeletal pushout}
		\xymatrix{
			\displaystyle\coprod_{x\in E_m(X)}\partial\Box^m
			\ar[r]\ar@{^{(}->}[d]&X^{(m-1)}\ar[d]\\
			\displaystyle\coprod_{x\in E_m(X)}\Box^m
			\ar[r]&X^{(m)}
		}
	\end{equation}
	by the cubical Eilenberg--Zilber lemma \cite[Proposition 3.6 and Theorem 7.9]{Ca}.
	
	
	We prove by induction on $m$ that $\mu_Z$ is a weak equivalence whenever $Z$ is generated by cubes of dimensions at most $m$. The case $m=-1$ is clear. Suppose the assertion holds for $m-1$, and let $Z=Z^{(m)}$. Then we have a commutative diagram of spaces
	\begingroup
	\small
	\[
	\xymatrix@!0@C=3cm@R=1.8cm{
		\displaystyle\coprod_{x\in E_m(Z)}|\partial\Box^m|
		\ar[rr]
		\ar@{^{(}->}[dr]
		\ar[dd]_{\coprod|\mu_{\partial\Box^m}|}
		&&
		|Z^{(m-1)}|
		\ar[dr]
		\ar@{-->}[dd]^{|\mu_{Z^{(m-1)}}|}
		&
		\\
		&
		\displaystyle\coprod_{x\in E_m(Z)}|\Box^m|
		\ar[rr]
		\ar[dd]_(.45){\coprod|\mu_{\Box^m}|}
		&&
		|Z|
		\ar[dd]^{|\mu_{Z}|}
		\\
		\displaystyle\coprod_{x\in E_m(Z)}
		|\mathcal{US}(\partial\Box^m)|
		\ar@{-->}[rr]
		\ar@{^{(}->}[dr]
		&&
		|\mathcal{US}(Z^{(m-1)})|
		\ar@{-->}[dr]
		&
		\\
		&
		\displaystyle\coprod_{x\in E_m(Z)}
		|\mathcal{US}(\Box^m)|
		\ar[rr]
		&&
		|\mathcal{US}(Z)|,
	}
	\]
	\endgroup
	where the top and the bottom squares are homotopy pushout diagrams by Lemma \ref{S gluing}. The maps
	\[
	|\mu_{Z^{(m-1)}}|\quad\text{and}\quad
	\coprod_{x\in E_m(Z)}|\mu_{\partial\Box^m}|
	\]
	are homotopy equivalences by the induction hypothesis, while
	\[
	\coprod_{x\in E_m(Z)}|\mu_{\Box^m}|
	\]
	is a homotopy equivalence by Lemma \ref{S t=0,1}. This implies that $|\mu_Z|$ is a homotopy equivalence and a cofibration. Since
	\[
	|X|=\underset{m}{\mathrm{colim}}\,|X^{(m)}|\text{ and }
	|\mathcal{US}X|=\underset{m}{\mathrm{colim}}\,
	|\mathcal{US}(X^{(m)})|,
	\]
	$|\mu_X|$ is a homotopy equivalence for any cubical set $X$, completing the proof.
\end{proof}
Since $\mathcal{U}$ is computed degreewise and $\mathcal{S}$ is a left adjoint, $\mathcal{US}$ preserves colimits. By Lemma \ref{S gluing}, the functor $\mathcal{US}$ also preserves cofibrations. Combining with Theorem \ref{main 2}, we have the following corollary by \cite[Theorem 11.3.2]{H}, which implies that quasisymmetric cubical sets are also combinatorial models for the homotopy theory of spaces.
\begin{corollary}
	\label{QSymcSet model structure}
	The category $\QSymcSet$ carries a cofibrantly generated
	model category structure right-transferred along
	$\mathcal{U}\colon\QSymcSet\to\cSet$ from the model category
	structure of Theorem \ref{cSet model structure}.
	Explicitly, a map $f$ in $\QSymcSet$ is a weak equivalence
	or a fibration if and only if $\mathcal{U}f$ is,
	respectively, a weak equivalence or a fibration in $\cSet$.
	Moreover, the adjunction
	\[
	\mathcal{S}\colon\cSet
	\rightleftarrows\QSymcSet\colon\mathcal{U}
	\]
	is a Quillen equivalence.
\end{corollary}

\section{Fibrant replacement}\label{Fibrant replacement}

In this section, we explicitly construct a fibrant replacement for a quasisymmetric cubical set, which is the key ingredient in proving Theorems \ref{main 1} and \ref{main 3}.

\subsection{The functor \texorpdfstring{$\mathcal{R}$}{R}}

To construct a fibrant replacement, we first need a mechanism to formally invert cubes. We define a functor $\mathcal{R}\colon\cSet\to\cSet$ by reversing cube coordinates. Reversing cube coordinates plays an important role in several contexts, including cubical $\omega$-groupoids in \cite{BHS}.

\begin{definition}
	We define a functor
	\[
	\mathcal{R}\colon\cSet\to\cSet
	\]
	by setting $(\mathcal{R}X)_0=X_0$ and
	\[
	(\mathcal{R}X)_n=\{x^{s_1\cdots s_n}\mid x\in X_n\text{ and }s_1,\ldots,s_n\in\{0,1\}\}
	\]
	for $n\ge 1$, subject to the identification
	\begin{equation}
		\label{R identification}
		(x\sigma_i)^{s_1\cdots s_{i-1}0s_{i+1}\cdots s_n}=(x\sigma_i)^{s_1\cdots s_{i-1}1s_{i+1}\cdots s_n}.
	\end{equation}
	Faces, degeneracies, and connections are defined by
	\begin{align*}
		(x^{s_1\cdots s_n})\partial_{i,\alpha}&=(x\partial_{i,s_i+\alpha})^{s_1\cdots s_{i-1}s_{i+1}\cdots s_n}\\
		(x^{s_1\cdots s_n})\sigma_i&=(x\sigma_i)^{s_1\cdots s_{i-1}0s_{i}\cdots s_n}\\
		(x^{s_1\cdots s_n})\gamma_{i,\alpha}&=(x\gamma_{i,s_i+\alpha})^{s_1\cdots s_is_i\cdots s_n}.
	\end{align*}
\end{definition}


\begin{lemma}
	The functor $\mathcal{R}$ is well-defined.
\end{lemma}

\begin{proof}
	Let $X$ be a cubical set. It suffices to show that $\mathcal{R}X$ is a well-defined cubical set. Let $x\in X_{n-1}$, $i\in\{1,\ldots,n\}$ and $s_1,\ldots,s_{i-1},s_{i+1},\ldots,s_n\in\{0,1\}$. Suppose that $j<i$. Then
	
	\begin{align*}
		(x\sigma_i)^{s_1\cdots s_{i-1}0s_{i+1}\cdots s_n}\partial_{j,\alpha}&=(x\sigma_i\partial_{j,s_j+\alpha})^{s_1\cdots s_{j-1}s_{j+1}\cdots s_{i-1}0s_{i+1}\cdots s_n}\\
		&=(x\partial_{j,s_j+\alpha}\sigma_{i-1})^{s_1\cdots s_{j-1}s_{j+1}\cdots s_{i-1}0s_{i+1}\cdots s_n}\\
		&=(x\partial_{j,s_j+\alpha}\sigma_{i-1})^{s_1\cdots s_{j-1}s_{j+1}\cdots s_{i-1}1s_{i+1}\cdots s_n}\\
		&=(x\sigma_{i}\partial_{j,s_j+\alpha})^{s_1\cdots s_{j-1}s_{j+1}\cdots s_{i-1}1s_{i+1}\cdots s_n}\\
		&=(x\sigma_i)^{s_1\cdots s_{i-1}1s_{i+1}\cdots s_n}\partial_{j,\alpha}.
	\end{align*}
	The remaining cases are verified in the same manner. Hence faces are well-defined. It is analogously proved that degeneracies and connections are well-defined. We now verify the cubical identities. Let $x^{s_1\cdots s_n}\in(\mathcal{R}X)_n$. For $j\le i$,
	\begin{align*}
		(x^{s_1\cdots s_n})\partial_{j,\beta}\partial_{i,\alpha}&=((x\partial_{j,s_j+\beta})^{s_1\cdots s_{j-1}s_{j+1}\cdots s_n})\partial_{i,\alpha}\\
		&=(x\partial_{j,s_j+\beta}\partial_{i,s_{i+1}+\alpha})^{s_1\cdots s_{j-1}s_{j+1}\cdots s_{i}s_{i+2}\cdots s_n}\\
		&=(x\partial_{i+1,s_{i+1}+\alpha}\partial_{j,s_j+\beta})^{s_1\cdots s_{j-1}s_{j+1}\cdots s_{i}s_{i+2}\cdots s_n}\\
		&=((x\partial_{i+1,s_{i+1}+\alpha})^{s_1\cdots s_{i}s_{i+2}\cdots s_n})\partial_{j,\beta}\\
		&=(x^{s_1\cdots s_n})\partial_{i+1,\alpha}\partial_{j,\beta}.
	\end{align*}
	One also has
	\begin{align*}
		(x^{s_1\cdots s_n})\gamma_{i,\alpha}\partial_{i,1-\alpha}&=((x\gamma_{i,s_i+\alpha})^{s_1\cdots s_is_is_{i+1}\cdots s_n})\partial_{i,1-\alpha}\\
		&=(x\gamma_{i,s_i+\alpha}\partial_{i,s_i+\alpha+1})^{s_1\cdots s_n}\\
		&=(x\partial_{i,s_i+\alpha+1}\sigma_i)^{s_1\cdots s_n}\\
		&=((x\partial_{i,s_i+\alpha+1})^{s_1\cdots s_{i-1}s_{i+1}\cdots s_n})\sigma_i\\
		&=x^{s_1\cdots s_n}\partial_{i,1-\alpha}\sigma_i,
	\end{align*}
	where we use the identification \eqref{R identification} for the fourth equality. The remaining cubical identities are verified analogously, and therefore the claim follows.
\end{proof}

Define the natural transformation $\iota\colon 1_{\cSet}\Rightarrow\mathcal{R}$ by
\[
\iota_X\colon X\to\mathcal{R}X,\quad x\mapsto x^{0\cdots 0}.
\]
It is straightforward to verify that $\iota$ is well-defined. We show that $\iota$ is a natural weak equivalence. Let $X$ be a cubical set. For $x^{s_1\cdots s_n}\in(\mathcal{R}X)_n$ with $n\ge 1$, let $Q(x^{s_1\cdots s_n})=(x\gamma_{n,1}\cdots\gamma_{1,1})^{0s_10s_2\cdots0s_n}$. We define a map
\[
h(x^{s_1\cdots s_n})\colon|x^{s_1\cdots s_n}|\times[0,1]\to|Q(x^{s_1\cdots s_n})|
\]
by
\begin{align*}
	&h(x^{s_1\cdots s_n})(x_1,\ldots,x_n,t)\\
	&=
	\begin{cases}
		(\bar{x}_1,2(x_1+s_1-1)t+1-s_1,\ldots,\bar{x}_n,2(x_n+s_n-1)t+1-s_n)&(0\le t\le\frac{1}{2})\\
		(2(1-\bar{x}_1)t+2\bar{x}_1-1,x_1,\ldots,2(1-\bar{x}_n)t+2\bar{x}_n-1,x_n)&(\frac{1}{2}\le t\le 1)
	\end{cases}
\end{align*}
where $\bar{x}_i=(1-s_i)x_i+s_i(1-x_i)$. For $y\in (\mathcal{R}X)_0=X_0$, we define $h(y)\colon|y|\times[0,1]\to|y|$ as the projection. For every $z\in\mathcal{R}X$, set
\[
\bar{h}(z)=\rho\circ h(z)\colon|z|\times[0,1]\to|\mathcal{R}X|,
\]
where
\[
\rho\colon|Q(x^{s_1\cdots s_n})|\to|\mathcal{R}X|
\]
is the canonical map.

\begin{lemma}
	\label{R t=0,1}
	Let $X$ be a cubical set. For every $x^{s_1\cdots s_n}\in(\mathcal{R}X)_n$, $\bar{h}(x^{s_1\cdots s_n})(-,t)$ coincides with
	\[
	|x^{s_1\cdots s_n}|\xrightarrow{m_{s_1\cdots s_n}}|x^{0\cdots 0}|\xrightarrow{\rho}|\mathcal{R}X|\quad\text{and}\quad|x^{s_1\cdots s_n}|\xrightarrow{\rho}|\mathcal{R}X|
	\]
	for $t=0,1$, respectively, where
	\[
	m_{s_1\cdots s_n}(x_1,\ldots,x_n)=(\bar{x}_1,\ldots,\bar{x}_n).
	\]
\end{lemma}

\begin{proof}
	For $n=0$, the claim is obvious. Assume $n\ge 1$, and let $x^{s_1\cdots s_n}\in(\mathcal{R}X)_n$. By definition,
	\[
	h(x^{s_1\cdots s_n})(x_1,\ldots,x_n,0)=(\bar{x}_1,1-s_1,\ldots,\bar{x}_n,1-s_n)\in|Q(x^{s_1\cdots s_n})|.
	\]
	By the cubical identity,
	$(x\gamma_{n,1}\cdots\gamma_{1,1})^{0s_10s_2\cdots0s_n}\partial_{2,1-s_1}\cdots\partial_{n+1,1-s_n}=x^{0\cdots 0}$. Hence the claim for $t=0$ follows. The claim for $t=1$ is proved analogously.
\end{proof}

\begin{lemma}
	\label{R gluing}
	Let $X$ be a cubical set, and let $z\in(\mathcal{R}X)_n$. For all $i$ and $\alpha$, the diagrams
	\[
	\xymatrix{
		|z\partial_{i,\alpha}|\times[0,1]\ar[d]_{\partial_{i,\alpha}\times 1}\ar[rr]^(.6){\bar{h}(z\partial_{i,\alpha})}&&|\mathcal{R}X|\ar@{=}[d]\\
		|z|\times[0,1]\ar[rr]^(.57){\bar{h}(z)}&&|\mathcal{R}X|
	}
	\qquad
	\xymatrix{
		|z\sigma_i|\times[0,1]\ar[d]_{\sigma_i\times 1}\ar[rr]^(.59){\bar{h}(z\sigma_i)}&&|\mathcal{R}X|\ar@{=}[d]\\
		|z|\times[0,1]\ar[rr]^(.57){\bar{h}(z)}&&|\mathcal{R}X|
	}
	\]
	commute.
\end{lemma}

\begin{proof}
	Let $x^{s_1\cdots s_n}\in(\mathcal{R}X)_n$. By the cubical identity,
	\[
	Q(x^{s_1\cdots s_n}\sigma_i)=Q(x^{s_1\cdots s_n})\sigma_{2i-1}\sigma_{2i}\quad\text{and}\quad Q(x^{s_1\cdots s_n})\partial_{2i,1-s_i}\partial_{2i-1,1}=Q(x^{s_1\cdots s_n}\partial_{i,1-s_i}).
	\]
	Hence the right diagram and the left diagram for $\alpha=1-s_i$ commute. By the cubical identity,
	\[
	Q(x^{s_1\cdots s_n})\partial_{2i-1,0}=Q((x^{s_1\cdots s_n})\partial_{i,s_i})\sigma_{2i-1}.
	\]
	Hence for $\alpha=s_i$ and $0\le t\le\frac{1}{2}$
	\begin{align*}
		&\bar{h}(x^{s_1\cdots s_n})\partial_{i,s_i}(x_1,\ldots,x_{i-1},x_{i+1},\ldots,x_n,t)\\
		&=\bar{h}(x^{s_1\cdots s_n})(x_1,\ldots,x_{i-1},s_i,x_{i+1},\ldots,x_n,t)\\
		&=\rho(\bar{x}_1,\tilde{x}_1,\ldots,\bar{x}_{i-1},\tilde{x}_{i-1},0,2(2s_i-1)t+1-s_i,\bar{x}_{i+1},\tilde{x}_{i+1},\ldots,\bar{x}_n,\tilde{x}_n)\\
		&=\rho(\bar{x}_1,\tilde{x}_1,\ldots,\bar{x}_{i-1},\tilde{x}_{i-1},\bar{x}_{i+1},\tilde{x}_{i+1},\ldots,\bar{x}_n,\tilde{x}_n)\\
		&=\bar{h}((x^{s_1\cdots s_n})\partial_{i,s_i})(x_1,\ldots,x_{i-1},x_{i+1},\ldots,x_n,t),
	\end{align*}
	where $\tilde{x}_k=2(x_k+s_k-1)t+1-s_k$ for $k=1,\ldots,i-1,i+1,\ldots,n$. Therefore the left diagram commutes for $\alpha=s_i$ and $0\le t\le\frac{1}{2}$. The case where $\alpha=s_i$ and $\frac{1}{2}\le t\le 1$ is proved analogously.
\end{proof}

\begin{proposition}
	\label{R}
	The natural transformation $\iota\colon 1_{\cSet}\Rightarrow\mathcal{R}$ is a natural weak equivalence.
\end{proposition}

\begin{proof}
	Let $X$ be a cubical set. By Lemmas \ref{R t=0,1} and \ref{R gluing}, the family of homotopies $\{\bar{h}(z)\}_{z\in\mathcal{R}X}$ assembles to a homotopy $\bar{h}\colon|\mathcal{R}X|\times[0,1]\to|\mathcal{R}X|$ such that $\bar{h}(-,1)=1_{|\mathcal{R}X|}$ and $\bar{h}(-,0)$ factors as
	\[
	|\mathcal{R}X|\xrightarrow{f}|X|\xrightarrow{|\iota_X|}|\mathcal{R}X|
	\]
	such that $f\circ|\iota_X|=1_{|X|}$. Therefore $f$ is the homotopy inverse of $|\iota_X|\colon|X|\to|\mathcal{R}X|$.
\end{proof}

\subsection{A cubical \texorpdfstring{$\mathrm{Ex}$}{Ex} functor}
The classical $\ex$ and $\ex^\infty$ functors in simplicial homotopy theory \cite{K} may be viewed as repeatedly subdividing simplices until horn-filling data can be extended. Let $\mathsf{sSet}$ denote the category of simplicial sets and simplicial maps between them. Then $\mathsf{sSet}$ carries a model category structure which is Quillen equivalent to the category of topological spaces, and the functor 
\[
\ex^\infty\colon\mathsf{sSet}\to\mathsf{sSet}
\]
is a fibrant replacement. In this subsection we consider a naive cubical analogue of $\ex^\infty$. For $l\ge1$ and
\[
s\colon\{1,\ldots,l\}\to\{0,1\},
\]
we denote
\[
P_l=I_s
\]
if $s$ is constant at 0. In particular we have $P_1=\Box^1$, and we do not distinguish them hereafter. For a cubical set $X$ and $n\ge0$, let 
\[(\ex X)_n=\mathrm{Hom}_{\cSet}(P_3^{\otimes n},X).\]
Then there is a functor
\[
\ex\colon\cSet\to\cSet
\]
with an obvious left adjoint
\[
\sd\colon\cSet\to\cSet,\ \sd X=\underset{\Box\downarrow X}{\mathrm{colim}}P_3^{\otimes n}.
\]
Let 
\[
\epsilon\colon[3]\to[1],\quad\epsilon(i)=\min\{1,\max\{0,i-1\}\}
\]
be a contraction. Then there is a natural transformation $\eta\colon1_{\cSet}\Rightarrow\ex$ such that for a cubical set $X$ and $x\in X_n$, 
\[\eta_X(x)=x(\epsilon^{\otimes n})\in(\ex X)_n.\] 
We show that $\eta$ is a natural weak equivalence by arguing similarly to the proof of Theorem \ref{main 2}.

Let $X$ be a cubical set. For an $n$-grid 
\[x\colon P_3^{\otimes n}\to X,\]
we set $\|x\|=[0,3]^n$. Faces and degeneracies on $\|x\|$ are defined in the obvious manner. We then define
\[
\|\ex X\|=\left(\coprod_{x\in\ex X}\|x\|\right)/\sim
\]
where the equivalence relation is generated by $\partial_{i,\alpha}\|x\|\sim\|x\partial_{i,\alpha}\|$ and $\sigma_i\|x\|\sim\|x\sigma_i\|$. One has that $\|\ex X\|$ is well-defined. For every grid 
\[x\colon P_3^{\otimes n}\to X,\]
let $\rho\colon\|x\|\to\|\ex X\|$ denote the canonical map. The canonical linear homeomorphisms $\|x\|\to|x|$ induce a natural homeomorphism
\begin{equation}
	\label{ex homeo}
	\|\ex X\|\xrightarrow{\cong}|\ex X|.
\end{equation}

Let 
\[x\colon P_3^{\otimes n}\to X,\]
be an $n$-grid in $X$. Then one has a $2n$-grid
\[
(\cdots((x\gamma_{n,1}\cdots\gamma_{1,1})[2,i_1])\cdots)[2n,i_n]\colon P_3\otimes P_1\otimes\cdots\otimes P_3\otimes P_1\to X
\]
and a $2n$-grid
\[
((\cdots((x\gamma_{n,1}\cdots\gamma_{1,1})[2,i_1])\cdots)[2n,i_n])(1\otimes\epsilon\otimes\cdots\otimes1\otimes\epsilon)\colon P_3^{\otimes 2n}\to X,
\]
where $i_1,\ldots,i_n\in\{1,2,3\}$. These are glued together to form
\[
Q(x)=\bigcup_{i_1=1,2,3}\cdots\bigcup_{i_n=1,2,3}\|((\cdots((x\gamma_{n,1}\cdots\gamma_{1,1})[2,i_1])\cdots)[2n,i_n])(1\otimes\epsilon\otimes\cdots\otimes1\otimes\epsilon)\|
\]
such that the maps 
\[\rho\colon\|((\cdots((x\gamma_{n,1}\cdots\gamma_{1,1})[2,i_1])\cdots)[2n,i_n])(1\otimes\epsilon\otimes\cdots\otimes1\otimes\epsilon)\|\to\|\ex X\|\] 
for all $i_1,\ldots,i_n\in\{1,2,3\}$ assemble to a map $\rho\colon Q(x)\to\|\ex X\|$. Then we have
\[
Q(x)\cong[0,3]\times[0,9]\times\cdots\times[0,3]\times[0,9]
\] Define a map $h(x)\colon\|x\|\times[0,1]\to Q(x)$ by
\[
\begin{aligned}
	&h(x)(x_1,\ldots,x_n,t)=\\
	&\qquad
	\begin{cases}
		(x_1,6(x_1-3)t+9,\ldots,x_n,6(x_n-3)t+9)&(0\le t\le\frac{1}{2})\\
		(2(3-x_1)t+2x_1-3,3x_1,\ldots,2(3-x_n)t+2x_n-3,3x_n)&(\frac{1}{2}\le t\le 1).
	\end{cases}
\end{aligned}
\]
We then set $\bar{h}(x)=\rho\circ h(x)$. The following lemma is immediate from the definition.

\begin{lemma}
	\label{G t=0,1}
	Let $X$ be a cubical set, and let $x\in(\ex X)_n$. Then
	\[
	\bar{h}(x)(-,0)=\rho\colon\|x\|\to\|\ex X\|
	\]
	and $\bar{h}(x)(-,1)$ takes values in $\eta(|X|)\subset\|\ex X\|$. Moreover, if $x=y(\epsilon^{\otimes n})$ for some $y\in X_n$, then $\bar{h}(x)(-,1)$ coincides with the composite
	\[
	\|x\|=\|y(\epsilon^{\otimes n})\|\xrightarrow{\epsilon^n}|y|\xrightarrow{\rho}|X|\xrightarrow{\eta}\|\ex X\|,
	\]
	where 
	\[\epsilon\colon[0,3]\to[0,1],\quad\epsilon(t)=\min\{1,\max\{0,t-1\}\}\]
	and
	\[
	\eta\colon|X|\to\|\ex X\|
	\]
	is the composite of the map $|\eta|\colon|X|\to|\ex X|$ and the inverse of the homeomorphism \eqref{ex homeo}.
\end{lemma}

The following lemma is proved quite analogously to Lemma \ref{R gluing}.

\begin{lemma}
	\label{G gluing}
	Let $X$ be a cubical set, and let $x\in(\ex X)_n$. For all $i$ and $\alpha$, the diagrams
	\[
	\xymatrix{
		\|x\partial_{i,\alpha}\|\times[0,1]\ar[d]_{\partial_{i,\alpha}\times 1}\ar[rr]^(.6){\bar{h}(x\partial_{i,\alpha})}&&\|\ex X\|\ar@{=}[d]\\
		\|x\|\times[0,1]\ar[rr]^(.57){\bar{h}(x)}&&\|\ex X\|
	}
	\qquad
	\xymatrix{
		\|x\sigma_i\|\times[0,1]\ar[d]_{\sigma_i\times 1}\ar[rr]^(.59){\bar{h}(x\sigma_i)}&&\|\ex X\|\ar@{=}[d]\\
		\|x\|\times[0,1]\ar[rr]^(.57){\bar{h}(x)}&&\|\ex X\|
	}
	\]
	commute.
\end{lemma}

\begin{proof}
	Let $x\in(\ex X)_n$. By the cubical identity,
	\[
	Q(x\sigma_i)=Q(x)\sigma_{2i-1}\sigma_{2i}\quad\text{and}\quad Q(x)\partial_{2i,1}\partial_{2i-1,1}=Q(x\partial_{i,1}).
	\]
	Hence the right diagram commutes. One also has
	\[
	Q(x)\partial_{2i-1,0}=Q(x\partial_{i,0})\sigma_{2i-1}.
	\]
	Hence for $\alpha=0$ and $0\le t\le\frac{1}{2}$,
	\begin{align*}
		&\bar{h}(x)(x_1,\ldots,x_{i-1},0,x_{i+1},\ldots,x_n,t)\\
		&=\rho(x_1,\bar{x}_1,\ldots,x_{i-1},\bar{x}_{i-1},0,9-18t,x_{i+1},\bar{x}_{i+1},\ldots,x_n,\bar{x}_n)\\
		&=\rho(x_1,\bar{x}_1,\ldots,x_{i-1},\bar{x}_{i-1},x_{i+1},\bar{x}_{i+1},\ldots,x_n,\bar{x}_n)
	\end{align*}
	where $\bar{x}_k=6(x_k-3)t+9$. Therefore the left diagram commutes for $\alpha=0$ and $0\le t\le\frac{1}{2}$. The remaining cases are proved analogously.
\end{proof}

\begin{proposition}
	\label{ex weak equivalence}
	The natural transformation $\eta\colon 1_{\cSet}\Rightarrow\ex$ is a natural weak equivalence.
\end{proposition}

\begin{proof}
	Let $X$ be a cubical set. By Lemmas \ref{G t=0,1} and \ref{G gluing}, the family of homotopies $\{\bar{h}(x)\}$, where $x$ ranges over all cubes in $\ex X$, assembles to a homotopy $\bar{h}\colon\|\ex X\|\times[0,1]\to\|\ex X\|$ such that $\bar{h}(-,0)=1_{\|\ex X\|}$ and $\bar{h}(-,1)$ factors as
	\[
	\|\ex X\|\xrightarrow{f}|X|\xrightarrow{\eta}\|\ex X\|.
	\]
	On the other hand, define a map $h'(y)\colon|y|\times[0,1]\to|y|$ by
	\[
	h'(y)(y_1,\ldots,y_n,t)=((1-t)y_1+t\epsilon(3y_1),\ldots,(1-t)y_n+t\epsilon(3y_n))
	\]
	for $y\in X_n$, and let $\bar{h'}(y)=\rho\circ h'(y)$. Then the family of homotopies $\{\bar{h'}(y)\}$, where $y$ ranges over all cubes in $X$, assembles to a homotopy $\bar{h'}\colon|X|\times[0,1]\to|X|$ such that $\bar{h'}(-,0)=1_{|X|}$ and $\bar{h'}(-,1)$ factors as
	\[
	|X|\xrightarrow{\eta}\|\ex X\|\xrightarrow{f}|X|.
	\]
	Therefore $f$ is the homotopy inverse of $\eta\colon|X|\to\|\ex X\|$. Composing with the homeomorphism \eqref{ex homeo}, the claim follows.
\end{proof}

Next we define the functor $\ex^\infty$ in a manner analogous to that in simplicial homotopy theory.
\begin{definition}
	Let $\ex^\infty\colon\cSet\to\cSet$ be the functor sending $X\in\cSet$ to 
	\[
	\mathrm{colim}(X\xrightarrow{\eta_X}\ex X\xrightarrow{\eta_{\ex X}}\ex^2 X\xrightarrow{\eta_{\ex^2X}}\cdots).
	\]
	We denote the natural transformation $1_{\cSet}\Rightarrow\ex^\infty$ by $\lambda$.
\end{definition}

The functor $\ex^\infty$ has similar properties to the one in simplicial homotopy theory.
\begin{proposition}
	\label{cubical Ex}
	The functor $\ex^\infty\colon\cSet\to\cSet$ and the natural transformation $\lambda\colon\allowbreak1_{\cSet}\Rightarrow\ex^\infty$ satisfy the following:

	\noindent(i) $\ex^\infty$ commutes with finite limits and filtered colimits;

	\noindent(ii) $\ex^\infty$ preserves Kan fibrations;

	\noindent(iii) $\lambda_X\colon X\to\ex^\infty X$ is an anodyne extension for any $X$.


\end{proposition}
\begin{proof}
	Since $\ex$ is a right adjoint, it commutes with finite limits. Since $P_3^{\otimes n}$ is a finite colimit of representable cubical sets, $\ex$ also commutes with filtered colimits. Filtered colimits in $\cSet$ are computed degreewise and commute with finite limits. Interchanging the filtered colimits defining $\ex^\infty$ with the given filtered colimit, we obtain (i). Note that the inclusion
	\[
	\sd\sqcap_{i,\alpha}^n\hookrightarrow\sd\Box^n
	\]
	is a weak equivalence and a monomorphism, hence an anodyne extension for any $i=1,\ldots,n$ and $\alpha=0,1$, $n\ge0$. The adjunction $\sd\dashv\ex$ therefore implies that $\ex$ preserves Kan fibrations. Since open boxes and representable cubes are finitely presentable, every lifting problem for $\ex^\infty f$ against an open box inclusion factors through $\ex^m f$ for some $m$. If $f$ is a Kan fibration, this lifting problem has a solution. Thus (ii) holds. Since $\eta_X\colon X\to\ex X$ is a weak equivalence by Proposition \ref{ex weak equivalence} and a monomorphism, it is an anodyne extension. Consequently, every map in
	\[
	X\xrightarrow{\eta_X}\ex X
	\xrightarrow{\eta_{\ex X}}\ex^2X
	\xrightarrow{\eta_{\ex^2X}}\cdots
	\]
	is an anodyne extension. Since anodyne extensions are closed under transfinite compositions, their composite
	$\lambda_X\colon X\to\ex^\infty X$ is an anodyne extension. This proves (iii).
\end{proof}

\subsection{Kan condition}
In the previous subsections, we constructed three functors 
\[
\mathcal{R},\ex,\ex^\infty\colon\cSet\to\cSet.
\] 
We will show that $\mathcal{R}$ sends a quasisymmetric cubical set to a signed quasisymmetric one, and $\ex^\infty$ sends a signed quasisymmetric cubical set to a signed quasisymmetric Kan complex.

\begin{lemma}
	\label{R symmetric}
	Let $X$ be a quasisymmetric cubical set. Then $\mathcal{R}X$ is signed quasisymmetric.
\end{lemma}

\begin{proof}
	For $(a,\kappa)\in B_n$ and $x^{s_1\cdots s_n}\in(\mathcal{R}X)_n$, let
	\[
	(x^{s_1\cdots s_n})(a,\kappa)=(x\kappa)^{(a_{\kappa^{-1}(1)}+s_{\kappa^{-1}(1)})\cdots(a_{\kappa^{-1}(n)}+s_{\kappa^{-1}(n)})},
	\]
	where $a_i+s_i$ is considered modulo $2$. It is straightforward to verify that this defines an action of $B_n$ on $(\mathcal{R}X)_n$. We have
	\begin{align*}
		((x^{s_1\cdots s_n})(a,\kappa))\partial_{\kappa(i),\alpha+a_i}
		&=((x\kappa)^{t_1\cdots t_n})\partial_{\kappa(i),\alpha+a_i}\\
		&=((x\kappa)\partial_{\kappa(i),\alpha+s_i})^{t_1\cdots t_{\kappa(i)-1}t_{\kappa(i)+1}\cdots t_n}\\
		&=((x\partial_{i,\alpha+s_i})\kappa^{(i)})^{t_1\cdots t_{\kappa(i)-1}t_{\kappa(i)+1}\cdots t_n}\\
		&=((x\partial_{i,\alpha+s_i})^{s_1\cdots s_{i-1}s_{i+1}\cdots s_n})(a,\kappa)^{(i)}\\
		&=((x^{s_1\cdots s_n})\partial_{i,\alpha})(a,\kappa)^{(i)},
	\end{align*}
	where $t_i=s_{\kappa^{-1}(i)}+a_{\kappa^{-1}(i)}$. The identity for degeneracies is obtained analogously.
\end{proof}

In simplicial homotopy theory, $\ex^\infty X$ is a Kan complex for any simplicial set $X$. However, $\ex^\infty X$ need not be a Kan complex for an arbitrary cubical set $X$. Here is an example.
\begin{example}
	Let $X=\Box^1$ and $x\in\Box_1^1$ denote the unique nondegenerate 1-cube. Let
	\[
	u_{2,1}=\lambda(x),u_{1,1}=u_{2,0}=\lambda(x)\partial_{1,1}\sigma_1\in(\ex^\infty\Box^1)_1,
	\]
	then they determine a horn $u\colon\sqcap_{1,0}^2\to\ex^\infty\Box^1$. Suppose that it has a filler $w\in(\ex^\infty\Box^1)_2$. Then $w\partial_{1,0}\in(\ex^\infty\Box^1)_1$ satisfies
	\[
	(w\partial_{1,0})\partial_{1,0}=\lambda(x)\partial_{1,1}=x\partial_{1,1},\quad(w\partial_{1,0})\partial_{1,1}=\lambda(x)\partial_{1,0}=x\partial_{1,0},
	\]
	which is impossible.
\end{example}
Instead, we will show that $\ex^\infty X$ is a Kan complex whenever $X$ is signed quasisymmetric. The main technical argument is as follows.
\begin{lemma}
	\label{ex filler}
	If $X$ is a signed quasisymmetric cubical set, then for any cubical map $u\colon\sqcap_{n+1,0}^{n+1}\to X$, there is a cubical map $w\colon\Box^{n+1}\to\ex X$ such that the following diagram commutes.
	\begin{equation}
		\label{ex diagram}
		\xymatrix{
			\sqcap_{n+1,0}^{n+1}\ar[r]^u\ar@{^(->}[d]&X\ar[d]^{\eta}\\
			\Box^{n+1}\ar[r]_w&\ex X
		}
	\end{equation}
\end{lemma}
To prove the key lemma, we find an explicit grid
\begin{equation}
	\label{filler construction}
	w\colon P_3^{\otimes n+1}\to X
\end{equation}
to make the diagram commute. To this end, we define
\[
\{z(p_1,\ldots,p_n,1)\}_{(p_1,\ldots,p_n)\in\{1,2,3\}^n}\subset X_{n+1},
\]
which patch together to form
\[
z\colon P_3^{\otimes n}\otimes P_1\to X,
\]
and we let 
\[w=z(1^{\otimes n}\otimes\epsilon).\]
The horn $u\colon\sqcap_{n+1,0}^{n+1}\to X$ provides $x_{i,\alpha},x\in X_n$ such that
\[
x\partial_{i,\alpha}=x_{i,\alpha}\partial_{n,1},\ x_{j,\beta}\partial_{i,\alpha}=x_{i,\alpha}\partial_{j-1,\beta}
\]
for $1\le i<j\le n$ and $\alpha,\beta=0,1$. For $(p_1,\ldots,p_{n})\in\{1,2,3\}^n$, let 
\[
\begin{gathered}
	\{i_1<\cdots<i_r\}=\{i\in\{1,\ldots,n\}\mid p_i\ne2\},\\
	\alpha_j=\alpha(p_{i_j})=\begin{cases}
		0&(p_{i_j}=1)\\
		1&(p_{i_j}=3)
	\end{cases},\quad y=x\partial_{i_r,\alpha_r}\cdots\partial_{i_1,\alpha_1},
\end{gathered}
\]
and let 
\[(a_{x,y},\kappa_{x,y})=\begin{pmatrix}
	1&2&\cdots&n-r&n-r+1&\cdots&n&n+1\\
	j_1&j_2&\cdots&j_{n-r}&(-1)^{\alpha_1}i_1&\cdots&(-1)^{\alpha_r}i_r&n+1
\end{pmatrix}\in B_{n+1},\]
where $\{j_1<\cdots <j_{n-r}\}=\{1,\ldots,n\}-\{i_1,\ldots,i_r\}$. We let
\begin{equation}
	\label{z def}
	z(p_1,\ldots,p_n,1)=\begin{cases}
		x\sigma_{n+1}&((p_1,\ldots,p_n)=(2,\ldots,2))\\
		z_{(i_1,\alpha_1)\cdots(i_r,\alpha_r)}&((p_1,\ldots,p_n)\ne(2,\ldots,2)),
	\end{cases}
\end{equation}
where
\[
z_{(i_1,\alpha_1),\ldots,(i_r,\alpha_r)}=
\begin{cases}
	(x_{i_1,\alpha_1}\gamma_{n,0})(a_{x,y},\kappa_{x,y})&(r=1)\\
	(x_{i_r,\alpha_r}\partial_{i_{r-1},\alpha_{r-1}}\cdots\partial_{i_1,\alpha_1}\gamma_{n-r+1,0}\gamma_{n-r+1,1}\cdots\gamma_{n-1,1})(a_{x,y},\kappa_{x,y})&(r\ge 2).
\end{cases}
\]

\begin{lemma}
	\label{top face}
	For $i=1,\ldots,n$ and $\alpha=0,1$,
	\[
	x\sigma_{n+1}\partial_{i,\alpha}=z_{(i,\alpha)}\partial_{i,1-\alpha}.
	\]
\end{lemma}

\begin{proof}
	By definition,
	\[
	x\sigma_{n+1}\partial_{i,\alpha}=x\partial_{i,\alpha}\sigma_n=x_{i,\alpha}\partial_{n,1}\sigma_n=x_{i,\alpha}\gamma_{n,0}\partial_{n,1}.
	\]
	Since
	\[(a_{x,x\partial_{i,\alpha}},\kappa_{x,x\partial_{i,\alpha}})=\begin{pmatrix}
		1&2&\cdots&i-1&i&\cdots&n-1&n&n+1\\
		1&2&\cdots&i-1&i+1&\cdots&n&(-1)^\alpha i&n+1
	\end{pmatrix}\in B_{n+1},\]
	one has $(a_{x,x\partial_{i,\alpha}},\kappa_{x,x\partial_{i,\alpha}})^{(n)}=1$. Hence
	\begin{align*}
		z_{(i,\alpha)}\partial_{i,1-\alpha}&=((x_{i,\alpha}\gamma_{n,0})(a_{x,x\partial_{i,\alpha}},\kappa_{x,x\partial_{i,\alpha}}))\partial_{i,1-\alpha}\\
		&=(x_{i,\alpha}\gamma_{n,0}\partial_{n,1})(a_{x,x\partial_{i,\alpha}},\kappa_{x,x\partial_{i,\alpha}})^{(n)}\\
		&=x_{i,\alpha}\partial_{n,1}\sigma_n\\
		&=x\partial_{i,\alpha}\sigma_n\\
		&=x\sigma_{n+1}\partial_{i,\alpha}
	\end{align*}
	
	as stated.
\end{proof}

\begin{lemma}
	\label{side face}
	For $k=1,\ldots,r$ with $r\ge 2$, we have
	\[
	z_{(i_1,\alpha_1)\cdots(i_r,\alpha_r)}\partial_{i_k,1-\alpha_k}=z_{(i_1,\alpha_1)\cdots(i_{k-1},\alpha_{k-1})(i_{k+1},\alpha_{k+1})\cdots(i_r,\alpha_r)}\partial_{i_k,\alpha_k}.
	\]
\end{lemma}

\begin{proof}
	Suppose $k<r$, and set
	\[
	v=x_{i_r,\alpha_r}\partial_{i_{r-1},\alpha_{r-1}}\cdots\partial_{i_{k+1},\alpha_{k+1}}\partial_{i_{k-1},\alpha_{k-1}}\cdots\partial_{i_1,\alpha_1}\quad\text{and}\quad v'=x_{i_r,\alpha_r}\partial_{i_{r-1},\alpha_{r-1}}\cdots\partial_{i_1,\alpha_1}.
	\]
	Then we have $v'=v\partial_{i_k-k+1,\alpha_k}$. Set
	\[
	y_k=x\partial_{i_r,\alpha_r}\cdots\partial_{i_{k+1},\alpha_{k+1}}\partial_{i_{k-1},\alpha_{k-1}}\cdots\partial_{i_1,\alpha_1},
	\]
	then we have $y=y_k\partial_{i_k-k+1,\alpha_k}$ and
	\[
	z_{(i_1,\alpha_1)\cdots(i_{k-1},\alpha_{k-1})(i_{k+1},\alpha_{k+1})\cdots(i_r,\alpha_r)}=\begin{cases}
		(v\gamma_{n,0})(a_{x,y_k},\kappa_{x,y_k})&(r=2)\\
		(v\gamma_{n-r+2,0}\gamma_{n-r+2,1}\cdots\gamma_{n-1,1})(a_{x,y_k},\kappa_{x,y_k})&(r\ge3).
	\end{cases}
	\]
	Since
	\[
	(a_{x,y_k})_l=\begin{cases}
		\alpha_{l-n+r-1}&(l=n-r+2,\ldots,n-r+k)\\
		\alpha_{l-n+r}&(l=n-r+k+1,\ldots,n)\\
		0&(\text{otherwise})
	\end{cases}
	\]
	and
	\[
	\kappa_{x,y_k}(l)=\begin{cases}
		j_l&(l=1,\ldots,i_k-k)\\
		i_k&(l=i_k-k+1)\\
		j_{l-1}&(l=i_k-k+2,\ldots,n-r+1)\\
		i_{l-n+r-1}&(l=n-r+2,\ldots,n-r+k)\\
		i_{l-n+r}&(l=n-r+k+1,\ldots,n)\\
		n+1&(l=n+1),
	\end{cases}
	\]
	we have $(a_{x,y},\kappa_{x,y})^{(n-r+k)}=(a_{x,y_k},\kappa_{x,y_k})^{(i_k-k+1)}$. Hence
	\begin{align*}
		z_{(i_1,\alpha_1)\cdots(i_r,\alpha_r)}\partial_{i_k,1-\alpha_k}
		&=((v'\gamma_{n-r+1,0}\gamma_{n-r+1,1}\cdots\gamma_{n-1,1})(a_{x,y},\kappa_{x,y}))\partial_{i_k,1-\alpha_k}\\
		&=(v'\gamma_{n-r+1,0}\gamma_{n-r+1,1}\cdots\gamma_{n-1,1}\partial_{n-r+k,1})(a_{x,y},\kappa_{x,y})^{(n-r+k)}\\
		&=(v'\gamma_{n-r+1,0}\gamma_{n-r+1,1}\cdots\gamma_{n-2,1})(a_{x,y},\kappa_{x,y})^{(n-r+k)}\\
		&=(v\partial_{i_k-k+1,\alpha_k}\gamma_{n-r+1,0}\gamma_{n-r+1,1}\cdots\gamma_{n-2,1})(a_{x,y_k},\kappa_{x,y_k})^{(i_k-k+1)}\\
		&=(v\gamma_{n-r+2,0}\gamma_{n-r+2,1}\cdots\gamma_{n-1,1}\partial_{i_k-k+1,\alpha_k})(a_{x,y_k},\kappa_{x,y_k})^{(i_k-k+1)}\\
		&=((v\gamma_{n-r+2,0}\gamma_{n-r+2,1}\cdots\gamma_{n-1,1})(a_{x,y_k},\kappa_{x,y_k}))\partial_{i_k,\alpha_k}\\
		&=z_{(i_1,\alpha_1)\cdots(i_{k-1},\alpha_{k-1})(i_{k+1},\alpha_{k+1})\cdots(i_r,\alpha_r)}\partial_{i_k,\alpha_k}.
	\end{align*}
	The case $k=r$ is proved analogously.
\end{proof}
Thus
\[
\{z(p_1,\ldots,p_n,1)\}_{(p_1,\ldots,p_n)\in\{1,2,3\}^n}\subset X_{n+1}
\]
patch together to form
\[
z\colon P_3^{\otimes n}\otimes P_1\to X
\]
by Lemmas \ref{top face} and \ref{side face}. To show that
\[
w=z(1^{\otimes n}\otimes\epsilon)\colon P_3^{\otimes n+1}\to X
\]
makes the diagram \eqref{ex diagram} commute, we need the following lemmas.

\begin{lemma}
	\label{z top face}
	We have
	\[
	z\partial_{n+1,1}=\eta(x)=x(\epsilon^{\otimes n}).
	\]
\end{lemma}

\begin{proof}
	We have $z(2,2,\ldots,2,1)\partial_{n+1,1}=x\sigma_{n+1}\partial_{n+1,1}=x$. For $r=1$,
	\[
	z_{(i_1,\alpha_1)}\partial_{n+1,1}=((x_{i_1,\alpha_1}\gamma_{n,0})(a_{x,y},\kappa_{x,y}))\partial_{n+1,1}=(x_{i_1,\alpha_1}\partial_{n,1}\sigma_n)(a_{x,y},\kappa_{x,y})^{(n+1)}=x\partial_{i_1,\alpha_1}\sigma_{i_1}.
	\]
	For $r\ge 2$, we can see in exactly the same manner that
	\[
	z_{(i_1,\alpha_1)\cdots(i_r,\alpha_r)}\partial_{n+1,1}=x\partial_{i_r,\alpha_r}\cdots\partial_{i_1,\alpha_1}\sigma_{i_1}\cdots\sigma_{i_r}.
	\]
	Thus since $z\partial_{n+1,1}$ is obtained by concatenating $z(2,2,\ldots,2,1)\partial_{n+1,1}$ and $z_{(i_1,\alpha_1)\cdots(i_r,\alpha_r)}\partial_{n+1,1}$, the claim follows.
\end{proof}

\begin{lemma}
	\label{z side face}
	For $i=1,\ldots,n$ and $\alpha=0,1$,
	\[
	z\partial_{i,\alpha}=x_{i,\alpha}(\epsilon^{\otimes n-1}\otimes1).
	\]
\end{lemma}

\begin{proof}
	As in the proof of Lemma \ref{z top face}, it suffices to show that
	\[
	z_{(i_1,\alpha_1)\cdots(i_r,\alpha_r)}\partial_{i_k,\alpha_k}=x_{i_r,\alpha_r}\partial_{i_{r-1},\alpha_{r-1}}\cdots\partial_{i_1,\alpha_1}\sigma_{i_1}\cdots\sigma_{i_{k-1}}\sigma_{i_{k+1}-1}\cdots\sigma_{i_r-1}.
	\]
	For $r=1$,
	\[
	z_{(i_1,\alpha_1)}\partial_{i_1,\alpha_1}=((x_{i_1,\alpha_1}\gamma_{n,0})(a_{x,y},\kappa_{x,y}))\partial_{i_1,\alpha_1}=(x_{i_1,\alpha_1}\gamma_{n,0}\partial_{n,0})(a_{x,y},\kappa_{x,y})^{(n)}=x_{i_1,\alpha_1}.
	\]
	For $r\ge 2$, it is straightforward to identify $(a_{x,y},\kappa_{x,y})^{(n-r+k)}=(a_{x\partial_{i_k,\alpha_k},y},\kappa_{x\partial_{i_k,\alpha_k},y})$. Hence
	\begin{align*}
		&z_{(i_1,\alpha_1)\cdots(i_r,\alpha_r)}\partial_{i_k,\alpha_k}\\
		&=((x_{i_r,\alpha_r}\partial_{i_{r-1},\alpha_{r-1}}\cdots\partial_{i_1,\alpha_1}\gamma_{n-r+1,0}\gamma_{n-r+1,1}\cdots\gamma_{n-1,1})(a_{x,y},\kappa_{x,y}))\partial_{i_k,\alpha_k}\\
		&=(x_{i_r,\alpha_r}\partial_{i_{r-1},\alpha_{r-1}}\cdots\partial_{i_1,\alpha_1}\gamma_{n-r+1,0}\gamma_{n-r+1,1}\cdots\gamma_{n-1,1}\partial_{n-r+k,0})(a_{x,y},\kappa_{x,y})^{(n-r+k)}\\
		&=(x_{i_r,\alpha_r}\partial_{i_{r-1},\alpha_{r-1}}\cdots\partial_{i_1,\alpha_1}\sigma_{n-r+1}\cdots\sigma_{n-1})(a_{x\partial_{i_k,\alpha_k},y},\kappa_{x\partial_{i_k,\alpha_k},y})\\
		&=(x_{i_r,\alpha_r}\partial_{i_{r-1},\alpha_{r-1}}\cdots\partial_{i_1,\alpha_1})\sigma_{i_1}\cdots\sigma_{i_{k-1}}\sigma_{i_{k+1}-1}\cdots\sigma_{i_r-1}\\*
		&\qquad(\cdots(a_{x\partial_{i_k,\alpha_k},y},\kappa_{x\partial_{i_k,\alpha_k},y})^{(n-1)}\cdots)^{(n-r+1)}\\
		&=x_{i_r,\alpha_r}\partial_{i_{r-1},\alpha_{r-1}}\cdots\partial_{i_1,\alpha_1}\sigma_{i_1}\cdots\sigma_{i_{k-1}}\sigma_{i_{k+1}-1}\cdots\sigma_{i_r-1}
	\end{align*}
	as desired.
\end{proof}

\begin{proof}
	[Proof of Lemma \ref{ex filler}]
	To show that the diagram \eqref{ex diagram} commutes, it suffices to show that
	\[
	w\partial_{n+1,1}=\eta(x),\quad w\partial_{i,\alpha}=\eta(x_{i,\alpha})
	\]
	for $i=1,\ldots,n$, $\alpha=0,1$. Since
	\[
	w=z(1^{\otimes n}\otimes\epsilon),
	\]
	the first one follows immediately from Lemma \ref{z top face}. For the second, we have
	\[
	w\partial_{i,\alpha}=z(1^{\otimes n}\otimes\epsilon)\partial_{i,\alpha}=z\partial_{i,\alpha}(1^{\otimes n-1}\otimes\epsilon)=x_{i,\alpha}(\epsilon\otimes\cdots\otimes\epsilon)=\eta(x_{i,\alpha})
	\]
	by Lemma \ref{z side face}. This completes the proof.
\end{proof}

\begin{theorem}
	\label{exinfty Kan complex}
	The cubical set $\ex^\infty X$ is a Kan complex whenever $X$ is signed quasisymmetric. 
\end{theorem}
\begin{proof}
	Let $u\colon\sqcap_{i,\alpha}^{n+1}\to\ex^\infty X$ be an $(n+1)$-horn where $i=1,\ldots,n+1$, $\alpha=0,1$. Then there is $m\ge0$ such that $u$ factors through an $(n+1)$-horn in $\ex^m X$. If $(i,\alpha)=(n+1,0)$, then $u$ has a filler by Lemma \ref{ex filler}. One notes that $\ex^m X$ and the map
	\[
	\eta_{\ex^mX}\colon\ex^mX\to\ex^{m+1}X
	\]
	are signed quasisymmetric whenever $X$ is. This implies that $\ex^\infty X$ satisfies the Kan conditions for the remaining horns.
\end{proof}
Since the natural transformations $\mu$, $\iota$, $\lambda$ are degreewise injective and natural weak equivalences by Theorem \ref{main 2}, Propositions \ref{R} and \ref{cubical Ex}, we have a functorial fibrant replacement on $\cSet$, equipped with the model category structure as in Theorem \ref{cSet model structure}.
\begin{corollary}
	\label{fib replacement}
	For a cubical set $X$, the assignment
	\[
	X\xrightarrow{\mu_X}\mathcal{US}X\xrightarrow{\iota_{\mathcal{US}X}}\mathcal{RUS}X\xrightarrow{\lambda_{\mathcal{RUS}X}}\ex^\infty\mathcal{RUS}X
	\]
	is a funtorial fibrant replacement.
\end{corollary}

\begin{remark}
	\label{quasisymm vs symm}
	We can see that proofs of Lemmas \ref{top face} and \ref{side face} use the first identity in \eqref{signed permutation identity 1}, and the proofs of Lemmas \ref{z top face} and \ref{z side face} are due to the second identity. Since the identities in \eqref{signed permutation identity 2} play no role here, (signed) quasisymmetry would be sufficient throughout this paper.
\end{remark}

\section{Discrete homotopy groups of quasisymmetric cubical sets}\label{Discrete homotopy groups of quasisymmetric cubical sets}

In this section, we prove that the discrete homotopy groups of a quasisymmetric cubical set are naturally isomorphic to the homotopy groups of its geometric realization. Applying this result, we also prove the Hurewicz theorem for the discrete homotopy groups of a quasisymmetric cubical set.

\begin{proposition}
	\label{DHG main}
	For a pointed quasisymmetric cubical set $(X,x_0)$, there is a natural isomorphism
	\[
	\pi_n^\delta( X,x_0)\xrightarrow{\cong}\pi_n(\ex^\infty\mathcal{R}X,x_0).
	\]
\end{proposition}
The $n=0$ case is obvious, and we assume $n\ge1$ hereafter. In order to prove the proposition, we construct explicit homomorphisms
\[
\Phi\colon\pi_n^\delta( X,x_0)\leftrightarrow\pi_n(\ex^\infty\mathcal{R}X,x_0)\colon\Psi
\]
and show that they are inverse to each other. For $m\ge0$, let
\[
\lambda_m\colon\ex^m\Rightarrow\ex^\infty
\]
be the canonical natural transformation. Then we have 
\[
\lambda_0=\lambda\colon1_{\cSet}\Rightarrow\ex^\infty.
\]
To characterize cubes in $\ex^m Y$ for a cubical set $Y$, let
\[\epsilon_m\colon[3^{m+1}]\to[3^{m}],\quad\epsilon_m(i)=\min\{3^m,\max\{0,i-3^m\}\},\] then we have $\epsilon_0=\epsilon$, and we have the following.
\begin{lemma}
	\label{ex description}
	For $n\ge0$, elements of $(\ex^\infty X)_n$ are represented by $n$-grids 
	\[
	P_{3^m}^{\otimes n}\to X
	\]
	in $X$ for $m\ge0$, subject to the following identifications. For grids
	\[
	x\colon P_{3^m}^{\otimes n}\to X\text{ and }y\colon P_{3^k}^{\otimes n}\to X
	\]
	such that $m\ge k$, $\lambda_m(x)=\lambda_k(y)$ if and only if
	\[
	x=y(({\epsilon_{k}\circ\cdots\circ\epsilon_{m-1}})^{\otimes n}).
	\]
\end{lemma}

We start by defining $\Phi$. An element in $\pi_n^\delta( X,x_0)$ is represented by a spherical grid
\[
x\colon\bigl(I_{s_1}\otimes\cdots\otimes I_{s_n},\partial(I_{s_1}\otimes\cdots\otimes I_{s_n})\bigr)\to(X,x_0),
\]
where 
\[
s_k\colon\{1,\ldots,l_k\}\to\{0,1\}
\]
for $k=1,\ldots,n$, $l_1,\ldots,l_n\ge1$. 
\begin{definition}
	\label{Phi def}
	For a pointed cubical set $(X,x_0)$ and a spherical grid
	\[
	x\colon\bigl(I_{s_1}\otimes\cdots\otimes I_{s_n},\partial(I_{s_1}\otimes\cdots\otimes I_{s_n})\bigr)\to(X,x_0),
	\]
	we have a spherical grid
	\[
	\begin{gathered}	\widetilde{x}\colon\bigl(P_{l_1}\otimes\cdots\otimes P_{l_n},\partial(P_{l_1}\otimes\cdots\otimes P_{l_n})\bigr)\to(\mathcal{R}X,x_0),\\\widetilde{x}(i_1,\ldots,i_n)=(x(i_1,\ldots,i_n))^{s_1(i_1)\cdots s_n(i_n)},\end{gathered}
	\]
	where $i_k\in\{1,\ldots,l_k\}$ for $k=1,\ldots,n$. Let $m=\min\{k\ge0\mid3^k\ge\max\{l_1,\ldots,l_n\}\}$, 
	\[
	r_k\colon[3^m]\to[l_k],\quad i\mapsto\min\{i,l_k\}
	\]
	for $k=1,\ldots,n$, then
	\[
	\widetilde{x}(r_1\otimes\cdots\otimes r_n)\colon\bigl((P_{3^m})^{\otimes n},\partial((P_{3^m})^{\otimes n})\bigr)\to(\mathcal{R}X,x_0)
	\]
	corresponds under the adjunction to a cubical map
	\[
	(\Box^n,\partial\Box^n)\to(\ex^m\mathcal{R}X,x_0).
	\]
	If $X$ is quasisymmetric, $\ex^\infty\mathcal{R}X$ is a Kan complex by Lemma \ref{R symmetric} and Theorem \ref{exinfty Kan complex}. We define $\Phi(x)\in\pi_n(\ex^\infty\mathcal{R}X,x_0)$ be the cubical homotopy class $[\lambda_m(\widetilde{x}(r_1\otimes\cdots\otimes r_n))]_\Box$ represented by
	\[
	\lambda_m(\widetilde{x}(r_1\otimes\cdots\otimes r_n))\colon(\Box^n,\partial\Box^n)\to(\ex^\infty\mathcal{R}X,x_0).
	\]
\end{definition}
\begin{lemma}
	\label{contraction are homotopic to each other}
	For $l'\ge l\ge1$, any two contractions
	\[
	r,r'\colon[l']\to[l]
	\]
	are cubically homotopic as cubical maps $P_{l'}\to P_{l}$ relative to the endpoints. 
\end{lemma}
\begin{proof}
	It suffices to assume that $l'=l+1$, and we prove the lemma by induction on $l$. A contraction $[l+1]\to[l]$ contracts exactly the $k$-th 1-cube in $P_{l+1}$ for $k=1,\ldots,l+1$, and we denote it by $r(k,l+1)$. For $1\le k<k'<l+1$, $r(k,l+1)$ and $r(k',l+1)$ restrict to $r(k,l)$ and $r(k',l)$ respectively. By the induction hypothesis, there is $p\ge1$ and
	\[
	s\colon\{1,\ldots,p\}\to\{0,1\}
	\] 
	such that
	$r(k,l)$ and $r(k',l)$ are cubically homotopic via
	\[
	H\colon(P_{l}\otimes I_s,\partial P_l\otimes I_s)\to(P_{l-1},\partial P_{l-1}).
	\]
	Then $r(k,l+1)$ and $r(k',l+1)$ are cubically homotopic via
	\[
	H'=H+_1(a_l\sigma_2)\colon(P_{l+1}\otimes I_s,\partial P_{l+1}\otimes I_s)\to(P_{l},\partial P_l).
	\]
	It suffices to show that $r(l,l+1)$ and $r(l+1,l+1)$ are cubically homotopic via
	\[
	H''=(a_1\sigma_2)+_1\cdots+_1(a_{l-1}\sigma_2)+_1(a_l\gamma_{1,1})+_1(a_l\gamma_{1,0})\colon(P_{l+1}\otimes \Box^1,\partial P_{l+1}\otimes \Box^1)\to(P_{l},\partial P_l).
	\]
	The proof is finished.
\end{proof}
\begin{lemma}
	\label{cubical homotopy vs grid homotopy}
	Let $(Z,z_0)$ be a pointed quasisymmetric cubical set and
	\[
	x\colon(P_{l_1}\otimes\cdots\otimes P_{l_n},\partial(P_{l_1}\otimes\cdots\otimes P_{l_n}))\to(Z,z_0).
	\]
	For $i=1,\ldots,n$, let
	\[
	r_i,r_i'\colon[l_i']\to[l_i]
	\]
	be contractions. Then $x(r_1\otimes\cdots\otimes r_n)$ and $x(r_1'\otimes\cdots\otimes r_n')$ are cubically homotopic.
\end{lemma}
\begin{proof}
	It suffices to assume that there is $k=1,\ldots,n$ such that $r_i=r_i'$ for $i\in\{1,\ldots,n\}-\{k\}$. By Lemma \ref{contraction are homotopic to each other}, there is an $(n+1)$-grid
	\[
	y\colon P_{l_1'}\otimes\cdots\otimes P_{l_k'}\otimes I_s\otimes P_{l_{k+1}'}\otimes\cdots\otimes P_{l_n'}\to Z
	\]
	such that
	\[
	y\partial_{i,\alpha}=\begin{cases}
		x(r_1\otimes\cdots\otimes r_n)&((i,\alpha)=(k+1,0))\\
		x(r_1'\otimes\cdots\otimes r_n')&((i,\alpha)=(k+1,1))\\
		z_0&(\text{otherwise}).
	\end{cases}
	\]
	Thus $x(r_1\otimes\cdots\otimes r_n)$ and $x(r_1'\otimes\cdots\otimes r_n')$ are cubically homotopic via the $(n+1)$-grid $yC_{k+1,n+1}$, where
	\[
	C_{k+1,n+1}=\begin{pmatrix}
		1&2&\cdots&k&k+1&k+2&\cdots&n+1\\
		1&2&\cdots&k&n+1&k+1&\cdots&n
	\end{pmatrix}
	\in\Sigma_{n+1}
	\]
	by Proposition \ref{rotate grid}. The proof is finished.
\end{proof}
\begin{lemma}
	The map 
	\[\Phi\colon\pi_n^\delta( X,x_0)\to\pi_n(\ex^\infty\mathcal{R}X,x_0),\quad[x]\mapsto\Phi(x)\] 
	is well-defined.
\end{lemma}
\begin{proof}
	We first show that $\Phi(x)=\Phi(x')$ if $x\simeq_1 x'$. Suppose that there is $k=1,\ldots,n+1$ such that 
	\[
	x\overset{k}{\Leftrightarrow}x',
	\]
	then we have
	\[
	x\overset{n+1}{\Leftrightarrow}x'
	\]
	by Proposition \ref{rotate grid}. 
	Then there is an $(n+1)$-grid
	\[
	y\colon\bigl((I_{s_1}\otimes\cdots\otimes I_{s_n})\otimes I_s,\partial(I_{s_1}\otimes\cdots\otimes I_{s_n})\otimes I_s\bigr)\to(X,x_0)
	\]
	such that
	\[
	y\partial_{i,\alpha}=\begin{cases}
		x&((i,\alpha)=(n+1,0))\\
		x'&((i,\alpha)=(n+1,1))\\
		x_0&(\text{otherwise}).
	\end{cases}
	\]
	It suffices to assume that $I_s$ has length 1. Let
	\[
	r_{n+1}\colon[3^m]\to[1],\quad i\mapsto\min\{i,1\}, 
	\]
	then 
	\[
	\lambda_m(\widetilde{y}(r_1\otimes\cdots\otimes r_n\otimes r_{n+1}))
	\]
	determines a cubical map
	\[
	(\Box^{n}\otimes\Box^1,\partial\Box^n\otimes\Box^1)\to(\ex^\infty\mathcal{R}X,x_0),
	\]
	which implies that
	\[
	\Phi(x)=\Phi(x')\in\pi_n(\ex^\infty\mathcal{R}X,x_0).
	\]
	It remains to show that $\Phi(x)=\Phi(x')$ if $x\simeq_0x'$. Assume that there is $p=1,\ldots,n$, $q=1,\ldots,l_p+1$ and
	\[
	s_p'\colon\{1,\ldots,l_p+1\}\to\{0,1\},\quad s_p'(i)=\begin{cases}
		s_p(i)&(i=1,\ldots,q-1)\\
		s_p(i-1)&(i=q+1,\ldots,l_p+1).
	\end{cases}
	\]
	Then there is a contraction map
	\[
	r\colon I_{s_p'}\to I_{s_p}
	\]
	and we let
	\[
	x'=x(1^{\otimes p-1}\otimes r\otimes1^{\otimes n-p}).
	\]
	This implies that
	\[
	\widetilde{x'}=\widetilde{x}(1^{\otimes p-1}\otimes r\otimes1^{\otimes n-p}).
	\]
	If $l_p<3^m$, we have 
	\[
	\Phi(x)=\Phi(x')\in\pi_n(\ex^\infty\mathcal{R}X,x_0)
	\]
	immediately by Lemma \ref{contraction are homotopic to each other}. Suppose that $l_p=3^m$. Then $\Phi(x')\in\pi_n(\ex^\infty\mathcal{R}X,x_0)$ is represented by
	\[
	\lambda_{m+1}(\widetilde{x'}(r_1'\otimes\cdots\otimes r_n'))\colon(\Box^n,\partial\Box^n)\to(\ex^\infty\mathcal{R}X,x_0),
	\]
	where
	\[
	r_k'\colon[3^{m+1}]\to[l_k],\quad i\mapsto\min\{i,l_k\}
	\]
	for $k=1,\ldots,p-1,p+1,\ldots,n$ and
	\[
	r_p'\colon[3^{m+1}]\to[l_p+1],\quad i\mapsto\min\{i,l_p+1\}.
	\]
	To show that 
	\[
	\Phi(x)=\Phi(x')\in\pi_n(\ex^\infty\mathcal{R}X,x_0),
	\]
	note that 
	\begin{equation}
		\label{x C_0 case}
		\widetilde{x}((r_1\circ\epsilon_m)\otimes\cdots\otimes (r_n\circ\epsilon_m))\colon\bigl(P_{3^{m+1}}^{\otimes n},\partial(P_{3^{m+1}}^{\otimes n})\bigr)\to(\mathcal{R}X,x_0)
	\end{equation}
	also corresponds under the adjunction to a cubical map
	\[
	(\Box^n,\partial\Box^n)\to(\ex^{m+1}\mathcal{R}X,x_0).
	\]
	On the other hand, $r\colon I_{s_p'}\to I_{s_p}$ induces
	\[
	r\colon[l_p+1]\to[l_p],
	\]
	and we have
	\begin{equation}
		\label{x' C_0 case}
		\begin{gathered}
			\widetilde{x'}(r_1'\otimes\cdots\otimes r_n')
			=\widetilde{x}(r_1'\otimes\cdots\otimes r_{p-1}'\otimes(r\circ r_p')\otimes r_{p+1}'\otimes\cdots\otimes r_n')\colon\\
			\bigl(P_{3^{m+1}}^{\otimes n},\partial(P_{3^{m+1}}^{\otimes n})\bigr)\to(\mathcal{R}X,x_0)
		\end{gathered}
	\end{equation}
	which corresponds under the adjunction to a cubical map
	\[
	(\Box^n,\partial\Box^n)\to(\ex^{m+1}\mathcal{R}X,x_0).
	\]
	Then \eqref{x C_0 case} and \eqref{x' C_0 case} are cubically homotopic to each other by Lemma \ref{cubical homotopy vs grid homotopy}. The proof is finished.
\end{proof}
\begin{lemma}
	The map $\Phi$ is a homomorphism.
\end{lemma}
\begin{proof}
	For elements $\beta,\beta'\in\pi_n^\delta( X,x_0)$, we may choose representatives
	\[
	\begin{gathered}	x\colon\bigl(I_{s_1}\otimes\cdots\otimes I_{s_{n-1}}\otimes I_{s_n},\partial(I_{s_1}\otimes\cdots\otimes I_{s_{n-1}}\otimes I_{s_n})\bigr)\to(X,x_0),\\ x'\colon\bigl(I_{s_1}\otimes\cdots\otimes I_{s_{n-1}}\otimes I_{s_n'},\partial(I_{s_1}\otimes\cdots\otimes I_{s_{n-1}}\otimes I_{s_n'})\bigr)\to(X,x_0)\end{gathered}
	\]
	such that both are spherical grids and both are of size $(3^m,\ldots,3^m)$. Then $\beta\cdot\beta'$ is represented by $x+_nx'$ and we aim to show that
	\[
	\Phi(x)\cdot\Phi(x')=\Phi(x+_nx').
	\]
	By \eqref{product Kan complex}, $\Phi(x)\cdot\Phi(x')$ is represented by $\bar{h}\partial_{n+1,0}$, where 
	\[
	\bar{h}\colon\Box^{n+1}\to\ex^\infty\mathcal{R}X
	\] 
	is a filler of
	\[
	h\colon\sqcap_{n+1,0}^{n+1}\to\ex^\infty\mathcal{R}X.
	\]
	We choose the filler $\bar{h}$ as in Lemma \ref{ex filler}. Then the $\partial_{n,0}$ face of $h$ is 
	\[\lambda_m(\widetilde{x})\colon(\Box^n,\partial\Box^n)\to(\ex^\infty\mathcal{R}X,x_0),\] 
	the $\partial_{n+1,1}$ face of $h$ is 
	\[\lambda_m(\widetilde{x'})\colon(\Box^n,\partial\Box^n)\to(\ex^\infty\mathcal{R}X,x_0),\]
	and other faces of $h$ are $x_0$. By the choices of $x$ and $x'$, $h$ factors through
	\[
	h'\colon\sqcap_{n+1,0}^{n+1}\to\ex^m\mathcal{R}X
	\]
	whose $\partial_{n,0}$ face is 
	\[\widetilde{x}\colon(\Box^n,\partial\Box^n)\to(\ex^m\mathcal{R}X,x_0),\] 
	$\partial_{n+1,1}$ face is 
	\[\widetilde{x'}\colon(\Box^n,\partial\Box^n)\to(\ex^m\mathcal{R}X,x_0),\]
	and other faces are $x_0$. By Lemma \ref{ex filler}, the $\partial_{n+1,0}$ face of the filler $\bar{h}$ factors through
	\[
	\bigl(P_3^{\otimes n},\partial(P_3^{\otimes n})\bigr)\to(\ex^m\mathcal{R}X,x_0)
	\]
	whose $(2,2,\ldots,2,1)$-block is $\widetilde{x}$, $(2,2,\ldots,2,2)$-block is $\widetilde{x'}$, and other blocks are $x_0$. Let
	\[
	r\colon[3^{m+1}]\to[2\cdot3^m],\quad i\mapsto\min\{i,2\cdot3^m\},
	\]
	we have
	\[
	\Phi(x)\cdot\Phi(x')=[\bar{h}\partial_{n+1,0}]_\Box=[\lambda_{m+1}\bigl((\widetilde{x+_nx'})(\epsilon_m^{\otimes n-1}\otimes r)\bigr)]_\Box.
	\]
	On the other hand, $x+_nx'$ is of size $(3^m,\ldots,3^m,2\cdot3^m)$. Let
	\[
	r_1=r_2=\cdots=r_{n-1}\colon[3^{m+1}]\to[3^m],\quad i\mapsto\min\{i,3^m\},
	\]
	then $\Phi(x+_nx')$ is represented by
	\[
	\lambda_{m+1}\bigl((\widetilde{x+_nx'})(r_1\otimes\cdots\otimes r_{n-1}\otimes r)\bigr)\colon
	(\Box^n,\partial\Box^n)\to(\ex^\infty\mathcal{R}X,x_0).
	\]
	Here
	\[
	(\widetilde{x+_nx'})(r_1\otimes\cdots\otimes r_{n-1}\otimes r)\in(\ex^{m+1}\mathcal{R}X)_n
	\]
	corresponds under the adjunction to a spherical grid
	\[
	\bigl(P_3^{\otimes n},\partial(P_3^{\otimes n})\bigr)\to(\ex^{m}\mathcal{R}X,x_0)
	\]
	whose $(1,1,\ldots,1,1)$-block is $\widetilde{x}$, $(1,1,\ldots,1,2)$-block is $\widetilde{x'}$, and other blocks are $x_0$. The proof is finished.
\end{proof}
\begin{example}
	For $n=2$, the following grids 
	\[
	\bigl(P_3^{\otimes n},\partial(P_3^{\otimes n})\bigr)\to(\ex^{m}\mathcal{R}X,x_0)
	\]
	represent $\Phi(x)\cdot\Phi(x')$ and $\Phi(x+_nx')$ respectively.
	\begin{figure}[htbp]
		\centering
		\begin{tikzpicture}[x=5mm, y=5mm, thick]
			\foreach \a in {0,12}{
				\begin{scope}[shift={(\a,0)}]
					\foreach \i in {0,3,6,9}{
						\foreach \j in {0,3,6}{
							\draw[myarrow=.6]
							(\i,{\j+3})--(\i,\j);
						}
					}
					\foreach \j in {0,3,6,9}{
						\foreach \i in {0,3,6}{
							\draw[myarrow=.6]
							(\i,\j)--({\i+3},\j);
						}
					}
					\foreach \i in {0,3,6,9}{
						\foreach \j in {0,3,6,9}{
							\fill[black] (\i,\j) circle(2pt);
						}
					}
				\end{scope}
			}
			%
			\draw(1.5,0.9) node[above]{$x_0$};
			\draw(4.5,0.9) node[above]{$x_0$};
			\draw(7.5,0.9) node[above]{$x_0$};
			\draw(1.5,3.9) node[above]{$x_0$};
			\draw(4.5,3.9) node[above]{$\widetilde{x'}$};
			\draw(7.5,3.9) node[above]{$x_0$};
			\draw(1.5,6.9) node[above]{$x_0$};
			\draw(4.5,6.9) node[above]{$\widetilde{x}$};
			\draw(7.5,6.9) node[above]{$x_0$};
			%
			\begin{scope}[shift={(12,0)}]
				\draw(1.5,0.9) node[above]{$x_0$};
				\draw(4.5,0.9) node[above]{$x_0$};
				\draw(7.5,0.9) node[above]{$x_0$};
				\draw(1.5,3.9) node[above]{$\widetilde{x'}$};
				\draw(4.5,3.9) node[above]{$x_0$};
				\draw(7.5,3.9) node[above]{$x_0$};
				\draw(1.5,6.9) node[above]{$\widetilde{x}$};
				\draw(4.5,6.9) node[above]{$x_0$};
				\draw(7.5,6.9) node[above]{$x_0$};
			\end{scope}
		\end{tikzpicture}
	\end{figure}
\end{example}
Before we define the map in the other direction, we need the following definition and lemma. For $p\ge0$, let
\[
s^p\colon\{1,\ldots,2p\}\to\{0,1\},\quad i\mapsto(i-1\mod 2).
\]
Let $J_p=I_{s^p}$, and we denote $J_1=J$. This determines a functor
\[
\mathcal{J}\colon\cSet\to\cSet,\quad(\mathcal{J}X)_n=\mathrm{Hom}_{\cSet}(J^{\otimes n},X).
\]
Let
\[
q_0,q_1\colon J\to\Box^1,\quad q_0(0,1,2)=(0,1,1),\quad q_1(0,1,2)=(1,1,0)
\]
be the maps contracting the second and the first 1-cubes in $J$ respectively. 
\begin{definition}
	For a cubical set $X$, we define a cubical map
	\[
	\zeta_X\colon\mathcal{R}X\to\mathcal{J}X
	\]
	such that $\zeta_X(x^{s_1\cdots s_n})$ is the cubical map
	\[
	J^{\otimes n}\xrightarrow{q_{s_1}\otimes\cdots\otimes q_{s_n}}(\Box^1)^{\otimes n}\cong\Box^n\xrightarrow{x}X,
	\]
	where $x\in X_n$, $s_1,\ldots,s_n\in\{0,1\}$.
\end{definition}
\begin{example}
	Let $x\in X_2$ be a $2$-cube in a cubical set $X$. Its four faces
	are shown below.
	\begin{figure}[htbp]
		\centering
		\begin{tikzpicture}[x=4mm, y=4mm, thick, font=\small]
			\draw[myarrow=.6] (0,5)--(0,0);
			\draw[myarrow=.6] (5,5)--(5,0);
			\draw[myarrow=.6] (0,5)--(5,5);
			\draw[myarrow=.6] (0,0)--(5,0);
			\foreach \i in {0,5}{
				\foreach \j in {0,5}{
					\fill (\i,\j) circle(1.5pt);
				}
			}
			\node at (2.5,2.5) {$x$};
			\node[left]  at (0,2.5) {$x\partial_{1,0}$};
			\node[right] at (5,2.5) {$x\partial_{1,1}$};
			\node[above] at (2.5,5) {$x\partial_{2,0}$};
			\node[below] at (2.5,0) {$x\partial_{2,1}$};
		\end{tikzpicture}
	\end{figure}
	
	For $s_1,s_2\in\{0,1\}$, we have
	\[
	\zeta_X(x^{s_1s_2})
	=x(q_{s_1}\otimes q_{s_2})\colon J\otimes J\to X.
	\]
	The four grids are as below, where we denote $*=x\partial_{2,1}\partial_{1,1}\sigma_1\sigma_2$. 
	\begin{figure}[htbp]
		\centering
		\begin{tikzpicture}[x=4.5mm, y=3.5mm, thick, font=\small]
			\def\hface{x\partial_{1,1}\sigma_1}
			\def\vface{x\partial_{2,1}\sigma_2}
			\foreach \a/\b/\signs/\tl/\tr/\bl/\br in {
				0/11/00/x/\hface/\vface/*,
				10/11/10/\hface/x/*/\vface,
				0/0/01/\vface/*/x/\hface,
				10/0/11/*/\vface/\hface/x
			}{
				\begin{scope}[shift={(\a,\b)}]
					\foreach \i in {0,4,8}{
						\draw[myarrow=.6] (\i,8)--(\i,4);
						\draw[myarrow=.6] (\i,0)--(\i,4);
						\draw[myarrow=.6] (0,\i)--(4,\i);
						\draw[myarrow=.6] (8,\i)--(4,\i);
						\foreach \j in {0,4,8}{
							\fill (\i,\j) circle(1.5pt);
						}
					}
					\node at (2,6) {$\tl$};
					\node at (6,6) {$\tr$};
					\node at (2,2) {$\bl$};
					\node at (6,2) {$\br$};
					\node[below=2pt] at (4,0) {$\zeta_X(x^{\signs})$};
				\end{scope}
			}
		\end{tikzpicture}
	\end{figure}
\end{example}
\begin{lemma}
	The natural transformation $\zeta\colon\mathcal{R}\Rightarrow\mathcal{J}$ is well-defined.
\end{lemma}
\begin{proof}
	First we show that for $x\in X_{n-1}$ and $k=1,\ldots,n$, we have
	\[
	\zeta_X((x\sigma_k)^{s_1\cdots s_{k-1}0s_{k+1}\cdots s_n})=\zeta_X((x\sigma_k)^{s_1\cdots s_{k-1}1s_{k+1}\cdots s_n}).
	\]
	This follows from the commutative diagram
	\[
	\xymatrix{
		J^{\otimes n}\ar[r]\ar[d]&\Box^n\ar[d]^{\sigma_k}\\
		\Box^n\ar[r]_{\sigma_k}&\Box^{n-1},
	}
	\]
	where the top map and the left map are
	\[
	q_{s_1}\otimes\cdots\otimes q_{s_{k-1}}\otimes q_0\otimes q_{s_{k+1}}\otimes\cdots\otimes q_{s_n}\text{ and }q_{s_1}\otimes\cdots\otimes q_{s_{k-1}}\otimes q_1\otimes q_{s_{k+1}}\otimes\cdots\otimes q_{s_n}
	\]
	respectively. The compatibility with the faces follows from the commutative diagram
	\[
	\xymatrix{
		J^{\otimes n-1}\ar[rr]^{q_{s_1}\otimes\cdots\widehat{q_{s_k}}\otimes\cdots\otimes q_{s_n}}\ar[d]_{\partial_{k,\alpha}}&&\Box^{n-1}\ar[d]^{\partial_{k,\alpha+s_k}}\\
		J^{\otimes n}\ar[rr]_{q_{s_1}\otimes\cdots\otimes q_{s_n}}&&\Box^n.
	}
	\]
	The compatibility with the degeneracies and connections can be proved similarly.
\end{proof}
An element in $\pi_n(\ex^\infty\mathcal{R}X,x_0)$ is represented by a cubical map
\[
y\colon(\Box^n,\partial\Box^n)\to(\ex^M\mathcal{R}X,x_0)
\]
for some $M\ge0$, which corresponds under the adjunction to a spherical grid
\[
\bar{y}\colon\bigl(P_{3^M}^{\otimes n},\partial(P_{3^M}^{\otimes n})\bigr)\to(\mathcal{R}X,x_0).
\]
Composing with $\zeta_{X}\colon\mathcal{R}X\to\mathcal{J}X$, this corresponds to a spherical grid
\[
\psi(y)\colon\bigl(J_{3^M}^{\otimes n},\partial(J_{3^M}^{\otimes n})\bigr)\to(X,x_0),
\]
and we let $\Psi(y)\in\pi_n^\delta( X,x_0)$ be the grid homotopy class $[\psi(y)]$. 
\begin{lemma}
	The map
	\[
	\Psi\colon\pi_n(\ex^\infty\mathcal{R}X,x_0)\to\pi_n^\delta( X,x_0)
	\]
	is well-defined.
\end{lemma}
\begin{proof}
	We first show that $\Psi(y)=\Psi(\eta_{\ex^M\mathcal{R}X}(y))$. Note that
	\[
	\eta_{\ex^M\mathcal{R}X}(y)\colon(\Box^n,\partial\Box^n)\to(\ex^{M+1}\mathcal{R}X,x_0)
	\]
	corresponds under the adjunction to a spherical grid
	\[
	\bar y(\epsilon_M^{\otimes n})\colon\bigl(P_{3^{M+1}}^{\otimes n},\partial(P_{3^{M+1}}^{\otimes n})\bigr)\to(\mathcal{R}X,x_0).
	\]
	Let $r\colon J_{3^{M+1}}\to J_{3^M}$ be determined by
	\[
	r\colon[2\cdot3^{M+1}]\to[2\cdot3^M],\quad i\mapsto\min\{2\cdot3^M,\max\{0,i-2\cdot3^M\}\}
	\]
	Thus the corresponding spherical grids
	\[
	\begin{gathered}	\psi(y)\colon\bigl(J_{3^M}^{\otimes n},\partial(J_{3^M}^{\otimes n})\bigr)\to(X,x_0),\\ \psi(\eta_{\ex^M\mathcal{R}X}(y))\colon\bigl(J_{3^{M+1}}^{\otimes n},\partial(J_{3^{M+1}}^{\otimes n})\bigr)\to(X,x_0)\end{gathered}
	\]
	are 0-homotopic since
	\[
	\psi(\eta_{\ex^M\mathcal{R}X}(y))=\psi(y)(r^{\otimes n}).
	\]
	This implies that
	\[
	\Psi(y)=\Psi(\eta_{\ex^M\mathcal{R}X}(y))\in\pi_n^\delta( X,x_0). 
	\]
	Next, suppose that there is a cubical map
	\[
	z\colon(\Box^n\otimes\Box^1,\partial\Box^n\otimes\Box^1)\to(\ex^M\mathcal{R}X,x_0)
	\]
	such that $z\partial_{n+1,0}=y$ and $z\partial_{n+1,1}=y'$. This corresponds to a grid
	\[
	\bigl(J_{3^M}^{\otimes n}\otimes J_{3^M},\partial(J_{3^M}^{\otimes n})\otimes J_{3^M}\bigr)\to(X,x_0),
	\]
	and it implies that $\psi(y)$ and $\psi(y')$ are 1-homotopic. The proof is finished.
\end{proof}
\begin{lemma}
	\label{PsiPhi}
	For any $\beta\in\pi_n^\delta( X,x_0)$ we have $\Psi(\Phi(\beta))=\beta$.
\end{lemma}
\begin{proof}
	To show that
	\[
	\Psi\circ\Phi=1,
	\]
	we choose a representative
	\[
	x\colon\bigl(I_{s_1}\otimes\cdots\otimes I_{s_n},\partial(I_{s_1}\otimes\cdots\otimes I_{s_n})\bigr)\to(X,x_0)
	\]
	of $\beta\in\pi_n^\delta( X,x_0)$, where $x$ is of size $(3^m,\ldots,3^m)$. Then $\Phi(x)\in\pi_n(\ex^\infty\mathcal{R}X,x_0)$ is represented by
	\[
	\lambda_m(\widetilde{x})\colon(\Box^n,\partial\Box^n)\to(\ex^\infty\mathcal{R}X,x_0).
	\]
	Consider the composite
	\[
	\bigl(P_{3^m}^{\otimes n},\partial(P_{3^m}^{\otimes n})\bigr)\xrightarrow{\widetilde{x}}(\mathcal{R}X,x_0)\xrightarrow{\zeta_X}(\mathcal{J}X,x_0),
	\]
	which corresponds under the adjunction to a spherical grid
	\[
	\psi(\widetilde{x})\colon\bigl(J_{3^m}^{\otimes n},\partial(J_{3^m}^{\otimes n})\bigr)\to(X,x_0).
	\]
	For $k=1,\ldots,n$, let
	\[
	r_k\colon[2\cdot3^m]\to[3^m],\quad r_k(2p)=p,\quad r_k(2p-1)=p-s_k(p)
	\]
	for $p\in[3^m]$. This determines a contraction
	\[
	r_k\colon J_{3^m}\to I_{s_k},
	\]
	and we have
	\[
	\psi(\widetilde{x})=x(r_1\otimes\cdots\otimes r_n),
	\]
	implying that $\Psi(\Phi(\beta))=\beta$. The proof is finished.
	
\end{proof}
To show that
\[
\Phi\circ\Psi=1,
\]
we choose a representative
\[
y\colon(\Box^n,\partial\Box^n)\to(\ex^M\mathcal{R}X,x_0)
\]
of $\mu\in\pi_n(\ex^\infty\mathcal{R}X,x_0)$. Then $\Psi(y)\in\pi_n^\delta( X,x_0)$ is represented by
\[
\psi(y)\colon\bigl(J_{3^M}^{\otimes n},\partial(J_{3^M}^{\otimes n})\bigr)\to(X,x_0),
\]
which corresponds under the adjunction to the composite
\[
\bigl(P_{3^M}^{\otimes n},\partial(P_{3^M}^{\otimes n})\bigr)\xrightarrow{\bar{y}}(\mathcal{R}X,x_0)\xrightarrow{\zeta_X}(\mathcal{J}X,x_0),
\]
Thus it suffices to show that the spherical grids
\[
\begin{gathered}\bar{y}\colon\bigl(P_{3^M}^{\otimes n},\partial(P_{3^M}^{\otimes n})\bigr)\to(\mathcal{R}X,x_0),\\ \widetilde{\psi(y)}\colon\bigl(P_{2\cdot3^M}^{\otimes n},\partial(P_{2\cdot3^M}^{\otimes n})\bigr)\to(\mathcal{R}X,x_0)\end{gathered}
\]
are grid homotopic. For
\[
\nu=(\nu_1,\ldots,\nu_n)\in\{1,\ldots,3^M\}^n,
\]
we have
\begin{equation}
	\label{grid in RX}
	\bar{y}(\nu)=z(\nu)^{\xi_1(\nu)\cdots\xi_n(\nu)},
\end{equation}
where $z(\nu)\in X_n$, $\xi_1(\nu),\ldots,\xi_n(\nu)\in\{0,1\}$. Let
\[
p_0,p_1\colon P_2\to\Box^1,\quad p_0(0,1,2)=(0,1,1),\quad p_1(0,1,2)=(0,0,1)
\]
be the maps contracting the second and the first 1-cubes in $P_2$ respectively. Then for $\nu\in\{1,\ldots,3^M\}^n$, let
\[
\widetilde{\psi(y)}(\nu)=\bar{y}(\nu)(p_{\xi_1(\nu)}\otimes\cdots\otimes p_{\xi_n(\nu)})\colon P_2^{\otimes n}\to\mathcal{R}X.
\]
Let 
\[h_0,h_1\colon P_2\otimes\Box^1\to\Box^1,\quad h_0=\gamma_{1,1}+_1\gamma_{1,0},\quad h_1=p_1\sigma_2,\]
then we have
	\[
	h_i\partial_{1,\alpha}=\partial_{1,\alpha}\sigma_1,\quad h_0\partial_{2,1}=p_0,\quad h_0\partial_{2,0}=h_1\partial_{2,\alpha}=p_1
	\]
	for $i,\alpha\in\{0,1\}$.
	\begin{definition}
		For $i=1,\ldots,n$, let
		\[
		C_{i,n}=\begin{pmatrix}
			1&2&\cdots&i-1&i&i+1&\cdots&n\\
			1&2&\cdots&i-1&n&i&\cdots&n-1
		\end{pmatrix}\in\Sigma_n.
		\]
	\end{definition}
	For $i=1,\ldots,n$ and $\nu=(\nu_1,\ldots,\nu_n)\in\{1,\ldots,3^M\}^n$, we define
	\[
	K^{(i)}(\nu)=\bar{y}(\nu)(p_{\xi_1(\nu)}\otimes p_{\xi_2(\nu)}\otimes\cdots\otimes p_{\xi_{i-1}(\nu)}\otimes h_{\xi_i(\nu)}\otimes p_1^{\otimes n-i})\colon P_2^{\otimes i}\otimes\Box^1\otimes P_2^{\otimes n-i}\to\mathcal{R}X,
	\] 
	\[
	H^{(i)}(\nu)=K^{(i)}(\nu)C_{i+1,n+1}\colon P_2^{\otimes n}\otimes\Box^1\to\mathcal{R}X.
	\]
	Then we have
	\[
	H^{(1)}(\nu)\partial_{n+1,0}=K^{(1)}(\nu)\partial_{2,0}=\bar{y}(\nu)(p_1^{\otimes n})\colon P_2^{\otimes n}\to\mathcal{R}X,
	\]
	\[
	H^{(n)}(\nu)\partial_{n+1,1}=K^{(n)}(\nu)\partial_{n+1,1}=\bar{y}(\nu)(p_{\xi_1(\nu)}\otimes\cdots\otimes p_{\xi_n(\nu)})=\widetilde{\psi(y)}(\nu)\colon P_2^{\otimes n}\to\mathcal{R}X.
	\]
	When $\nu\in\{1,\ldots,3^M\}^n$ varies, they patch together to form
	\[
	\begin{gathered}\bar{y}(d^{\otimes n})\colon\bigl(P_{2\cdot3^M}^{\otimes n},\partial(P_{2\cdot3^M}^{\otimes n})\bigr)\to(\mathcal{R}X,x_0),\\ \widetilde{\psi(y)}\colon\bigl(P_{2\cdot3^M}^{\otimes n},\partial(P_{2\cdot3^M}^{\otimes n})\bigr)\to(\mathcal{R}X,x_0),\end{gathered}
	\]
	where $d\colon P_{2\cdot3^M}\to P_{3^M}$ is induced by
	\[
	[2\cdot3^M]\to[3^M],\quad i\mapsto\lfloor\frac{i}{2}\rfloor.
	\]
	\begin{lemma}
		\label{H welldefined 0}
		For $\nu\in\{1,\ldots,3^M\}^n$ and $i\in\{1,\ldots,n\}$, 
		\[
		K^{(i)}(\nu)\colon P_2^{\otimes i}\otimes\Box^1\otimes P_2^{\otimes n-i}\to\mathcal{R}X
		\]
		is well-defined.
	\end{lemma}
	\begin{proof}
		Suppose that 
		\[
		\bar{y}(\nu)=(z'(\nu)\sigma_k)^{\xi_1(\nu)\cdots\xi_n(\nu)}
		\]
		for $z'(\nu)\in X_{n-1}$ and $k\in\{1,\ldots,n\}$. Clearly $K^{(i)}(\nu)$ is well-defined if $k\in\{i+1,\ldots,n\}$. If $k=1,\ldots,i-1$ we have
		\begin{align*}
			&(z'(\nu)\sigma_k)^{\xi_1(\nu)\cdots\xi_{k-1}(\nu)0\xi_{k+1}(\nu)\cdots\xi_n(\nu)}\\*
			&\qquad(p_{\xi_1(\nu)}\otimes\cdots\otimes p_{\xi_{k-1}(\nu)}\otimes p_0\otimes p_{\xi_{k+1}(\nu)}\otimes\cdots\otimes p_{\xi_{i-1}(\nu)}\otimes h_{\xi_i(\nu)}\otimes p_1^{\otimes n-i})\\
			&=(z'(\nu))^{\xi_1(\nu)\cdots\xi_{k-1}(\nu)\xi_{k+1}(\nu)\cdots\xi_n(\nu)}\sigma_k\\*
			&\qquad(p_{\xi_1(\nu)}\otimes\cdots\otimes p_{\xi_{k-1}(\nu)}\otimes p_0\otimes p_{\xi_{k+1}(\nu)}\otimes\cdots\otimes p_{\xi_{i-1}(\nu)}\otimes h_{\xi_i(\nu)}\otimes p_1^{\otimes n-i})\\
			&=(z'(\nu))^{\xi_1(\nu)\cdots\xi_{k-1}(\nu)\xi_{k+1}(\nu)\cdots\xi_n(\nu)}\sigma_k\\*
			&\qquad(p_{\xi_1(\nu)}\otimes\cdots\otimes p_{\xi_{k-1}(\nu)}\otimes p_1\otimes p_{\xi_{k+1}(\nu)}\otimes\cdots\otimes p_{\xi_{i-1}(\nu)}\otimes h_{\xi_i(\nu)}\otimes p_1^{\otimes n-i})\\
			&=(z'(\nu)\sigma_k)^{\xi_1(\nu)\cdots\xi_{k-1}(\nu)1\xi_{k+1}(\nu)\cdots\xi_n(\nu)}\\*
			&\qquad(p_{\xi_1(\nu)}\otimes\cdots\otimes p_{\xi_{k-1}(\nu)}\otimes p_1\otimes p_{\xi_{k+1}(\nu)}\otimes\cdots\otimes p_{\xi_{i-1}(\nu)}\otimes h_{\xi_i(\nu)}\otimes p_1^{\otimes n-i}).
		\end{align*}
		If $k=i$, we have
		\begin{align*}
			&(z'(\nu)\sigma_i)^{\xi_1(\nu)\cdots\xi_{i-1}(\nu)0\xi_{i+1}(\nu)\cdots\xi_n(\nu)}(p_{\xi_1(\nu)}\otimes\cdots\otimes p_{\xi_{i-1}(\nu)}\otimes h_0\otimes p_1^{\otimes n-i})\\
			&=(z'(\nu))^{\xi_1(\nu)\cdots\xi_{i-1}(\nu)\xi_{i+1}(\nu)\cdots\xi_n(\nu)}\sigma_i(p_{\xi_1(\nu)}\otimes\cdots\otimes p_{\xi_{i-1}(\nu)}\otimes h_0\otimes p_1^{\otimes n-i})\\
			&=(z'(\nu))^{\xi_1(\nu)\cdots\xi_{i-1}(\nu)\xi_{i+1}(\nu)\cdots\xi_n(\nu)}\sigma_i(p_{\xi_1(\nu)}\otimes\cdots\otimes p_{\xi_{i-1}(\nu)}\otimes h_1\otimes p_1^{\otimes n-i})\\
			&=(z'(\nu)\sigma_i)^{\xi_1(\nu)\cdots\xi_{i-1}(\nu)1\xi_{i+1}(\nu)\cdots\xi_n(\nu)}(p_{\xi_1(\nu)}\otimes\cdots\otimes p_{\xi_{i-1}(\nu)}\otimes h_1\otimes p_1^{\otimes n-i}).
		\end{align*}
		The second equality is due to the commutative diagram
		\[
		\xymatrix{
			P_2\otimes\Box^1\ar[r]^{h_0}\ar[d]_{h_1}&\Box^1\ar[d]^{\sigma_1}\\
			\Box^1\ar[r]_{\sigma_1}&\Box^0,
		}
		\]
		which can be checked easily. The proof is finished.
	\end{proof}
	\begin{lemma}
		\label{patching}
		The family
		\[
		\{H^{(i)}(\nu)\colon P_{2}^{\otimes n}\otimes\Box^1\to\mathcal{R}X\}_{\nu\in\{1,\ldots,3^M\}^n,i\in\{1,\ldots,n\}}
		\]
		patches together to a cubical map
		\[
		H\colon\bigl(P_{2\cdot3^M}^{\otimes n}\otimes P_n,\partial(P_{2\cdot3^M}^{\otimes n})\otimes P_n\bigr)\to(\mathcal{R}X,x_0)
		\]
		such that
		\[
		H\partial_{i,\alpha}=\begin{cases}
			\bar{y}(d^{\otimes n})&((i,\alpha)=(n+1,0))\\
			\widetilde{\psi(y)}&((i,\alpha)=(n+1,1)).
		\end{cases}
		\]
	\end{lemma}
	\begin{proof}
		For $k=1,\ldots,n$, let $e_k\in\Z^n$ such that
		\[
		(e_k)_j=\delta_{k,j}.
		\]
		In order to prove the lemma, it suffices to prove that
		\begin{equation}
			\label{H welldefined 1}
			H^{(i)}(\nu)\partial_{n+1,1}=H^{(i+1)}(\nu)\partial_{n+1,0}
		\end{equation}
		for $\nu\in\{1,\ldots,3^M\}^n$, $i=1,\ldots,n-1$,
		\begin{equation}
			\label{H welldefined 2}
			H^{(i)}(\nu)\partial_{k,1}=H^{(i)}(\nu+e_k)\partial_{k,0}
		\end{equation}
		for $\nu\in\{1,\ldots,3^M\}^{n}$ such that $\nu_k<3^M$, $i,k=1,\ldots,n$, 
		\begin{equation}
			\label{H welldefined 3}
			H^{(i)}(\nu)\partial_{k,0}=x_0
		\end{equation}
		for $\nu\in\{1,\ldots,3^M\}^{n}$ such that $\nu_k=1$, $i,k=1,\ldots,n$ and
		\begin{equation}
			\label{H welldefined 4}
			H^{(i)}(\nu)\partial_{k,1}=x_0
		\end{equation}
		for $\nu\in\{1,\ldots,3^M\}^{n}$ such that $\nu_k=3^M$, $i,k=1,\ldots,n$. To prove \eqref{H welldefined 1}, we have
		\enlargethispage{-\baselineskip}
		\begin{align*}
			&H^{(i)}(\nu)\partial_{n+1,1}\\
			&=K^{(i)}(\nu)C_{i+1,n+1}\partial_{n+1,1}\\
			&=\bar{y}(\nu)(p_{\xi_1(\nu)}\otimes p_{\xi_2(\nu)}\otimes\cdots\otimes p_{\xi_{i-1}(\nu)}\otimes h_{\xi_i(\nu)}\otimes p_1^{\otimes n-i})\partial_{i+1,1}\\
			&=\bar{y}(\nu)(p_{\xi_1(\nu)}\otimes p_{\xi_2(\nu)}\otimes\cdots\otimes p_{\xi_{i-1}(\nu)}\otimes p_{\xi_i(\nu)}\otimes p_1^{\otimes n-i})\\
			&=\bar{y}(\nu)(p_{\xi_1(\nu)}\otimes p_{\xi_2(\nu)}\otimes\cdots\otimes p_{\xi_{i}(\nu)}\otimes h_{\xi_{i+1}(\nu)}\otimes p_1^{\otimes n-i-1})\partial_{i+2,0}\\
			&=K^{(i+1)}(\nu)C_{i+2,n+1}\partial_{n+1,0}\\
			&=H^{(i+1)}(\nu)\partial_{n+1,0}.
		\end{align*}
		To prove \eqref{H welldefined 2}, we have
		\begin{align}
			\notag&H^{(i)}(\nu)\partial_{k,1}\\
			\notag&=K^{(i)}(\nu)C_{i+1,n+1}\partial_{k,1}\\
			\notag&=\begin{cases}
				\begin{aligned}[t]
					&\bar{y}(\nu)(p_{\xi_1(\nu)}\otimes p_{\xi_2(\nu)}\otimes\cdots\otimes p_{\xi_{i-1}(\nu)}\\
					&\qquad{}\otimes h_{\xi_i(\nu)}\otimes p_1^{\otimes n-i})\partial_{k,1}C_{i,n}
				\end{aligned}&(k=1,\ldots,i)\\
				\begin{aligned}[t]
					&\bar{y}(\nu)(p_{\xi_1(\nu)}\otimes p_{\xi_2(\nu)}\otimes\cdots\otimes p_{\xi_{i-1}(\nu)}\\
					&\qquad{}\otimes h_{\xi_i(\nu)}\otimes p_1^{\otimes n-i})\partial_{k+1,1}C_{i+1,n}
				\end{aligned}&(k=i+1,\ldots,n)
			\end{cases}\displaybreak[1]\\
			\label{H welldefined 2.1}&=\begin{cases}
				\begin{aligned}[t]
					&\bar{y}(\nu)\partial_{k,1}(p_{\xi_1(\nu)}\otimes p_{\xi_2(\nu)}\otimes\cdots\otimes\widehat{p_{\xi_k(\nu)}}\\
					&\qquad{}\otimes\cdots\otimes p_{\xi_{i-1}(\nu)}\otimes h_{\xi_i(\nu)}\otimes p_1^{\otimes n-i})C_{i,n}
				\end{aligned}&(k=1,\ldots,i-1)\\
				\begin{aligned}[t]
					&\bar{y}(\nu)\partial_{k,1}(p_{\xi_1(\nu)}\otimes p_{\xi_2(\nu)}\otimes\cdots\\
					&\qquad{}\otimes p_{\xi_{i-1}(\nu)}\otimes \sigma_1\otimes p_1^{\otimes n-i})C_{i,n}
				\end{aligned}&(k=i)\\
				\begin{aligned}[t]
					&\bar{y}(\nu)\partial_{k,1}(p_{\xi_1(\nu)}\otimes p_{\xi_2(\nu)}\otimes\cdots\\
					&\qquad{}\otimes p_{\xi_{i-1}(\nu)}\otimes h_{\xi_i(\nu)}\otimes p_1^{\otimes n-i-1})C_{i+1,n}
				\end{aligned}&(k=i+1,\ldots,n).
			\end{cases}
		\end{align}
		Similarly, we have
		\begin{align}
			\notag&H^{(i)}(\nu+e_k)\partial_{k,0}\\
			\notag&=K^{(i)}(\nu+e_k)C_{i+1,n+1}\partial_{k,0}\\
			\notag&=\begin{cases}
				\begin{aligned}[t]
					&\bar{y}(\nu+e_k)(p_{\xi_1({\nu+e_k})}\otimes p_{\xi_2({\nu+e_k})}\otimes\cdots\\
					&\qquad{}\otimes p_{\xi_{i-1}({\nu+e_k})}\otimes h_{\xi_i({\nu+e_k})}\otimes p_1^{\otimes n-i})\partial_{k,0}C_{i,n}
				\end{aligned}&(k=1,\ldots,i)\\
				\begin{aligned}[t]
					&\bar{y}(\nu+e_k)(p_{\xi_1({\nu+e_k})}\otimes p_{\xi_2({\nu+e_k})}\otimes\cdots\otimes p_{\xi_{i-1}({\nu+e_k})}\\
					&\qquad{}\otimes h_{\xi_i({\nu+e_k})}\otimes p_1^{\otimes n-i})\partial_{k+1,0}C_{i+1,n}
				\end{aligned}&(k=i+1,\ldots,n)
			\end{cases}\displaybreak[1]\\
			\label{H welldefined 2.2}&=\begin{cases}
				\begin{aligned}[t]
					&\bar{y}(\nu+e_k)\partial_{k,0}(p_{\xi_1({\nu+e_k})}\otimes p_{\xi_2({\nu+e_k})}\otimes\cdots\otimes\widehat{p_{\xi_k({\nu+e_k})}}\\
					&\qquad{}\otimes\cdots\otimes p_{\xi_{i-1}({\nu+e_k})}\otimes h_{\xi_i({\nu+e_k})}\otimes p_1^{\otimes n-i})C_{i,n}
				\end{aligned}&(k=1,\ldots,i-1)\\
				\begin{aligned}[t]
					&\bar{y}(\nu+e_k)\partial_{k,0}(p_{\xi_1({\nu+e_k})}\otimes p_{\xi_2({\nu+e_k})}\otimes\cdots\\
					&\qquad{}\otimes p_{\xi_{i-1}({\nu+e_k})}\otimes \sigma_1\otimes p_1^{\otimes n-i})C_{i,n}
				\end{aligned}&(k=i)\\
				\begin{aligned}[t]
					&\bar{y}(\nu+e_k)\partial_{k,0}(p_{\xi_1({\nu+e_k})}\otimes p_{\xi_2({\nu+e_k})}\otimes\cdots\\
					&\qquad{}\otimes p_{\xi_{i-1}({\nu+e_k})}\otimes h_{\xi_i({\nu+e_k})}\otimes p_1^{\otimes n-i-1})C_{i+1,n}
				\end{aligned}&(k=i+1,\ldots,n).
			\end{cases}
		\end{align}
		Since 
		\[\bar{y}(\nu)\partial_{k,1}=\bar{y}(\nu+e_k)\partial_{k,0},\]
		we may identify \eqref{H welldefined 2.1} and \eqref{H welldefined 2.2} by arguing similarly to the proof of Lemma \ref{H welldefined 0}. To prove \eqref{H welldefined 3}, we have
		\begin{align}
			\notag&H^{(i)}(\nu)\partial_{k,0}\\
			\notag&=K^{(i)}(\nu)C_{i+1,n+1}\partial_{k,0}\\
			\notag&=\begin{cases}
				\begin{aligned}[t]
					&\bar{y}(\nu)(p_{\xi_1(\nu)}\otimes p_{\xi_2(\nu)}\otimes\cdots\otimes p_{\xi_{i-1}(\nu)}\\
					&\qquad{}\otimes h_{\xi_i(\nu)}\otimes p_1^{\otimes n-i})\partial_{k,0}C_{i,n}
				\end{aligned}&(k=1,\ldots,i)\\
				\begin{aligned}[t]
					&\bar{y}(\nu)(p_{\xi_1(\nu)}\otimes p_{\xi_2(\nu)}\otimes\cdots\otimes p_{\xi_{i-1}(\nu)}\\
					&\qquad{}\otimes h_{\xi_i(\nu)}\otimes p_1^{\otimes n-i})\partial_{k+1,0}C_{i+1,n}
				\end{aligned}&(k=i+1,\ldots,n)
			\end{cases}\displaybreak[1]\\
			\notag&=\begin{cases}
				\begin{aligned}[t]
					&\bar{y}(\nu)\partial_{k,0}(p_{\xi_1(\nu)}\otimes p_{\xi_2(\nu)}\otimes\cdots\otimes\widehat{p_{\xi_k(\nu)}}\\
					&\qquad{}\otimes\cdots\otimes p_{\xi_{i-1}(\nu)}\otimes h_{\xi_i(\nu)}\otimes p_1^{\otimes n-i})C_{i,n}
				\end{aligned}&(k=1,\ldots,i-1)\\
				\begin{aligned}[t]
					&\bar{y}(\nu)\partial_{k,0}(p_{\xi_1(\nu)}\otimes p_{\xi_2(\nu)}\otimes\cdots\\
					&\qquad{}\otimes p_{\xi_{i-1}(\nu)}\otimes \sigma_1\otimes p_1^{\otimes n-i})C_{i,n}
				\end{aligned}&(k=i)\\
				\begin{aligned}[t]
					&\bar{y}(\nu)\partial_{k,0}(p_{\xi_1(\nu)}\otimes p_{\xi_2(\nu)}\otimes\cdots\\
					&\qquad{}\otimes p_{\xi_{i-1}(\nu)}\otimes h_{\xi_i(\nu)}\otimes p_1^{\otimes n-i-1})C_{i+1,n}
				\end{aligned}&(k=i+1,\ldots,n)
			\end{cases}\displaybreak[1]\\
			\notag&=x_0
		\end{align}
		as $\bar{y}(\nu)\partial_{k,0}=x_0$ whenever $\nu_k=1$. The proof of \eqref{H welldefined 4} is analogous, and the proof is finished.
	\end{proof}
	Now we are ready to prove the main result.
	\begin{proof}
		[Proof of Proposition \ref{DHG main}]
		By Lemma \ref{PsiPhi}, it suffices to show that $\Phi\circ\Psi=1$. For $\mu\in\pi_n(\ex^\infty\mathcal{R}X,x_0)$, we choose a representative
		\[
		y\colon(\Box^n,\partial\Box^n)\to(\ex^M\mathcal{R}X,x_0)
		\]
		of $\mu$, which corresponds under the adjunction to
		\[
		\bar{y}\colon\bigl(P_{3^M}^{\otimes n},\partial(P_{3^M}^{\otimes n})\bigr)\to(\mathcal{R}X,x_0).
		\]
		Then $\Psi(y)\in\pi_n^\delta( X,x_0)$ is represented by the spherical grid
		\[
		\psi(y)\colon\bigl(J_{3^M}^{\otimes n},\partial(J_{3^M}^{\otimes n})\bigr)\to(X,x_0).
		\]
		Let 
		\[
		r\colon[3^{M+1}]\to[2\cdot3^M],\quad i\mapsto\min\{i,2\cdot3^M\},
		\]
		then $\Phi(\Psi(y))\in\pi_n(\ex^\infty\mathcal{R}X,x_0)$ is represented by
		\[
		\lambda_{M+1}(\widetilde{\psi(y)}(r^{\otimes n})),
		\]
		where 
		\[\widetilde{\psi(y)}(r^{\otimes n})\colon(\Box^n,\partial\Box^n)\to(\ex^{M+1}\mathcal{R}X,x_0)\]
		corresponds under the adjunction to the spherical grid
		\[
		\bigl(P_{3^{M+1}}^{\otimes n},\partial(P_{3^{M+1}}^{\otimes n})\bigr)\xrightarrow{r^{\otimes n}}\bigl(P_{2\cdot3^{M}}^{\otimes n},\partial(P_{2\cdot3^{M}}^{\otimes n})\bigr)\xrightarrow{\widetilde{\psi(y)}}(\mathcal{R}X,x_0).
		\]
		To prove that $\mu=\Phi(\Psi(y))$, choose $M$ big enough so that $n\le3^{M+1}$. Let
		\[
		r'\colon[3^{M+1}]\to[n],\quad i\mapsto\min\{i,n\},
		\]
		then Lemma \ref{patching} implies that there is a map
		\[
		\bigl(P_{3^{M+1}}^{\otimes n}\otimes P_{3^{M+1}},\partial(P_{3^{M+1}}^{\otimes n})\otimes P_{3^{M+1}}\bigr)\xrightarrow{r^{\otimes n}\otimes r'}\bigl(P_{2\cdot3^M}^{\otimes n}\otimes P_n,\partial(P_{2\cdot3^M}^{\otimes n})\otimes P_n\bigr)\xrightarrow{H}(\mathcal{R}X,x_0).
		\]
		Therefore $\bar{y}((d\circ r)^{\otimes n})$ and $\widetilde{\psi(y)}(r^{\otimes n})$ are cubically homotopic as maps
		\[
		(\Box^n,\partial\Box^n)\to(\ex^{M+1}\mathcal{R}X,x_0).
		\]
		By Lemma \ref{cubical homotopy vs grid homotopy}, $\bar{y}((d\circ r)^{\otimes n})$ and $\bar{y}((\epsilon_M)^{\otimes n})$ are cubically homotopic, and the latter grid corresponds under the adjunction to
		\[
		\eta_{\ex^M\mathcal{R}X}(y)\colon(\Box^n,\partial\Box^n)\to(\ex^{M+1}\mathcal{R}X,x_0).
		\]
		Thus we have
		\begin{align*}
			\mu&=[\lambda_M(y)]_\Box=[\lambda_{M+1}(\eta_{\ex^M\mathcal{R}X}(y))]_\Box\\
			&=[\lambda_{M+1}(\bar{y}((d\circ r)^{\otimes n}))]_\Box\\
			&=[\lambda_{M+1}(\widetilde{\psi(y)}(r^{\otimes n}))]_\Box=\Phi(\Psi(\mu)),
		\end{align*}
		and the proof is complete.
	\end{proof}
	We are ready to prove Theorem \ref{main 1}.
	
	\begin{proof}
		[Proof of Theorem \ref{main 1}]
		The proof follows from the following natural isomorphisms
		\[
		\pi_n^\delta(X,x_0)\cong\pi_n(\ex^\infty\mathcal{R}X,x_0)\cong\pi_n(|\ex^\infty\mathcal{R}X|,x_0)\cong\pi_n(|X|,x_0),
		\]
		where the first is due to Proposition \ref{DHG main}, the second is due to Theorem \ref{homotopy group realization}, and the third is due to Propositions \ref{R} and \ref{cubical Ex}.
	\end{proof}

	Having established the isomorphism between discrete and topological homotopy groups, we now construct the discrete Hurewicz map by assigning signs to the constituent cubes of a grid. Let $X$ be a cubical set and $n,M\ge0$. For $y\in(\ex^M\mathcal{R}X)_n$,	set
	\[
	f(y)=\sum_{\nu\in\{1,\ldots,3^M\}^n}(-1)^{\xi(\nu)}z(\nu)\in N_n(X).
	\]
	Here the notation is as in \eqref{grid in RX} and $\xi(\nu)=\xi_1(\nu)+\cdots+\xi_n(\nu)$. This is independent of the representatives in \eqref{grid in RX}, since the identification \eqref{R identification} changes a sign only on a degenerate cube, which is zero in $N_n(X)$.
	\begin{lemma}
		\label{lambdaiota inverse}
		Let $X$ be a cubical set and $n\ge0$. The map
		\[
		f_*\colon N_n(\ex^\infty\mathcal{R}X)\to N_n(X),\quad\lambda_My\mapsto f(y)
		\]
		for $n\ge0$ is a well-defined chain map and an isomorphism on homology.
	\end{lemma}
	\begin{proof}
		Since degenerate cubes are trivial in $N_n(X)$, we have 
		\[
		f(y)=f(\eta_{\ex^M\mathcal{R}X}(y)).
		\]
		If $y=y'\sigma_k$ for some $k=1,\ldots,n$ and $y'\in(\ex^M\mathcal{R}X)_{n-1}$, we have $f(y)=0$, which implies that
		\[
		f_*\colon N_n(\ex^\infty\mathcal{R}X)\to N_n(X),\quad\lambda_My\mapsto f(y)
		\]
		is well-defined. Since
		\[
		(z(\nu)\partial_{k,1-\xi_k(\nu)})^{\xi_1(\nu)\cdots\xi_{k-1}(\nu)\xi_{k+1}(\nu)\cdots\xi_n(\nu)}=(z(\nu+e_k)\partial_{k,\xi_k(\nu+e_k)})^{\xi_1(\nu+e_k)\cdots\xi_{k-1}(\nu+e_k)\xi_{k+1}(\nu+e_k)\cdots\xi_n(\nu+e_k)}
		\]
		for $k=1,\ldots,n$, $\nu_k<3^M$, one has
		\[
		(-1)^{\xi(\nu)+\xi_k(\nu)}z(\nu)\partial_{k,1-\xi_k(\nu)}=(-1)^{\xi(\nu+e_k)+\xi_k(\nu+e_k)}z(\nu+e_k)\partial_{k,\xi_k(\nu+e_k)}
		\]
		in $N_{n-1}(X)$. Hence
		\begin{align*}
			&\partial f(y)\\
			&=\sum_{k=1}^{n}\sum_{\nu}(-1)^{k+\xi(\nu)}(z(\nu)\partial_{k,0}-z(\nu)\partial_{k,1})\\
			&=\sum_{k=1}^{n}(\sum_{\nu_k=1}(-1)^{k+\xi(\nu)+\xi_k(\nu)}z(\nu)\partial_{k,\xi_k(\nu)}-\sum_{\nu_k=3^M}(-1)^{k+\xi(\nu)+\xi_k(\nu)}z(\nu)\partial_{k,\xi_k(\nu)+1})\\
			&=\sum_{k=1}^{n}(-1)^k(f(y\partial_{k,0})-f(y\partial_{k,1}))\\
			&=f(\partial y).
		\end{align*}
		Therefore $f_*$ is a chain map and it induces a map $f_*\colon H_n(\ex^\infty\mathcal{R}X)\to H_n(X)$ on homology. Note that $(\lambda\circ\iota)_*\colon H_n(X)\to H_n(\ex^\infty\mathcal{R}X)$ is induced by
		\[
		N_n(X)\to N_n(\ex^\infty\mathcal{R}X),\quad x\mapsto\lambda(x^{00\cdots0})
		\]
		for $x\in X_n$. Thus $f_*\circ(\lambda\circ\iota)_*$ is the identity map on $H_n(X)$. Since $(\lambda\circ\iota)_*$ is an isomorphism by Propositions \ref{R} and \ref{cubical Ex}, $f_*\colon H_n(\ex^\infty\mathcal{R}X)\to H_n(X)$ is the inverse of $(\lambda\circ\iota)_*$, hence also an isomorphism. The proof is completed.
	\end{proof}
	For an $n$-grid 
	\[x\colon I_{s_1}\otimes\cdots\otimes I_{s_n}\to X,\] 
	set
	\[
	\mathrm{sgn}_{i_1,\ldots,i_n}(x)=(-1)^{s_1(i_1)+\cdots+s_n(i_n)}.
	\]
	The sign records the number of coordinate directions traversed in reverse. It is precisely this sign that makes the contributions from internal faces cancel in pairs. Define
	\[
	h(x)=\sum_{i_1,\ldots,i_n}\mathrm{sgn}_{i_1,\ldots,i_n}(x)x(i_1,\ldots,i_n)\in N_n(X).
	\]
	Now we prove Theorem \ref{main 3}, assuming that $X$ is quasisymmetric and with a basepoint $x_0$. 
	\begin{proof}
		[Proof of Theorem \ref{main 3}]
		Carranza, Kapulkin and Tonks \cite{CKT} proved that if $(Y,y_0)$ is a pointed Kan complex, then the map
		\[
		\bar{h}\colon\pi_n(Y,y_0)\to H_n(Y),\quad[y]_\Box\mapsto[y]
		\]
		is a well-defined homomorphism such that if $(Y,y_0)$ is $(n-1)$-connected, then this map is the abelianization for $n=1$ and is an isomorphism for $n\ge 2$. Consider the composite
		\begin{equation}
			\label{composite}
			\pi_n^\delta(X,x_0)\xrightarrow{\Phi}\pi_n(\ex^\infty\mathcal{R}X,x_0)\xrightarrow{\bar{h}}H_n(\ex^\infty\mathcal{R}X)\xrightarrow{f_*}H_n(X),
		\end{equation}
		where the first map is given by Definition \ref{Phi def}. For a spherical $n$-grid 
		\[x\colon (I_{s_1}\otimes\cdots\otimes I_{s_n},\partial(I_{s_1}\otimes\cdots\otimes I_{s_n}))\to (X,x_0),\]
		we have
		\[
		f_*(\bar{h}(\Phi(x)))=f_*[\lambda_m(\widetilde{x}(r_1\otimes\cdots\otimes r_n))]=[h(x)],
		\]
		and the map
		\[
		h_*\colon \pi_n^\delta(X,x_0)\to H_n(X),\quad[x]\mapsto[h(x)]=\left[\sum_{i_1,\ldots,i_n}
		\mathrm{sgn}_{i_1,\ldots,i_n}(x)x(i_1,\ldots,i_n)\right]
		\]
		is a well-defined homomorphism. Since $\Phi$ and $f_*$ in the composite \eqref{composite} are isomorphisms by Proposition \ref{DHG main} and Lemma \ref{lambdaiota inverse} respectively, the claim follows.
	\end{proof}
	\begin{remark}
		For a pointed cubical set $(X,x_0)$, the map
		\[
		h_*\colon \pi_n^\delta(X,x_0)\to H_n(X),\quad[x]\mapsto[h(x)]=\left[\sum_{i_1,\ldots,i_n}
		\mathrm{sgn}_{i_1,\ldots,i_n}(x)x(i_1,\ldots,i_n)\right]
		\]
		is a well-defined homomorphism even when $X$ is not quasisymmetric. However, the name "discrete Hurewicz map" is justified by Theorem \ref{main 3}, where $X$ is quasisymmetric. At this moment, the authors do not know whether we have a Hurewicz theorem for this map in the general case.
	\end{remark}

	As shown in Section \ref{Discrete homotopy groups of cubical sets}, discrete homotopy groups can be defined for an arbitrary cubical set $X$, and they enjoy properties analogous to those of the homotopy groups of a topological space. Their usefulness has been demonstrated in Theorems \ref{main 1} and \ref{main 3} in the case where $X$ is quasisymmetric. However, these groups remain rather mysterious when $X$ is not cubically homotopy equivalent to a quasisymmetric cubical set. We believe that this provides a rich source of problems for future research, and we invite the reader to consider the following questions.
	
	\begin{problem}
		Find a cubical set $X$ which is not cubically homotopy equivalent to a quasisymmetric cubical set, such that $X$ has nontrivial higher discrete homotopy groups.
	\end{problem}
	
	We expect that a Seifert--van Kampen theorem for the discrete fundamental group can be proved in the same manner as the classical Seifert--van Kampen theorem. In particular, any group should be realizable as the discrete fundamental group of some non-quasisymmetric cubical set. At present, however, we do not know any examples of non-quasisymmetric cubical sets with nontrivial higher discrete homotopy groups. The following problem is also natural.
	
	\begin{problem}
		Prove or disprove that a weak equivalence between cubical sets induces isomorphisms on discrete homotopy groups. If not, describe the kernel and cokernel of the induced map.
	\end{problem}
	
	By Corollary \ref{dhg under he}, a pointed homotopy equivalence between pointed cubical sets induces isomorphisms on discrete homotopy groups. An affirmative answer to the preceding problem would therefore, together with Theorem \ref{main 2}, provide a way to understand the homotopy groups of the geometric realization of an arbitrary cubical set in terms of discrete homotopy groups. Finally, we would like to understand whether the discrete homotopy groups of an arbitrary cubical set carry additional structures analogous to those carried by the homotopy groups of a topological space.
	
	\begin{problem}
		Construct homotopy operations on discrete homotopy groups in a purely combinatorial manner.
	\end{problem}
	
	The most elementary and important such operation is the action of the fundamental group on the higher homotopy groups, which leads to the following problem.
	
	\begin{problem}
		Construct combinatorially the action of $\pi_1^\delta(X,x_0)$ on $\pi_*^\delta(X,x_0)$.
	\end{problem}
	
	It is straightforward to define such actions in degrees $1$ and $2$. However, for degrees $*\ge 3$, it is unclear how to proceed for a non-quasisymmetric cubical set, because coordinate-permutation symmetries are not available.
	
	\section{Discrete homotopy groups of directed graphs}\label{Discrete homotopy groups of directed graphs}
	
	In this section, we recall basic notions concerning directed graphs, including maps, products, and homotopies. We then define the 1-nerve of a directed graph and define the discrete homotopy groups and the homology groups of a pointed directed graph to be those of its 1-nerve, extending the discrete homotopy groups of graphs introduced in \cite{BBdLL,BKLW}. Since the 1-nerve is symmetric, one has the Hurewicz theorem between the discrete homotopy groups and the homology groups. To understand the discrete Hurewicz map, we show that the discrete homotopy groups are naturally isomorphic to the homotopy groups of directed graphs in the sense of Li, Wu, Yau and Zhang \cite{LWYZ}, and the homology groups coincide with the singular cubical homology groups of a directed graph introduced in \cite{GM2}. By embedding the category of undirected graphs into the category of directed graphs, we show that the Hurewicz theorem for directed graphs is a generalization of the Hurewicz theorem for undirected graphs proved in \cite{CK2}.

	\subsection{Directed graphs}
	
	We begin by defining directed graphs.

	\begin{definition}
		A \emph{directed graph} $G$ consists of a vertex set $V(G)$ and an edge set $E(G)\subset V(G)\times V(G)-\Delta$, where $\Delta$ denotes the diagonal, and an edge $(x,y)\in E(G)$ is regarded as being directed from $x$ to $y$.
	\end{definition}

	By definition, directed graphs do not admit loops or multiple edges with the same direction. However, pairs of oppositely directed edges are allowed. A \emph{subgraph} of a directed graph $G$ is a directed graph whose vertex and edge sets are subsets of those of $G$. 

	\begin{example}
		For $l\ge0$ and a function
		\[
		s\colon\{1,\ldots,l\}\to\{0,1\},
		\]
		a \emph{path graph} of length $l$ is a directed graph $I_s$ with vertex set $V(I_s)=[l]=\{0,\ldots,l\}$ such that 
		\[
		E(I_s)=\{(i-1,i)\}_{s(i)=0}\sqcup\{(i,i-1)\}_{s(i)=1}
		\]
		for $i=1,\ldots,l$. Let $\mathcal{I}_l$ denote the set of all path graphs of length $l$.
	\end{example}

	\begin{definition}
		A \emph{map} $f\colon G\to H$ between directed graphs $G$ and $H$ is a map $f\colon V(G)\to V(H)$ such that  whenever $(x,y)\in E(G)$, either $f(x)=f(y)$ or $(f(x),f(y))\in E(H)$. We denote $f_V\colon V(G)\to V(H)$ for the map between vertex sets.
	\end{definition}

	The identity map of a directed graph and the composites of maps between directed graphs are again maps. Thus directed graphs together with maps between them form a category, which we denote by $\mathsf{DiGraph}$.

	\begin{definition}
		Define a functor
		\[
		\otimes\colon\mathsf{DiGraph}\times\mathsf{DiGraph}\to\mathsf{DiGraph}
		\]
		by setting $V(G\otimes H)=V(G)\times V(H)$ and $E(G\otimes H)$ to be the set of all elements $((x_0,y_0),(x_1,y_1))\in V(G\otimes H)\times V(G\otimes H)$ satisfying either $x_0=x_1$ and $(y_0,y_1)\in E(H)$, or $(x_0,x_1)\in E(G)$ and $y_0=y_1$.
	\end{definition}

	This is the box product of directed graphs. The functor $\otimes$ defines a monoidal product on the category $\mathsf{DiGraph}$.

	\begin{example}
		The directed graph $I_{s_1}\otimes\cdots\otimes I_{s_n}$, where $I_{s_i}\in\mathcal{I}_{l_i}$ ($i=1,\ldots,n$), is called an \emph{$n$-grid graph}. Its \emph{boundary} is a subgraph
		\[
		\partial(I_{s_1}\otimes\cdots\otimes I_{s_n})=\bigcup_{i=1}^nI_{s_1}\otimes\cdots\otimes I_{s_{i-1}}\otimes\{0,l_i\}\otimes I_{s_{i+1}}\otimes\cdots\otimes I_{s_n}.
		\]
	\end{example}
	For a directed graph $G$, a map $I_{s_1}\otimes\cdots\otimes I_{s_n}\to G$ is called an \emph{$n$-grid} in $G$.

	\begin{definition}
		Maps $f,g\colon G\to H$ between directed graphs are said to be \emph{1-homotopic} if there exists a path graph $I\in\mathcal{I}_l$ for some $l$ and a map $h\colon G\otimes I\to H$ such that $h(-,0)=f$ and $h(-,l)=g$. In this case, we write $f\simeq_1 g$.
	\end{definition}
	It is straightforward to verify that 1-homotopies define an equivalence relation on the set of maps between directed graphs. A pair of directed graphs $(G,H)$ consists of a directed graph $G$ together with a subgraph $H\subset G$. Maps and homotopies between pairs are defined in the obvious manner. Any vertex $v\in V(G)$ may be regarded as a subgraph of a directed graph $G$. A pair $(G,v)$, where $v$ is a vertex of $G$, is called a \emph{pointed directed graph}. Pointed maps and pointed homotopies are defined as maps and homotopies of pairs.
	\subsection{1-nerve of a directed graph}
	We define the 1-nerve of a directed graph and show that it is symmetric. Let $\vec{I}\in\mathcal{I}_1$ such that $E(\vec{I})=\{(0,1)\}$.
	\begin{definition}
		\label{1-nerve}
		For a directed graph $G$, we define the \emph{1-nerve} $\mathrm{N}_1G$ to be the cubical set such that
		\[
		(\mathrm{N}_1G)_n=\mathsf{DiGraph}(\vec{I}^{\otimes n},G)
		\]
		for $n\ge0$. Thus an $n$-cube of $\mathrm{N}_1G$ is precisely a directed $n$-dimensional cube in $G$, with edges allowed to collapse to vertices. 
	\end{definition}
	It is straightforward to check that $\mathrm{N}_1G$ is indeed a cubical set. We recall the definition of the singular cubical homology groups of a directed graph $G$, which were introduced in \cite{GM2}. The undirected version of this homology theory was used in \cite{CK2}.
	\begin{definition}
		For a directed graph $G$ and $n\ge0$, let $Q_n(G)$ be the free abelian group generated by all maps $f\colon\vec{I}^{\otimes n}\to G$. The homomorphism
		\[
		\partial f=\sum_{k=1}^n(-1)^k(f\partial_{k,0}-f\partial_{k,1})\in Q_{n-1}(G)
		\]
		provides a chain complex structure on $Q_*(G)=\{Q_n(G)\}_{n\ge0}$. Let $B_n(G)$ denote the subgroup of $Q_n(G)$ generated by all degenerate $n$-cubes of $\mathrm{N}_1(G)$. Then $B_*(G)=\{B_n(G)\}_{n\ge0}$ is a chain subcomplex of $Q_*(G)$ and the $n$-th \emph{singular cubical homology group} $\mathrm{QH}_n(G)$ is defined to be the $n$-th homology group of the chain complex $Q_*(G)/B_*(G)$.
	\end{definition}
	The following is immediate from the definition.
	\begin{lemma}
		\label{two homology groups}
		For a directed graph $G$ and $n\ge0$, we have
		\[
		H_n(\mathrm{N}_1G)=\mathrm{QH}_n(G).
		\]
	\end{lemma} 
	Moreover, we show that $\mathrm{N}_1G$ is symmetric.
	\begin{lemma}
		\label{1-nerve symmetric}
		For a directed graph $G$, the 1-nerve $\mathrm{N}_1G$ is a symmetric cubical set.
	\end{lemma}
	\begin{proof}
		For $f\in(\mathrm{N}_1G)_n=\mathsf{DiGraph}(\vec{I}^{\otimes n},G)$ and $\kappa\in\Sigma_n$, we let
		\[
		f\kappa=f\circ\kappa\colon\vec{I}^{\otimes n}\to\vec{I}^{\otimes n}\to G.
		\]
		It is straightforward to verify the identities \eqref{permutation identity 1} and \eqref{permutation identity 2}, which implies that $\mathrm{N}_1G$ is symmetric.
	\end{proof}
	Thus coordinate permutations of the directed cube induce the required symmetries on the 1-nerve. If $G$ is pointed with basepoint $v$, we define the \emph{$n$-th discrete homotopy group} $\pi_n^\delta(G,v)$ to be $\pi_n^\delta(\mathrm{N}_1G,v)$. We say that $G$ is \emph{$n$-connected} if $\mathrm{N}_1G$ is $n$-connected. By applying Theorem \ref{main 3} to $\mathrm{N}_1G$, the following is immediate.
	\begin{proposition}
		\label{Hurewicz theorem discrete homotopy group of digraph}
		Let $(G,v)$ be a pointed directed graph. If $(G,v)$ is $(n-1)$-connected, then the map
		\[
		h_*\colon\pi_n^\delta(G,v)\to\mathrm{QH}_n(G)
		\]
		is the abelianization for $n=1$ and is an isomorphism for $n\ge 2$.
	\end{proposition}
	
	\subsection{Discrete homotopy groups and discrete Hurewicz map of a pointed directed graph}
	In this subsection we provide a graph-theoretical description of the group $\pi_n^\delta(G,v)$ and identify it with the homotopy group of a directed graph in the sense of \cite{LWYZ}. We also prove Theorem \ref{main 4}. By definition, an $n$-grid of size $(l_1,\ldots,l_n)$ in $\mathrm{N}_1G$ is a cubical map
	\[
	x\colon I_{s_1}\otimes\cdots\otimes I_{s_n}\to\mathrm{N}_1(G),
	\]
	where
	\[
	s_i\colon\{1,\ldots,l_i\}\to\{0,1\}
	\]
	for $i=1,\ldots,n$. Therefore $x$ is nothing but a grid
	\[
	x\colon I_{s_1}\otimes\cdots\otimes I_{s_n}\to G
	\]
	in $G$. Next we describe the 0-homotopy and homotopy between grids in a directed graph $G$. Let $f\colon K\to G$ and $g\colon L\to G$ be maps from $n$-grid graphs $K$ and $L$ into a directed graph $G$. We write $f\xrightarrow{k}g$ if there exists $k\in\{1,\ldots,n\}$ such that $K=I_1\otimes\cdots\otimes I_{k-1}\otimes I\otimes I_{k+1}\otimes\cdots\otimes I_n$ and $L=I_1\otimes\cdots\otimes I_{k-1}\otimes J\otimes I_{k+1}\otimes\cdots\otimes I_n$ such that $J$ is obtained from $I$ by collapsing a single edge and $f=g\circ p$, where $p\colon K\to L$ denotes the projection. If $f=f_1\xrightarrow{k_1}\cdots\xrightarrow{k_{m-1}}f_m=g$ for some $f_1,\ldots,f_m$ and $k_1,\ldots,k_{m-1}$, then we write $f\ge g$. Clearly, this provides a poset structure on the set of $n$-grids.

	\begin{definition}
		Let $f$ and $g$ be $n$-grids in a directed graph $G$. We say that $f$ and $g$ are \emph{0-homotopic} if there exist $n$-grids $f_1,\ldots,f_m$ in $G$ such that $f_1=f$, $f_m=g$ and for each $i=1,\ldots,m-1$, either $f_i\le f_{i+1}$ or $f_{i+1}\le f_i$. In this case, we write $f\simeq_0g$.
	\end{definition}
	This notion generalizes the $C_0$-homotopy between paths in a directed graph, which was introduced in \cite{KT}. The following is immediate from the definitions.
	\begin{lemma}
		\label{two C_0-homotopies}
		Let $f$ and $g$ be $n$-grids in a directed graph $G$. Then they are 0-homotopic as $n$-grids in $G$ if and only if they are 0-homotopic as $n$-grids in $\mathrm{N}_1G$.
	\end{lemma}
	\begin{definition}
		\label{homotopy digraph def}
		We say that $n$-grids $f$ and $g$ are \emph{homotopic} if there exist $n$-grids $f_1,\ldots,f_m$ such that $f=f_1$, $g=f_m$, and for each $i=1,\ldots,m-1$, either $f_i\simeq_0f_{i+1}$ or $f_i\simeq_1 f_{i+1}$ holds. In this case, we write $f\simeq g$.
	\end{definition}
	This definition generalizes the $C$-homotopy between paths in a directed graph $G$, which was introduced in \cite{GLMY2}. For two $n$-grids $f$ and $g$ in a directed graph $G$, we have two notions of homotopy between $f$ and $g$, namely, as grids in $G$ (Definition \ref{homotopy digraph def}) and as grids in $\mathrm{N}_1G$ (Definition \ref{homotopy grids def}). However, we can show that these two notions are equivalent by Proposition \ref{rotate grid} and Lemma \ref{1-nerve symmetric}. 
	\begin{lemma}
		\label{two homotopies}
		Let $f$ and $g$ be $n$-grids in a directed graph $G$. Then they are homotopic as $n$-grids in $G$ if and only if they are homotopic as $n$-grids in $\mathrm{N}_1G$.
	\end{lemma}
	Let $(G,v)$ be a pointed directed graph and $f\colon I_{s_1}\otimes\cdots\otimes I_{s_n}\to G$. We say that $f$ is spherical if $f$ maps $\partial(I_{s_1}\otimes\cdots\otimes I_{s_n})$ to $v$, equivalently, $f\colon(I_{s_1}\otimes\cdots\otimes I_{s_n},\partial(I_{s_1}\otimes\cdots\otimes I_{s_n}))\to(G,v)$ is a map of pairs, and we may define homotopies between spherical $n$-grids in the obvious manner. Then for $n\ge 1$, $\pi^\delta_n(G,v)$ is the set of all homotopy classes of spherical $n$-grids in $(G,v)$. For path graphs $I\in\mathcal{I}_{l}$, $J\in\mathcal{I}_m$, let $I\cdot J\in\mathcal{I}_{l+m}$ be the path graph such that
	\[
	E(I\cdot J)=E(I)\sqcup\{(j+l,j'+l)\mid (j,j')\in E(J)\}.
	\]
	For spherical $n$-grids
	\[
	f\colon(I_{s_1}\otimes\cdots\otimes I_{s_n},\partial(I_{s_1}\otimes\cdots\otimes I_{s_n}))\to(G,v),
	\]
	\[
	g\colon(I_{s_1'}\otimes\cdots\otimes I_{s_n'},\partial(I_{s_1'}\otimes\cdots\otimes I_{s_n'}))\to(G,v)
	\] 
	in a pointed directed graph $(G,v)$, define the multiplication $f\cdot g$ to be the grid
	\[
	f\cdot g\colon\bigl((I_{s_1}\cdot I_{s_1'})\otimes\cdots\otimes(I_{s_n}\cdot I_{s_n'}),\partial((I_{s_1}\cdot I_{s_1'})\otimes\cdots\otimes(I_{s_n}\cdot I_{s_n'}))\bigr)\to(G,v)
	\]
	such that
	\begin{align*}
		(f\cdot g)_V(i_1,\ldots,i_n)&=
		\begin{cases}
			f_V(i_1,\ldots,i_n)&(i_1\le l_1,\ldots,i_n\le l_n)\\
			g_V(i_1-l_1,\ldots,i_n-l_n)&(i_1>l_1,\ldots,i_n>l_n)\\
			v&(\text{otherwise}),
		\end{cases}
	\end{align*}
	where $I_{s_k}\in\mathcal{I}_{l_k}$ $(k=1,\ldots,n)$. Then it is clear that $f\cdot g$ is spherical.

	Now we recall the homotopy group $\bar{\pi}_n(G,v)$ for a pointed directed graph $(G,v)$ in \cite{LWYZ} and show that it is isomorphic to the $n$-th discrete homotopy group of $(G,v)$. For $k\ge0$, let $L_k\in\mathcal{I}_k$ such that $(j,j+1)\in E(L_k)$ for $j$ even and $(j+1,j)\in E(L_k)$ for $j$ odd. The following two definitions are due to \cite{LWYZ}.
	\begin{definition}
		\label{LWYZ1}
		Let $(G,v)$ be a pointed directed graph and let $n\ge1$. Consider
		\[
		\underset{\rightarrow}{\mathrm{colim}}\,
		\mathsf{DiGraph}\left(
		\left(L_{k_1}\otimes\cdots\otimes L_{k_n},
		\partial(L_{k_1}\otimes\cdots\otimes L_{k_n})\right);
		(G,v)\right),
		\]
		where $k_1,\ldots,k_n\ge1$ and the transition maps are defined as follows.
		For $k=(k_1,\ldots,k_n)$ and $k'=(k_1',\ldots,k_n')$
		with $k_i\le k_i'$ for all $i$, let
		\[
		(e_{k,k'}f)_V(i_1,\ldots,i_n)=
		\begin{cases}
			f_V(i_1,\ldots,i_n)&(i_j\le k_j\text{ for all }j),\\
			v&(\text{otherwise}).
		\end{cases}
		\]
		Thus two spherical grids $f$ and $g$, indexed by $k$ and $k'$
		respectively, represent the same element of the colimit if and only if
		\[
		e_{k,l}f=e_{k',l}g
		\]
		for some $l=(l_1,\ldots,l_n)$ with
		$l_i\ge\max\{k_i,k_i'\}$ for all $i$.
	\end{definition}
	
	In \cite{LWYZ}, a grid $f'$ with $f'\ge f$ is called a
	\emph{subdivision} of $f$.
	
	\begin{definition}
		\label{LWYZ2}
		For arbitrary spherical $n$-grids in $(G,v)$, let $\simeq_F$
		be the equivalence relation generated by pairs $f,g$ admitting
		subdivisions $f'\ge f$ and $g'\ge g$ on the same grid graph such that
		\[
		f'\simeq_1g'
		\]
		relative to its boundary. Restricting a 1-homotopy to successive
		edges of its path graph gives the one-step homotopies used in
		\cite{LWYZ}.
		
		For $n\ge1$, define
		\[
		\bar{\pi}_n(G,v)=
		[\left(L,\partial L\right)^{\otimes n};(G,v)]
		\]
		to be the quotient of the colimit in Definition \ref{LWYZ1}
		by the relation induced by $\simeq_F$. We denote the class of $f$
		by $[f]_F$. The group operation is induced by $+_1$, after applying
		the extensions in Definition \ref{LWYZ1} so that the representatives
		have the same domain with even side lengths.
		For $n=0$, let $\bar{\pi}_0(G,v)$ be the set of connected components
		of $G$, pointed by the component containing $v$.
	\end{definition}
	
	\begin{proposition}
		\label{two homotopy groups}
		There is a natural isomorphism
		\[
		\pi_n^\delta(G,v)\cong\bar{\pi}_n(G,v)
		\]
		of groups for $n\ge1$ and of pointed sets for $n=0$.
	\end{proposition}
	
	\begin{proof}
		For $n=0$, both sides are the pointed set of connected components
		of $G$. Suppose that $n\ge1$. By Lemmas \ref{homotopy description new} and \ref{two homotopies}, $f\simeq g$ if and only if $f\simeq_Fg$ for spherical grids $f$ and $g$. Moreover, $e_{k,k'}f\ge f$ since there are
		\[
		r_j\colon L_{k_j'}\to L_{k_j},\quad
		i\mapsto\min\{i,k_j\}
		\]
		for $j=1,\ldots,n$ such that
		\[
		e_{k,k'}f=f(r_1\otimes\cdots\otimes r_n).
		\]
		Therefore the assignment
		\[
		\Theta\colon\bar{\pi}_n(G,v)\to\pi_n^\delta(G,v),
		\qquad [f]_F\mapsto[f]
		\]
		is well-defined and injective. To prove the surjectivity, let
		\[
		f\colon
		\bigl(I_{s_1}\otimes\cdots\otimes I_{s_n},
		\partial(I_{s_1}\otimes\cdots\otimes I_{s_n})\bigr)
		\to(G,v)
		\]
		be a spherical grid, where $I_{s_j}\in\mathcal{I}_{l_j}$.
		For $j=1,\ldots,n$, define
		\[
		r_j\colon L_{2l_j}\to I_{s_j}
		\]
		by
		\[
		r_j(2p)=p\quad(0\le p\le l_j),\qquad
		r_j(2p-1)=p-s_j(p)\quad(1\le p\le l_j).
		\]
		Thus
		\[
		f\le f(r_1\otimes\cdots\otimes r_n),
		\]
		implying that $\Theta$ is surjective. By Lemma \ref{grid slide} it is also clear that $\Theta$ is a homomorphism. The proof is completed.
	\end{proof}
	
	It remains to describe the discrete Hurewicz map. For a pointed directed graph $(G,v)$ and $n\ge0$, we define the Hurewicz map
	\[
	h_*\colon\pi_n^\delta(G,v)\to \mathrm{QH}_n(G)
	\]
	as follows. Let $f\colon I_{s_1}\otimes\cdots\otimes I_{s_n}\to G$ be a grid in $G$ such that $I_{s_k}\in\mathcal{I}_{l_k}$ $(k=1,\ldots,n)$, and let $i_k=1,\ldots,l_k$. Let
	\[
	f(i_1,\ldots,i_n)\colon\vec{I}^{\otimes n}\to G,\quad(p_1,\ldots,p_n)\mapsto f_V(i_1-1+p_1',\ldots,i_n-1+p_n')
	\]
	where
	\[
	p_k'=\begin{cases}
		p_k&((i_k-1,i_k)\in E(I_{s_k}))\\
		1-p_k&((i_k,i_k-1)\in E(I_{s_k})).
	\end{cases}
	\]
	Also let
	\[
	\mathrm{sgn}_{i_1,\ldots,i_n}(f)=(-1)^{s_1(i_1)+\cdots+s_n(i_n)},
	\]
	where
	\[
	s_k(i_k)=\begin{cases}
		0&((i_k-1,i_k)\in E(I_{s_k}))\\
		1&((i_k,i_k-1)\in E(I_{s_k})).
	\end{cases}
	\]
	Then if $f$ is a spherical $n$-grid representing a homotopy class in $\pi_n^\delta(G,v)$, there is a map
	\[
	h_*\colon\pi_n^\delta(G,v)\to \mathrm{QH}_n(G),\quad [f]\mapsto\left[\sum_{i_1,\ldots,i_n}\mathrm{sgn}_{i_1,\ldots,i_n}(f)f(i_1,\ldots,i_n)\right],
	\]
	which is a well-defined homomorphism by Theorem \ref{main 3}. A homomorphism 
	\[
	h_*\colon\bar{\pi}_n(G,v)\to \mathrm{QH}_n(G)
	\]
	was introduced in \cite{TWYZ} in order to construct directed graphs with nontrivial higher homotopy groups, and it coincides with the Hurewicz map up to the isomorphism in Proposition \ref{two homotopy groups}. We are ready to prove Theorem \ref{main 4}.
	\begin{proof}
		[Proof of Theorem \ref{main 4}]
		By Proposition \ref{two homotopy groups} and Lemma \ref{two homology groups}, there is a commutative diagram
		\[
		\xymatrix{
			\pi_n^\delta(G,v)\ar[r]^{h_*}\ar[d]_\cong&\mathrm{QH}_n(G)\ar@{=}[d]\\
			\bar{\pi}_n(G,v)\ar[r]_{h_*}&\mathrm{QH}_n(G),
		}
		\]
		and the proof is finished by applying Proposition \ref{Hurewicz theorem discrete homotopy group of digraph}.
	\end{proof}

	\subsection{1-nerves of directed and undirected graphs}
	\begin{definition}
		An \emph{undirected graph} $G$ consists of a vertex set $V(G)$ and an edge set $E(G)$ such that $E(G)\subset\{\{x,y\}\subset V(G)\mid x\ne y\}$, and an edge $\{x,y\}\in E(G)$ is regarded as an undirected edge between vertices $x$ and $y$.
	\end{definition}
	The maps between undirected graphs are defined similarly to those between directed ones. Thus undirected graphs and maps between them form a category, which we denote by $\mathsf{Graph}$. Similarly to the category $\mathsf{DiGraph}$, $\mathsf{Graph}$ is equipped with a monoidal product $\otimes$. Let $G$ be an undirected graph. In \cite{CK2}, Carranza and Kapulkin introduced the 1-nerve $\mathrm{N}_1G$ of an undirected graph $G$ and a fibrant replacement $\mathrm{N}G$ of $\mathrm{N}_1G$.
	The analogue of this construction for directed graphs is developed in \cite{EIXYZ}. Let
	\[
	\overline{(-)}\colon\mathsf{Graph}\to\mathsf{DiGraph}
	\]
	be the functor that replaces each edge $\{x,y\}\in E(G)$ with two directed edges $(x,y),(y,x)\in E(\overline{G})$.
	\begin{proposition}
		\label{nerve comparison}
		Let $G$ be an undirected graph. There is a natural isomorphism
		\[
		\mathrm{N}_1G\cong\mathrm{N}_{1}\overline{G}
		\]
		of cubical sets. As a consequence, $\mathrm{N}G$ and $\ex^\infty\mathcal{R}\mathrm{N}_1\overline{G}$ are homotopy equivalent.
	\end{proposition}
	We recall the 1-nerve of an undirected graph $G$. Let $I$ be the undirected graph such that
	\[
	V(I)=\{0,1\}\quad\text{and}\quad E(I)=\{\{0,1\}\}.
	\]
	Then
	\[
	\left(\mathrm{N}_1G\right)_n=\mathsf{Graph}(I^{\otimes n},G).
	\]
	\begin{proof}
		[Proof of Proposition \ref{nerve comparison}]
		For $n\ge 0$ and $f\colon I^{\otimes n}\to G$, let
		\[
		\phi(f)\colon\vec{I}^{\otimes n}\hookrightarrow \overline{I}^{\otimes n}\stackrel{\overline{f}}{\longrightarrow}\overline{G}
		\]
		be the map satisfying
		\[
		\phi(f)(v)=f(v)\text{ for }v\in V\left(I^{\otimes n}\right)=V\left(\vec{I}^{\otimes n}\right).
		\]
		Then we have a map
		\[
		\phi\colon\mathrm{N}_1G\to\mathrm{N}_1\overline{G}
		\]
		and it is easy to show that it is cubical. Conversely, for $g\colon\vec{I}^{\otimes n}\to \overline{G}$, let
		\[
		\psi(g)\colon I^{\otimes n}\to G
		\]
		be the map satisfying
		\[
		\psi(g)(v)=g(v)\text{ for }v\in V\left(I^{\otimes n}\right)=V\left(\vec{I}^{\otimes n}\right).
		\]
		Then clearly $\phi$ and $\psi$ are cubical maps and they are mutually inverse to each other, completing the proof.
	\end{proof}
	For a pointed undirected graph $(G,v)$, the $n$-th $A$-homotopy group $A_n(G,v)$ is a classical invariant, which plays the role of homotopy groups in the discrete homotopy theory of undirected graphs. On the other hand, the \emph{discrete homology groups} of an undirected graph $G$ are defined to be the homology groups $\mathrm{DH}_*(G)=H_*(\mathrm{N}_1G)$ of the 1-nerve of $G$. Carranza and Kapulkin \cite{CK2} constructed the discrete Hurewicz homomorphism
	\[
	h_*\colon A_n(G,v)\to \mathrm{DH}_n(G)
	\]
	and proved the Hurewicz theorem for a pointed undirected graph in \cite{CK2}. We give an alternative proof by applying Theorem \ref{main 4} and Proposition \ref{nerve comparison}.
	\begin{proposition}
		\label{Hurewicz theorem graph}
		Let $(G,v)$ be a pointed connected undirected graph and $n\ge1$. If $A_i(G,v)=0$ for $i=1,\ldots,n-1$, then the map
		\[
		h_*\colon A_n(G,v)\to \mathrm{DH}_n(G)
		\]
		is the abelianization for $n=1$ and is an isomorphism for $n\ge2$.
	\end{proposition}
	\begin{proof}
		By Theorem \ref{main 4}, the statement above holds for the homomorphism
		\[
		h_*\colon\pi_n^\delta(\overline{G},v)\to \mathrm{QH}_n(\overline{G}).
		\]
		The comparison of nerves identifies the graph-theoretic $A$-groups with the discrete homotopy groups of the directed replacement $\overline{G}$ as follows:
		\[
		A_n(G,v)\cong\pi_n(\mathrm{N}G,v)\cong\pi_n(\ex^\infty\mathcal{R}\mathrm{N}_1\overline{G},v)\cong\pi_n^\delta(\mathrm{N}_1\overline{G},v)=\pi_n^\delta(\overline{G},v),
		\]
		where the first isomorphism is proved in \cite{CK2}. There is also an isomorphism
		\[
		\mathrm{DH}_n(G)\cong \mathrm{QH}_n(\overline{G})
		\]
		by Proposition \ref{nerve comparison}. The proof is finished.
	\end{proof}

\end{document}